\documentclass{amsart}
\usepackage{amssymb}
\usepackage[all,cmtip]{xy}
\usepackage{mathrsfs}
\usepackage{upgreek}
\usepackage{stmaryrd}
\usepackage{rotating}
\usepackage{tikz}
\usetikzlibrary{cd}
\usepackage[style=trad-alpha]{biblatex}
\usepackage[T1]{fontenc}

\usepackage{color}%new color settings
\usepackage[hyphens]{xurl}
\usepackage{breakurl}
\usepackage[hyperindex,breaklinks]{hyperref}
\hypersetup{}
\hypersetup{
    breaklinks=true,
    bookmarksdepth=2,
    colorlinks,
    linkcolor={red!50!black},
    citecolor={blue!50!black},
    urlcolor={blue!50!black},
    filecolor={red!50!black}, % black
    pdfusetitle=true
}
\usepackage[compress,capitalize]{cleveref}

\newcommand{\Ga}{\mathbb{G}_{\mathrm{a}}}
\newcommand{\Dbc}{D^{\mathrm{b}}_{\mathrm{c}}}
\newcommand{\Dmix}{D^{\mathrm{mix}}}

\newcommand{\AS}{\mathrm{AS}}

\let\x\times

\newcommand{\sH}{\mathsf{H}}

\newcommand{\Z}{\mathbb{Z}}
\newcommand{\R}{\mathbb{R}}

\newcommand{\N}{\mathbb{N}}
\newcommand{\Q}{\mathbb{Q}}

\newcommand{\Gm}{\mathbb{G}_{\mathrm{m}}}

\newcommand{\ff}{\mathfrak{f}}
\newcommand{\fg}{\mathfrak{g}}

\newcommand{\coweyl}{\mathsf{N}}

\newcommand{\Waff}{W_{\mathrm{aff}}}
\newcommand{\Wfin}{W_{\mathrm{fin}}}

\newcommand{\cO}{\mathcal{O}}

\newcommand{\cB}{\mathcal{B}}

\newcommand{\IC}{\mathscr{I} \hspace{-1pt} \mathscr{C}}

\newcommand{\Gr}{\mathrm{Gr}}
\newcommand{\iCat}{\textup{Cat}_{\infty}}
\newcommand{\cC}{\mathcal{C}}

\newcommand{\mix}{{\mathrm{mix}}}
\newcommand{\Perv}{\mathsf{Perv}}

\newcommand{\Rep}{\mathsf{Rep}}

\newcommand{\Vect}{\mathsf{Vect}}
\newcommand{\Coh}{\mathsf{Coh}}

\newcommand{\DM}{\mathsf{DM}}
\newcommand{\Db}{D^{\mathrm{b}}}
\newcommand{\Kb}{K^{\mathrm{b}}}
\DeclareMathOperator{\Hom}{Hom}
\DeclareMathOperator{\Aut}{Aut}
\DeclareMathOperator{\Ext}{Ext}
\DeclareMathOperator{\End}{End}
\DeclareMathOperator{\Sym}{Sym}
\DeclareMathOperator{\Ind}{Ind}

\DeclareMathOperator*{\colim}{colim}

\DeclareMathOperator{\im}{im}

\newcommand{\id}{\mathrm{id}}

\newcommand{\res}{\mathrm{Res}}

\newcommand{\bn}{\mathbf{n}}

\makeatletter
\def\lotimes{\@ifnextchar_{\@lotimessub}{\@lotimesnosub}}
\def\@lotimessub_#1{\mathchoice{\mathbin{\mathop{\otimes}^L}_{#1}}%
  {\otimes^L_{#1}}{\otimes^L_{#1}}{\otimes^L_{#1}}}
\def\@lotimesnosub{\mathbin{\mathop{\otimes}^L}}
\makeatother

\makeatletter
\def\lboxtimes{\@ifnextchar_{\@lboxtimessub}{\@lboxtimesnosub}}
\def\@lboxtimessub_#1{\mathchoice{\mathbin{\mathop{\boxtimes}^L}_{#1}}%
  {\boxtimes^L_{#1}}{\boxtimes^L_{#1}}{\boxtimes^L_{#1}}}
\def\@lboxtimesnosub{\mathbin{\mathop{\boxtimes}^L}}
\makeatother

\newcommand{\tboxtimes}{\mathbin{\widetilde{\mathord{\boxtimes}}}}
\newcommand{\wttimes}{\mathbin{\widetilde{\times}}}

\newcommand{\Spec}{\mathrm{Spec}}

\mathchardef\mhyphen="2D

\newcommand{\Lie}{\operatorname{Lie}}

\newcommand{\cA}{\mathcal{A}}

\newcommand{\inv}{^{-1}}

\newcommand{\ad}{\mathrm{ad}}

\newcommand{\hk}{\ensuremath{\operatorname{Hk}_{G}} }
\newcommand{\Hk}[1]{\ensuremath{\operatorname{Hk}_{#1}}}
\newcommand{\chamb}{X_*(T)_I^+}
\newcommand{\splitchamb}{X_*(T)^+}
\newcommand{\cham}{{X_*}(T)_I}
\newcommand{\splitcham}{X_*(T)}
\newcommand{\posroot}{\Phi^+}
\newcommand{\starsup}{*\text{-}\operatorname{Supp}}
\newcommand{\shrieksup}{!\text{-}\operatorname{Supp}}
\newcommand{\Zent}{\operatorname{Z}}
\newcommand{\grW}{\operatorname{Gr}^W}
\newcommand{\isomto}{\stackrel{\textstyle\sim}{\smash{\longrightarrow}\rule{0pt}{0.4ex}}}
\newcommand{\Qell}{\ensuremath{\overline{\mathbb{Q}}_\ell}}
\newcommand{\stdmix}{\Delta^\text{mix}}
\newcommand{\costdmix}{\nabla^\text{mix}}
\newcommand{\ICmix}{\IC^\text{mix}}
\newcommand{\Pervmix}{\mathsf{Perv}^\text{mix}}

\newcommand{\Jmix}{J^\text{mix}}
\renewcommand{\a}{\mathbf{a}}
\newcommand{\f}{\mathbf{f}}
\newcommand{\Par}{\mathcal{G}}
\newcommand{\red}{\text{red}}
\newcommand{\Lplus}{\mathrm{L}^+}
\newcommand{\LG}{\mathrm{L}G}
\renewcommand{\L}{\mathrm{L}}

\newcommand{\I}{\mathcal{I}}
\newcommand{\G}{\mathcal{G}}

\newcommand{\Iop}{\mathcal{I}^\mathrm{op}}
\newcommand{\Las}{\mathcal{L}_\mathrm{AS}}
\newcommand{\defined}{\hspace{0.1cm}\stackrel{\text{\tiny \rm def}}{=}\hspace{0.1cm}}
\newcommand{\PIW}{\mathsf{P}_{\mathcal{IW}}}
\newcommand{\PervI}{\mathsf{Perv}(\Hk{\I})}
\newcommand{\IW}{\mathcal{IW}}
\newcommand{\Po}{\mathsf{P}^0}
\newcommand{\Pasph}{\mathsf{P}_{\mathsf{asph}}}
\newcommand{\stdiw}{\Delta^\mathcal{IW}}
\newcommand{\costdiw}{\nabla^\mathcal{IW}}
\newcommand{\std}{\Delta}
\newcommand{\costd}{\nabla}
\newcommand{\aviw}{\mathsf{av}^\mathcal{IW}}
\newcommand{\avasph}{\mathsf{av}^\mathsf{asph}_\mathcal{IW}}
\newcommand{\ICIW}{\IC^{\mathcal{IW}}}
\newcommand{\DIW}{D_{\mathcal{IW}}(\Fl{\I})}
\newcommand{\DIWX}{D_{\mathcal{IW}}}
\newcommand{\gr}{\operatorname{gr}}
\newcommand{\Grad}{\operatorname{Grad}}
\newcommand{\isoto}{\overset{\sim}{\,\to\,}}
\DeclareMathOperator{\opname}{op}
\newcommand{\op}{^{\opname}}
\newcommand{\Res}{\mathrm{Res}}
\newcommand{\ogm}{\Z[\mathbf{v},\mathbf{v}\inv]}

\newcommand{\Sprgi}{\widetilde{\mathcal{N}}_{\check G^I}}
\newcommand{\Sprgd}{\widetilde{\mathcal{N}}_{\check{G}^{I,\circ}}}

\newcommand{\Sprgafi}{\widetilde{\mathcal{N}}^\mathrm{af}_{{\check G}^I}}
\newcommand{\Sprgqafi}{\widetilde{\mathcal{N}}^\mathrm{qaf}_{{\check G}^I}}
\newcommand{\Sprgqafd}{\widetilde{\mathcal{N}}^\mathrm{qaf}_{\check{G}^{I,\circ}}}
\newcommand{\Gdual}{\check{G}^{I,\circ}}
\newcommand{\Bdual}{\check{B}^{I,\circ}}

\newcommand{\Udual}{\check{U}^{I}}

\usepackage{etoolbox}
\newcommand{\Fl}[2][]{%
  \ifblank{#1}{
    \operatorname{Fl}_{#2}}
  {
    \operatorname{Fl}_{#1,#2}}}

\newcommand{\GrBD}[2][]{%
  \ifblank{#1}{
	  \operatorname{Gr}^\mathrm{BD}_{#2}}
  {
    \operatorname{Gr}^\mathrm{BD}_{#1,#2}}}

\numberwithin{equation}{section}
\newtheorem{thrm}{Theorem}[section]
\newtheorem{lem}[thrm]{Lemma}
\newtheorem{prop}[thrm]{Proposition}
\newtheorem{cor}[thrm]{Corollary}

\newtheorem{ques}[thrm]{Question}
\theoremstyle{definition}
\newtheorem{defn}[thrm]{Definition}

\theoremstyle{remark}
\newtheorem{rmk}[thrm]{Remark}

\title[Twisted affine flag varieties and Langlands duality]{Perverse sheaves on twisted affine flag varieties and Langlands duality}
\author{R{\i}zacan \c{C}ilo\u{g}lu}
\address{Technische Universit\"at Darmstadt, Department of Mathematics, 64289 Darmstadt, Germany}
\email{rizacan.ciloglu@mathematik.tu-darmstadt.de}
\thanks{R{\i}zacan \c{C}ilo\u{g}lu acknowledges support (through Timo Richarz) by the European
Research Council (ERC) under Horizon Europe (grant agreement nº 101040935), by the Deutsche
Forschungsgemeinschaft (DFG, German Research Foundation) TRR 326 \textit{Geometry and Arithmetic
of Uniformized Structures}, project number 444845124 and the LOEWE professorship in Algebra,
project number LOEWE/4b//519/05/01.002(0004)/87.}
\begin{document}
\begin{abstract}
	We provide a description of Iwahori-Whittaker equivariant perverse sheaves on affine
	flag varieties associated to tamely ramified reductive groups, in terms of Langlands dual data. This
	extends the work of Arkhipov-Bezrukavnikov \cite{arkhipovBezrukavnikov} from the case of
	split reductive groups.
	To achieve this, we first extend the theory of Wakimoto
	sheaves to our context and prove convolution exact central objects admit a filtration by
	such. We then establish the tilting property of the Iwahori-Whittaker averaging of certain central objects arising from the geometric Satake equivalence, which enables us
        to address the absence of an appropriate analogue of Gaitsgory's central functor \cite{gaitsgoryCentral} for non-split groups.
\end{abstract}
\maketitle
\setcounter{tocdepth}{1}
\tableofcontents
\section{Introduction}
The geometric Satake equivalence \cite{mirkovicvilonenSatake} is fundamental to the geometric aspects of the Langlands program and various approaches to
geometric representation theory. It is a monoidal
equivalence of categories that relates perverse sheaves on the affine Grassmannian of a split
reductive group $G$ to the representations
of its Langlands dual group $\check G$. In recent years, this result has been extended to the case of non-split reductive groups in various settings, initiated by
\cite{zhuRamified} and followed by \cite{richarzRamified}, \cite{modularRamified}, and
\cite{thibaudSatake}. In these extensions, the group appearing on
the dual side is the group of fixed points $\check G^I$, for a certain natural action of
a group $I$, which acts trivially when $G$ is split. For example, if $G$ is the odd dimensional
projective unitary
group $\mathrm{PU}_{2n+1}$, then $\check G=\mathrm{SL}_{2n+1}$ and $\check G^I=\mathrm{SO}_{2n+1}$.

In \cite{arkhipovBezrukavnikov}, the geometric Satake equivalence
for a split group $G$ is extended to a description of the perverse sheaves on the (full) affine flag
variety in terms of the equivariant coherent sheaves on the Springer resolution associated to $\check G$. This
description has come to be known as the \textit{Arkhipov-Bezrukavnikov equivalence}.
The purpose of the work at hand is to extend the Arkhipov-Bezrukavnikov equivalence
to the case of non-split reductive groups.

Alongside direct applications to representations of quantum groups
\cite{bezrukavnikovQuantum},
semisimple Lie algebras \cite{bezrukavnikovSemisimpleLie}, and the
geometric Langlands program \cite{frenkelGaitsgory}; the Arkhipov-Bezrukavnikov equivalence also laid out the technical foundations for the later
generalization \cite{bezrukavnikovEquivalence}. This generalization has also found a remarkably
wide set of applications in \cite{nadlerCategoricalDeligneLanglands}, \cite{yunNadlerBetti},
\cite{hellmannPadicEigenforms}, \cite{yunCommutingStack}, to name a few. Accordingly, the work
at hand is a first step in exploring
such variety of questions for non-split reductive groups.
\subsection{Main result}
Let $k$ be the algebraic closure of a finite field of characteristic $p>0$, and
$G$ be a tamely ramified reductive group over the field of Laurent series $k(\!(t)\!)$. We consider an
\textit{Iwahori} model $\I$ of $G$ over the ring of power series $k[\![t]\!]$, in the sense of \cite{bruhatTits}.
Using the group functors
$$\LG:R\mapsto G\left(R(\!(t)\!)\right), \quad \Lplus\I:
R\mapsto\I\left(R[\![t]\!]\right),$$
on the category of $k$-algebras, we obtain the \textit{(twisted) affine flag variety} $\Fl{\I}$ as the
\'etale-quotient $\LG/\Lplus\I$ of $\LG$ by its subgroup functor $\Lplus\I$.

On the constructible side of the equivalence is the full triangulated subcategory
$$\DIW\subset \Dbc(\Fl{\I},\Qell)$$
of constructible $\ell$-adic \'etale sheaves on $\Fl{\I}$, which satisfy an equivariance condition with
respect to a Whittaker datum, made precise in \cref{section:iwahoriwhittaker}. Moreover, the
inclusion of this full subcategory commutes with the perverse truncation functors,
therefore $\DIW$ inherits the perverse t-structure.

For the coherent side of the equivalence, note that, although the group $\check G^I$ is potentially disconnected,
its identity component $\Gdual$ is a split reductive group \cite[Proposition 16]{hainesSatake}.
As such, we can associate to it the \textit{Springer resolution} $\Sprgd$, which is a resolution
of singularities of the variety of nilpotent elements in the Lie algebra of $\check G^I$,
see \cref{subsection:springer}.
The adjoint action of $\check G^I$ naturally extends to an action on $\Sprgd$, giving rise to
the category $\Coh^{\check G^I}(\Sprgd)$ of $\check G^I$-equivariant coherent sheaves on
$\Sprgd$. We now formulate our main theorem.
\begin{thrm}\label{intro:mainthrm}
	There is an equivalence of categories
	$$F_\IW:\Db\Coh^{\check G^I}(\Sprgd)\isomto \DIW,$$
	where $\Db\Coh^{\check G^I}(\Sprgd)$ denotes the bounded derived category of $\Coh^{\check
	G^I}(\Sprgd)$.
\end{thrm}
The proof of \cref{intro:mainthrm} is uniform for tamely ramified reductive groups, treating
split and non-split groups in equal footing.

Another thing to note is that the equivalence $F_\IW$ is not t-exact for the respective standard and the perverse
t-structures. However, transporting the perverse t-structure to $\Db\Coh^{\check G^I}(\Sprgd)$ along $F_\IW$ leads to
the so called \textit{exotic t-structure} which is of independent interest, and has found applications in the
aforementioned work \cite{bezrukavnikovQuantum} in the split setting.

The proof of \cref{intro:mainthrm} can be broken into two steps:
\begin{itemize}
	\item[Step 1:]Construct a
		functor
$$\Rep(\check G^I)\rightarrow \PIW.$$
This requires many ingredients and takes place predominantly in the constructible side of the
equivalence.
\item[Step 2:]
		Upgrade this functor to the desired equivalence by gradually extending the domain category
		to $\Coh^{\check G^I}(\Sprgd)$, using certain extra structure which arises
		naturally. This takes place predominantly in the coherent side.
\end{itemize}
Let us now discuss these in more detail.
\subsection{Strategy of the proof}\label{intro:central}
We will focus on explaining the first step, as the second step turns out to be a rather
straightforward exercise in carrying out the strategy from the split case.

In contrast, the first step requires introduction of new ideas. For a split group $G$, the
$I$-action on $\check G$ is trivial and the desired functor is obtained in a straightforward
manner from the \textit{central functor}
$$\Zent:\Rep(\check G)\rightarrow \Perv(\Hk{\I})$$
of Gaitsgory \cite{gaitsgoryCentral}, where $\Perv(\Hk{\I})$ is the full subcategory of
$\Lplus\I$-equivariant sheaves in
$\Perv(\Fl{\I})$ with respect to the left action. It is constructed via the geometric Satake equivalence,
from a nearby cycles construction for
\textit{Beilinson-Drinfeld Grassmannians}, which are degenerations of affine Grassmannians to
affine flag varieties, see \cref{subsection:BDGrass}. In the split setting of
\cite{arkhipovBezrukavnikov}, the functor $\Zent$ is
the fundamental building block for relating the constructible side to the coherent side. After the
establishment of this relation, the equivalence is proven via explicit cohomological calculations on the
respective sides.

This introduces the essential difficulty of the non-split
case: There is no known construction of a deformation, which degenerates twisted affine
Grassmannians\footnote{By which we mean the twisted flag
variety $\Fl{\G}$, associated to a special parahoric model $\G$ of a non-split group $G$.} to
twisted affine flag varieties;
 consequently there is no apparent construction of a functor from $\Rep(\check G^I)$ to $\Perv(\Hk{\I})$ analogous to the central functor.
 It is then natural to ask:
\begin{ques}\label{intro:question}
	Does there exist a functor from $\Rep(\check G^I)$ to
	$\Perv(\Hk{\I})$ such that the diagram
	$$\begin{tikzcd}
		\Rep(\check G)\ar[r,"\Zent"]\ar[d,"\Res"'] & \Perv(\Hk{\I}) \\
		\Rep(\check G^I)\ar[ur,dotted]&
	\end{tikzcd}$$
	commutes?
\end{ques}
\cref{intro:question} is essentially equivalent to asking that, for a $\check G$-representation $V$,
the sheaf $\Zent(V)$ admits a direct sum decomposition determined by the splitting of $\Res(V)$ to
irreducible $\check{G}^I$-representations. Since $\Zent(V)$ is obtained through a nearby cycles
construction, this is known to be a difficult question to answer in general.
In fact, an analogous question arises in \cite{zhuRamified} for a special parahoric model $\G$
instead of an Iwahori. In this case the category $\Perv(\Hk{\G})$ is semisimple, and
admits a Tannakian structure. Using this, the analogous question is then answered in the positive using categorical
methods, by appeal to Tannaka duality. Unfortunately, this fails in the case of an Iwahori, as
$\Perv(\Hk{\I})$ is far from semisimple, and does not carry an apparent monoidal structure let
alone a Tannakian one. To the author, a positive answer to \cref{intro:question} is currently only known in the case of tori and
$\mathrm{SU}_{3}$, obtained through extensions of such categorical methods.

Lacking the geometric means to answer this question, we opt instead to circumvent it
using the theory of highest weight categories and tilting objects.
First, let us expand a bit more on the equivariance condition defining $\DIW$. An
\textit{Iwahori-Whittaker datum} amounts to a tuple $(\chi,L_\AS)$ consisting of a generic linear
functional $\chi:\Lplus\Iop_u\rightarrow \Ga$
of a certain pro-unipotent group $\Lplus\Iop_u$; and an Artin-Schreier local system  $L_\AS$ on
$\Ga$. The group $\Lplus\Iop_u$ acts on $\Fl{\I}$, and the category $\DIW$ is the full
subcategory of sheaves $F\in \Dbc(\Fl{\I},\Qell)$, satisfying
$$a^*F\cong F\boxtimes\chi^*L_\AS$$
for the action map $a:\Lplus\Iop_u\x\Fl{\I}\rightarrow \Fl{\I}$. Then, one obtains a functor
$$\aviw:\Dbc(\Hk{\I})\rightarrow \DIW,$$
by averaging along the action of $\Lplus\Iop_u$, see \cref{subsection:averaging}.
As mentioned before, $\DIW$ inherits the perverse t-structure for which $\aviw$ is shown to be
perverse t-exact by \cref{cor:aviewtexact}. Let $\PIW$ be the heart of the perverse t-structure on
$\DIW$. \cref{section:iwahoriwhittaker} and \cref{section:tilting}
are devoted to the study of $\PIW$, culminating in the proof of the following:
\begin{thrm}\label{intro:tilting}
		 The category $\PIW$ carries the structure of a \textit{highest weight category}, in the sense of
	\cite[\textsection 1.12.3]{baumannRiche}. Moreover,
		for every representation $V\in \Rep(\check G^I)$, there exist a representation $V'\in \Rep(\check G)$ such that
	\begin{enumerate}
		\item the representation $V$ is a direct summand in $\Res(V')$.
		\item the object $\aviw\circ\Zent(V')\in\PIW$ is \textit{tilting} with respect to
			this highest weight structure;
	\end{enumerate}
\end{thrm}
Although \cref{intro:tilting} is of independent interest, it is also directly relevant to the
construction of the functor $\Rep(\check G^I)\rightarrow \PIW$. Namely, we prove in
\cref{section:tilting} that the \textit{regular
quotient}, $\PIW^0$, which is a certain Serre quotient
of $\PIW$ with a large kernel, admits the structure of a Tannakian category. This allows us to
construct a functor
$$\Rep(\check G^I)\rightarrow \PIW^0$$
by appealing to Tannaka duality. It turns out that the natural projection functor $\PIW\rightarrow
\PIW^0$ is fully faithful when restricted to the tilting objects; allowing us to use
\cref{intro:tilting} to prove the existence of a commutative diagram:
$$\begin{tikzcd}
	\Rep(\check G)\ar[d,"\res"']\ar[r,"\operatorname{Z}"] &
	\PervI \ar[d,"\aviw"]\\
	\Rep(\check G^I)\ar[r]& \PIW.
\end{tikzcd}$$
The functor $$\Rep(\check G^I)\rightarrow\PIW$$
from this diagram is the desired one.

Finally, let us mention that the tame ramification assumption is used only once in the entire
proof.
Namely, the proof of \cref{intro:tilting} makes use of a bound on the
dimension of $\Hom_{\PIW}(\aviw\circ\Zent(\Qell),\aviw\circ\Zent(V))$, where $V$ is a is
irreducible $\check G$-representation whose highest weight restricts to a quasi-minuscule
coweight for the \'echelonnage root system of $G$. The proof of this bound uses the interpretation of
the monodromy of nearby cycles, in terms of Verdier's construction of such for $\Gm$-monodromic
sheaves, using the loop rotation $\Gm$-action. Giving an independent proof of this bound would
extend our result to all non-split groups in a uniform manner.
\subsection{Other results}
The simple objects $\IC_w\in\Perv(\Hk{\I})$ are naturally indexed by elements $w\in
W$ of the \textit{Iwahori-Weyl
group}. The Iwahori-Weyl group  can be equipped with a length function
$\ell:W\rightarrow\Z_{\geq 0}$, and
admits an embedding from the finite Weyl group $\Wfin$ associated to $G$.

Let $\Pasph$ be the quotient of $\Perv(\Hk{\I})$ by the smallest Serre
subcategory containing those $\IC_w$, for which $w$ is \textit{not} of minimal length in
$\Wfin\cdot w$, its left finite Weyl group orbit. This category is a geometrization of the \textit{antispherical module}
of the Iwahori-Hecke algebra associated to $(G,\I)$.
Using \cref{intro:tilting} we deduce an equivalence between two geometrizations of this
antispherical module,
generalizing \cite[Theorem 2]{arkhipovBezrukavnikov} to include non-split groups:
\begin{thrm}\label{intro:key}
The functor $\aviw:\Perv(\Hk{\I})\rightarrow \PIW$
	induces
	$$\avasph:\Pasph \rightarrow\PIW$$
	via the universal property of Serre quotients, which is an equivalence of categories.
\end{thrm}
\subsection{Contents}
We begin with the study of the constructible side of the equivalence.
In \cref{sec:loopgroups}, we recall the geometry of affine flag varieties and establish our
conventions pertaining to the \'etale sheaves on such. Building upon this, in \cref{sec:wakimoto} we reproduce the
theory of Wakimoto sheaves in our context, culminating in the
construction of a functor
$$\mathbb{J}:\Rep(\check T^I)\rightarrow \Dbc(\Fl{\I},\Qell)$$
where $T$ is a maximal subtorus of $G$.

In \cref{decategorification}, we discuss the extension of results of previous sections to the
setting of mixed sheaves; and recall the connection of our categories of sheaves to
Iwahori-Hecke algebras, by the means of passing to the Grothendieck groups. We then
 put this connection to use during a few technical lemmas in \cref{section:iwahoriwhittaker} and
 \cref{section:tilting}; where we construct the aforementioned highest weight structure on $\DIW$ and
 prove the tilting property from \cref{intro:tilting}.

In \cref{section:coherentfunc}, we move onto the coherent side of the equivalence.
In particular, we start the construction of the functor $F_\IW$ of \cref{intro:mainthrm}.
Namely, we verify the Drinfeld-Pl\"ucker relations for the functor
$$Z\x\mathbb{J}:\Rep(\check G\x\check T^I)\rightarrow \Perv(\Hk{\I});$$
as in \cite[Section 3]{arkhipovBezrukavnikov},
this allows us to construct a functor
$$\Coh^{\check G\x\check T^I}(\Sprgafi)\rightarrow \Perv(\Hk{\I});$$
where $\Sprgafi$ is the affine completion of a certain natural $\check T^I$-torsor over the Springer resolution. Using the
tilting property proven in \cref{section:tilting}; we show that the postcomposition of this functor
with $\aviw$, factors through a functor
$$\Coh^{\check G^I\x\check T^I}(\Sprgafi)\rightarrow \PIW.$$
Finally, we relate $\Db\Coh^{\check G^I\x\check T^I}(\Sprgafi)$ to $\Db\Coh^{\check
G^I}(\Sprgi)$ through a Verdier quotient, using which we obtain the desired functor
$$F_\IW:\Db\Coh^{\check G^I}(\Sprgi)\rightarrow \DIW$$
via the universal property. In the short \cref{section:proof}, we verify that $F_\IW$ is an
equivalence, by comparing carefully chosen sets of generators for each side.
\subsection*{Acknowledgements}
First and foremost, I thank my advisor Timo Richarz, for introducing me to this problem and
painstakingly teaching me how to write mathematics. Special thanks are due to Simon
Riche and Jo{\~a}o Louren{\c c}o, conversations with whom led me to crucial insights about the
present article. I also thank Jo{\~a}o Louren{\c c}o once more, for spotting a mistake in a
prior version. Finally, I
thank Konstantin Jakob, Patrick Bieker, Thibaud van den Hove, and Can Yaylali for many
conversations surrounding the topic.
\section{Twisted affine flag varieties}\label{sec:loopgroups}
Let $k$ be an algebraically closed field, $K=k(\!(t)\!)$ the field of formal Laurent series, and
$O_K=k[\![t]\!]$ the ring of power series. Since $k$ is algebraically closed,
non-split reductive groups over $K$ are precisely the ramified ones.
\subsection{Loop groups and affine flag varieties}\label{subsec:loopgroups}
Let $G$ a reductive group over $K$, and $\mathcal{G}$ a parahoric model of $G$ over $O_K$, in the sense of \cite{bruhatTits}.
Consider the functors
$$\LG:R\mapsto G\left(R(\!(t)\!)\right), \quad \Lplus\mathcal{G}:
R\mapsto\mathcal{G}\left(R[\![t]\!]\right)),$$
from $k$-algebras to groups, called \textit{loop group} and \textit{positive loop group} respectively. The loop group is
representable by a strict ind-affine ind-group scheme, and the positive loop group is
representable by an affine scheme \cite[Proposition 7.1]{twistedLoop}. The main geometric object
we consider is the \textit{affine flag variety}, defined as the \'etale-sheafification of the
following
functor on $k$-algebras:
$$\Fl{\mathcal{G}}:R\mapsto \LG(R)/\Lplus\mathcal{G}(R).$$
Since $\mathcal{G}$ is parahoric, the sheaf $\Fl{\mathcal{G}}$ is representable by an ind-proper strict
ind-scheme \cite[Theorem A]{richarzRamified}.

There is an $\Lplus\mathcal{G}$-action on $\Fl{\mathcal{G}}$, induced by the multiplication from left.
The \textit{Hecke stack associated to $\mathcal{G}$} is defined to be the quotient stack
$$\Hk{\mathcal{G}}\defined \Lplus\mathcal{G}\backslash\Fl{\mathcal{G}}
,$$
in the \'etale topology. For the rest of the document, we will only consider such quotients.
By \cite[Lemma A.3.5]{richarzModuliShtuka}, $\Fl{\mathcal{G}}$ admits a presentation, $\colim_i
X_i$, as an
ind-proper ind-scheme by orbit closures for the $\Lplus{\mathcal{G}}$-action. Moreover, on each
orbit closure, $\Lplus\mathcal{G}$ acts through a finite dimensional quotient
$\Lplus_n\mathcal{G}$ for some $n\in\N$; which is defined as a functor on $k$-algebras by
$$\Lplus_n\mathcal{G}:R\mapsto \mathcal{G}\left(R[t]/(t^{n+1})\right).$$
In particular, $\Lplus\mathcal{G}\cong\lim_n\Lplus_n\mathcal{G}$ which induces
$$\Lplus\mathcal{G}\backslash X_i\cong \lim_n\Lplus_n\mathcal{G}\backslash
X_i.$$
As the quotient of a finite type scheme by a
finite type affine group scheme is an Artin stack of finite type, we get
$$\Hk{\mathcal{G}}\cong\colim_i\Lplus\mathcal{G}\backslash X_i\cong
\colim_i\lim_n \Lplus_n\mathcal{G}\backslash
X_i$$
where each $\Lplus_n\mathcal{G}\backslash
X_i$ is an Artin stack of finite type.

\subsection{Iwahori-Weyl group and Bruhat order}\label{iwahoriweylgroup}
We now specialize to the case in which the parahoric group scheme is an \textit{Iwahori}.
In this case, the geometry of the corresponding Hecke stack is closely related to the
\textit{Iwahori-Weyl group}. We recall the structure of this group following \cite{twistedLoop} and its appendix.

By a result of Steinberg in \cite{steinbergQuasiSplit}, every reductive group $G$ over $K$ is quasi-split.
Thus, we may fix a pair $(B,S)$ where $B$ is a Borel subgroup of $G$ and $S$ is a maximal
$K$-split torus contained in $B$. Let $\a$ be an alcove in the apartment
$\mathcal{A}(G,S,K)$ of the (extended) Bruhat-Tits building of $G$. The choice of such an alcove
yields the Iwahori group scheme $\mathcal{I}$,
whose associated flag variety will be denoted as $\Fl{\I}.$

The inertia subgroup $I$ of $\operatorname{Gal}(\bar{K}/K)$ acts on $\pi_1(G)$, and there is a
natural
assignment of a surjective map
$$\kappa_G: G(K)\rightarrow \pi_1(G)_I$$ to the coinvariants, called the Kottwitz
homomorphism \cite[\textsection 7]{kottwitzIsocrystal2}.

As $G$ is quasi-split, the centralizer $T:=Z_G(S)$ is a maximal torus
\cite[Proposition 20.6]{borelLinearAlgebraic}.
The Iwahori-Weyl group associated to $S$ is defined to be
$$W:=N_G(T)(K)/\left(T(K)\cap\ker(\kappa_T)\right).$$
Since $T$ is a torus, there is an isomorphism $\cham\isomto T(K)/\left(T(K)\cap\ker(\kappa_T)\right)$ via the assignment $\mu\mapsto
t^\mu$, and thus
exhibits an exact sequence:
\begin{equation}\label{eq:iwahoriweylexact}
0\rightarrow X_*(T)_I\rightarrow  W\rightarrow N_G(T)(K)/T(K)\rightarrow 1.
\end{equation}

There is a closely related variant called the \textit{affine Weyl group},
defined as
$$\Waff:= \left(N_G(T)(K)\cap \ker(\kappa_G)\right)/\left(T(K)\cap \ker(\kappa_G)\right).$$ In
other words, the affine Weyl group is the
Iwahori-Weyl group of the simply-connected cover of $G$ \cite[Equation 8.5]{twistedLoop}.

Let $\mathbb{S}$ denote the
set of reflections determined by the walls of the alcove $\a$. Then the
quadruple
$$\left(\ker\kappa_G,\I(K),\left(N_G(T)(K)\cap \ker(\kappa_G)\right),\mathbb{S}\right)$$ is a (double) Tits system
whose corresponding
Weyl group is by definition equal to
$\Waff$ {\cite[Prop. 5.2.12]{bruhatTits}.} Consequently, $\Waff$ is endowed with the structure of a
Coxeter group.
The Iwahori-Weyl group $W$ acts transitively on the
set of alcoves in the apartment $\mathcal{A}(G,S,K)$. Therefore
\begin{equation}\label{connectedcomponent}
W\cong \Waff\rtimes \Omega_\a,
\end{equation}
where $\Omega_\a\subset W$ is the normalizer of the alcove $\a$. In fact, more generally an
isomorphism
$\Omega_\a\cong \pi_1(G)_I$ can be constructed.
This splitting and the Coxeter group structure on $\Waff$, endow $W$ with the structure of a
quasi-Coxeter group; giving rise to a length function
$$\ell:W\rightarrow \N,$$ and the associated Bruhat order
$\leq$ on the Iwahori-Weyl group.
\begin{defn}\label{def:quotientbruhat}
		Let $(W,\mathbb{S})$ be a quasi-Coxeter system and $W_{J},W_{J^\prime}\subseteq W$
		subgroups.
		Let $v,w$ denote elements of the set of double cosets
		$W_J\backslash W/W_{J^\prime}$. The \textit{quotient Bruhat order} on
		$W_J\backslash W/W_{J^\prime}$ is defined by setting
		$v\leq w$ if and only if there exist $\dot v\in v$, $\dot w\in w$ with
		$\dot v\leq\dot w$.
	\end{defn}
		\begin{rmk}\label{eq:quotientbruhat}
		In the special case where $(W,\mathbb{S})$ is a Coxeter system, and $W_{J}$,
		$W_{J^\prime}$, are subgroups generated
		by reflections in $J\subset S$, $J^\prime\subset S$ respectively; the
		quotient Bruhat order admits the following equivalent characterizations:
		$v\leq w$ if and only if
		\begin{enumerate}
			\item  for the respective \textit{longest length} elements $\dot v_l\in
				v$, $\dot w_l\in w$
				we have $\dot v_l \leq \dot w_l$;
			\item for the respective \textit{shortest length} elements $\dot v_s\in
				v$, $\dot w_s\in w$
				we have $\dot v_s \leq \dot w_s$.
		\end{enumerate}
		Equivalence of these characterizations is proven in
		\cite[Lemma 2.2]{inversionDouglass}.
		\end{rmk}
\subsection{Coweights and some partial orders}\label{subsection:coweights}
Now, we spell out the quotient Bruhat order explicitly
in the case $J=J^\prime$ is the set of
reflections associated to walls passing through a special vertex.
Let $\f$ be a special vertex in the closure of the alcove $\a$ such that, $\a$ lies in the
dominant Weyl chamber determined by $B$. Choice of such an $\f$ induces a splitting of
\eqref{eq:iwahoriweylexact}, allowing us to view the Weyl group
$$\Wfin\defined N_G(T)(K)/T(K)$$ of $(G,T)$
as a subgroup of $W$. Namely, $W_J\subset W$ is the subgroup determined by the walls of $\a$
that pass through $\f$. In particular, $W_J\cong\Wfin$, which provides the desired splitting of
\eqref{eq:iwahoriweylexact}
through the inclusion $W_J\subset W$.

Given a cocharacter $\lambda\in\cham$, we may view it as an element of $W$ through the
inclusion
\eqref{eq:iwahoriweylexact}, in which case we will denote it by $t^\lambda$.
Choice of the Borel subgroup $B$ determines a set of
positive coroots $\Sigma^+\subset X_*(T)$, and a subset of dominant
elements $X_*(T)^+$.
Similarly, $\cham$ is part of the \'echelonnage root system
$\breve\Sigma$ (see \cite{hainesEchelonnage}), thus is equiped with a corresponding set of
dominant elements $\chamb$. Moreover, the positive coroots of the \'echelonnage root system are precisely
$\operatorname{im}(\Sigma^+\rightarrow X_*(T)_I)$. On the set $\cham$, there are three partial
orders, which do not have a straightforward relationship with each other:
\begin{enumerate}\label{orders}
	\item the \textit{coroot order}, $\mu\leq\lambda \Leftrightarrow
	\lambda-\mu\in\Z_{\geq 0}\cdot\operatorname{im}(\Sigma^+\rightarrow \cham)$;\label{corootorder}\\
	\item the \textit{dominance order}, $\mu\trianglelefteq\lambda\Leftrightarrow
\bar\lambda-\bar\mu\in\chamb$;\label{order}\\
	\item the \textit{Bruhat order}, $\mu\leq_{\mathrm{Bru}}\lambda$ which is the transport of
		the quotient Bruhat order under the isomorphism $\Wfin\backslash W\cong\cham$.
\end{enumerate}
\begin{rmk}
	In the literature, there seems to be a clash regarding the nomenclature of such partial
	orders. For the most part, these orders are not named; however in
	in our references \cite{hainesEchelonnage} and \cite{TimoTwistedSchubert} the order (1)
	is dubbed the dominance order, while in \cite{centralRiche} orders (1) and (3) are left unnamed
	and (2) is called the dominance order. As we systematically reference
	\cite{centralRiche}, we choose to follow their convention.
	Also note that, we record in \cref{lem:ordercomp} that for dominant
	coweights the order (1) agrees with (3); and in
	\cref{lem:concrete} we record another relationship in case $\lambda$ is dominant. Due to
	the
	existence of such relationships, order (1) is
	sometimes called the Bruhat order. However we prefer to avoid this as we apply (1) also to non-dominant
	coweights, for which there is no clear compatibility with the quotient Bruhat order. We
	refer the interested reader to \cite{acharRicheReductive}, which includes a more
	detailed study of the orders (1) and (3).
\end{rmk}

For any $\lambda\in\splitcham$, denote by $\bar\lambda\in\cham$ its image under the projection to coinvariants.
Since the Borel subgroup $B$ is defined over the base field $K$, $\Sigma^+$ is stable under the
action of $I$. As a consequence we have:
\begin{enumerate}
	\item 	The pairing $\langle\cdot,\cdot\rangle:X^*(T)\x\splitcham\rightarrow \Z$ is
		invariant under the $I$-action;
	\item  	the set $\Sigma^+$ is stable under the $I$-action and consequently
		the sum $2\rho$ of positive roots in $X^*(T)$ is also invariant under the
		$I$-action;
		\item  the natural
		projection ${X_*(T)\rightarrow \cham}$ preserves the respective coroot orders;
	\item  the natural
		projection ${X_*(T)\rightarrow \cham}$ preserves the respective dominance orders,
		in particular the image of $\splitchamb$ lands in $\chamb$. Note that this
		involves making use of the description of roots in the \'echelonnage root
		system, which are included in the invariants $X^*(T)^I$ via a slightly modified averaging
		process, see \cite[Proposition 5.2]{hainesEchelonnage}.
\end{enumerate}
For $\bar\lambda\in\cham$, we will denote by $\langle\bar\lambda,2\rho\rangle$ the number
$\langle\lambda,2\rho\rangle$ for any $\lambda\in \splitcham$ with image $\bar\lambda$ in
$\cham$. It is independent of choice of such $\lambda$ by items (1) and (2) above.

Note the following compatibility between the coroot and the Bruhat orders:
\begin{lem}\label{lem:ordercomp}
	For $\mu,\lambda\in \chamb$,
	\begin{enumerate}
		\item there is an equality $\ell(t^\mu+t^\lambda)=\ell(t^\mu)+\ell(t^\lambda)$;
		\item there is an inequality $t^\mu\leq t^\lambda$ if and only if
			$\mu\leq \lambda$ in the coroot order on $\cham$.
	\end{enumerate}
			In particular, under the identification
			$\Wfin\backslash W/\Wfin\cong
			\chamb$, quotient Bruhat order agrees with the coroot order.
	\label{comb}
\end{lem}
\begin{proof}
	\cite[Corollary 1.8]{TimoTwistedSchubert}.
\end{proof}
\subsection{Orbit stratification} We now use the results of \cref{subsection:coweights} to
describe the orbit stratification on flag varieties.
The Iwahori-Weyl group provides a set of representatives for the
$\Lplus\I$-orbits in the affine flag variety. Let $\Fl{\I,w}$ be the $\Lplus\I$-orbit of a geometric
point $\dot{w}$ corresponding to a representative of $w\in
W$ in $N_G(T)(K)$. We will regard $\Fl{\I,w}$ as a scheme with the reduced induced structure and
refer to it as the Schubert cell associated to $w$.

The Bruhat decomposition \cite[Theorem 7.8.1]{kalethaPrasad} implies the underlying
topological space of $\Fl{\I}$ admits a set theoretic
decomposition into Schubert cells
$$\Fl{\I}=\bigsqcup_{w\in  W}\Fl{\I,w}.$$
Moreover, the theory of Demazure resolutions
can be used to explicitly determine the $\Lplus\I$-orbit stratification together with its
closure relations, see \cite[\textsection 8]{twistedLoop}:

\begin{equation}\label{eq:orbitsaffine}
	{\Fl{\I,w}\cong \mathbb{A}^{\ell(w)}_k},\quad\text{and}\quad
\Fl{\I,\leq w}\defined\overline{\Fl{\I,w}}=\bigsqcup_{\substack{v\in  W\\v\leq w}}\Fl{\I,v}.
\end{equation}

\begin{rmk}\label{rmk:quotientbruhat}
Let $\G$, $\G'$ be  parahoric group schemes associated to facets $\f,\f'$ contained in the
closure of the alcove $\a$.
More generally, Demazure resolutions can be used to determine the stratification
by $\Lplus\G$-orbits in $\Fl{\G'}$, also known as $(\G,\G')$-Schubert
varieties. In this case, the strata are labelled by $W_{\f}\backslash W/W_{\f'}$, and the closure relations are
identified with the quotient Bruhat order, as can be seen from \cite[Proposition
2.8]{TimoTwistedSchubert}. Note that this stratification and its closure
relations also agree with those of right $\Lplus\G'$-orbits on
$\Lplus\G\backslash\L\G\cong\Fl{\G}$, as can be seen from
the equivalent descriptions of the quotient Bruhat order in \cref{eq:quotientbruhat}.
\end{rmk}

\subsection{Constructible sheaves on the Hecke stack}\label{subsec:sheaves}
We will consider categories of \'etale $\Qell$-sheaves and perverse t-structures on Hecke stacks, following the conventions set in \cite[Appendix B]{modularRamified}. We briefly
recall the definitions in order to fix notation, pointing the interested reader to loc. cit for
the details. Since our ring of coefficients $\Qell$ will be fixed throughout, we drop it from
the notation. Moreover, we fix once and for all a compatible system of $\ell^n$-th roots of
unity in $k$ and use it to trivialize the Tate twist.

Let $\colim_i X_i$ be a presentation of $\Fl{\mathcal G}$ as a strict ind-scheme. Fix for each
$i$,
 a natural number $n_i$ for which action of $\Lplus\mathcal{G}$ on $X_i$ factors through
 $\Lplus_{n_i}\mathcal{G}$. For any
$n\geq n_i$, the pullback functor
\begin{equation}\label{eq:nottexact}
\Dbc\left(\Lplus_{n_i}\mathcal{G}\backslash
X_i\right)\rightarrow
\Dbc\left(\Lplus_{n}\mathcal{G}\backslash
X_i\right)
\end{equation}
is an equivalence. Therefore
$$\Dbc\left(\Lplus\mathcal{G}\backslash
X_i\right)\defined\lim_{n>n_i}\Dbc\left(\Lplus_{n}\mathcal{G}\backslash
X_i\right)$$
with the limit being stationary. The closed immersion $X_i\rightarrow X_j$ induces a pushforward
$\Dbc\left(\Lplus\mathcal{G}\backslash
X_i\right)\rightarrow \Dbc\left(\Lplus\mathcal{G}\backslash
X_j\right)$.
The derived category of constructible \'etale
$\Qell$-sheaves on $\Fl{\mathcal G}$ and $\Hk{\mathcal G}$ are defined to be

\begin{equation}\label{eq:sheafdefinition}
	\Dbc\left(\Fl{\mathcal G}\right)\defined\colim_i\Dbc\left(X_i \right),\quad
 	\Dbc\left(\Hk{\mathcal G}\right)\defined\colim_i\Dbc\left(\Lplus\mathcal{G}\backslash
	X_i\right),
\end{equation}
which can be verified to be independent of the presentation $\colim_i X_i$ via standard
arguments \cite[Remark 2.2.2]{richarzModuliShtuka}. More generally, the same recipe works for any
strict ind-Artin stack ind-locally of finite type.

\begin{rmk}
	For the most part, we will regard our categories as triangulated categories. However, it is important to keep in mind that due to \cite{liuZheng}, they arise as
	homotopy categories of stable $(\infty,1)$-categories. Although none of our propositions
	have an explicit dependence on the higher structure, some of the proofs and definitions
	(most notably \cref{def:iwahoriwhittaker} and \cref{lem:conditionnotdatum}) make use of this fact.
\end{rmk}

Since pushforwards along closed immersions are perverse t-exact, perverse t-structure for
the middle perversity on the presentation
$$\colim_i\Dbc\left(X_i \right)\cong\Dbc\left(\Fl{\mathcal G}\right)$$
glue to a t-structure, which can be shown to be independent of presentation. We will also call
this t-structure the perverse t-structure. We can similarly put a perverse t-structure on
$\Dbc(\Hk{\mathcal{G}})$, by gluing perverse t-structures on
$\Dbc\left(\Lplus_n\mathcal{G}\backslash
	X_i\right)$ which are shifted to make the pullback
	$$\Dbc\left(\Lplus_n\mathcal{G}\backslash
	X_i\right)\rightarrow \Dbc(X_i)$$
	t-exact, see also \cite{farguesScholze} for similar constructions in a different
	setting.

The quotient map $h:\Fl{\mathcal{G}}\rightarrow \Hk{\mathcal{G}}$ induces a pullback functor
$$h^*:\Dbc(\Hk{\mathcal{G}})\rightarrow\Dbc(\Fl{\mathcal{G}}),$$
which is conservative, perverse t-exact and moreover detects perversity.
Restricted to perverse sheaves, $h^*$ is  even fully faithful. Indeed, denoting by
$a:\Lplus\mathcal{G}\x\Fl{\mathcal{G}}\rightarrow \Fl{\mathcal{G}}$ the action map, we have:
\begin{lem}\label{lem:perversedescent}
	The pullback restricts to a fully faithful functor
	$$h^*:\Perv(\Hk{\mathcal{G}})\rightarrow \Perv(\Fl{\mathcal{G}}).$$
	 Moreover, its
	essential image consists of objects $F\in\Perv(\Fl{\mathcal{G}})$ such that
	$$a^*F\cong pr^* F,$$
	where $pr:\Lplus\mathcal{G}\times\Fl{G}\rightarrow \Fl{G}$ denotes the projection map.
\end{lem}
\begin{proof}
	By definition, we may find a closed subscheme $X\subset \Fl{G}$ which is stable under
	the $\Lplus\mathcal{G}$ action, and contains the support
	of $F$. As $\mathcal{G}$ is connected, so is $\Lplus_n\mathcal{G}$ for every $n\geq 0$.
	Therefore, by \cite[Remark 5]{olssonPerverse} the analogous statement
	is true for the natural pullback
	$$\Perv\left(\Lplus_n\mathcal{G}\backslash X\right)\rightarrow \Perv(X),$$
	which finishes the proof by the definition of the perverse t-structures and
	\eqref{eq:sheafdefinition}.
\end{proof}
\subsection{Convolution product}\label{subsection:convolution}
As $\Fl{\mathcal{G}}$ and $\Hk{\G}$ are coset spaces of a group, we may endow their categories
of sheaves with a convolution structure. Consider the space:
$$\Hk{\mathcal{G}}\wttimes\Hk{\mathcal{G}}:=\Lplus\G\backslash\L\mathcal{G}\x^{\Lplus\G}\Fl{\G},$$
where the twisted product is formed in the \'etale topology. It has a
multiplication map $m:\Hk{\mathcal{G}}\wttimes\Hk{\mathcal{G}}\rightarrow \Hk{\G}$, induced by
the multiplication of $\L\G$. Moreover, the following isomorphism exhibits the map $m$ to be
ind-proper:
\begin{gather*}
	(pr_1,m):\Hk{\G}\wttimes\Hk{\G}\isoto \Hk{\G}\x\Hk{\mathcal{G}},\\
(g_1,g_2)\mapsto (g_1,g_1g_2).
\end{gather*}
Denote by $p:\Hk{\G}\wttimes\Hk{\mathcal{G}}\rightarrow
\Hk{\G}\x\Hk{\mathcal{G}}$ the natural projection map.
We define a bifunctor $\star:\Dbc(\Hk{\G})\x\Dbc(\Hk{\G})\rightarrow \Dbc(\Hk{\G})$
by specifying
$$F_1\star F_2:=m_!p^*(F_1\boxtimes F_2).$$
Standard arguments show that the bifunctor $\star$ endows $\Dbc(\Hk{\G})$ with a monoidal
structure, see for instance \cite{aglr}.

\begin{rmk}\label{remark:convolution}
	Similar constructions on the spaces
	$\Fl{\G}\wttimes\Hk{\G}$ and $\Hk{\G}\wttimes\Fl{\G}$
	 define left and right actions of $\Dbc(\Fl{\G})$ on $\Dbc(\Hk{\G})$ via convolution product. As
	$h:\Fl{\G}\rightarrow \Hk{\G}$ is an \'etale torsor under the pro-smooth group $\Lplus\G$, smooth
	base change implies all of these convolution structures are compatible with the conservative pullback
        functor	$h^*:\Dbc(\Fl{\G})\rightarrow \Dbc(\Hk{\G})$. Therefore, we will not distinguish
	between them in notation.
\end{rmk}
\subsection{Affine Grassmannians}\label{subsection:grassmannians}
Note the following absolute variants of the loop, resp. the positive loop, group:
$$\L_z G:R\mapsto G(R((z))), \quad \Lplus_z G:R\mapsto G(R[[z]]),$$
where $R$ is a $K$-algebra and $z$ is an additional formal variable. The \textit{affine Grassmannian} $G$,
	is the \'etale-quotient
	$$\Gr_G\defined\L_z G/\Lplus_z G.$$
	Similarly, we have the Hecke stack of $G$
	$$\Hk{G}\defined \Lplus_z G\backslash \Gr_G.$$
The categories $\Dbc(\Gr_G)$ and $\Dbc(\Hk{G})$ are also equipped with convolution monoidal
structures, as in \cref{subsection:convolution}.
\subsection{Geometric Satake equivalence} Let us briefly recall how Langlands duality enters the
picture. Let $\check{G}$ be
the pinned reductive group over $\Qell$, whose root datum is dual to the root datum of the split reductive group
$G_{\bar K}$. It is called the \textit{Langlands dual group of} $G$.
The group $\check{G}$ is equipped $\operatorname{Gal}(\bar K/K)$-action via pinned automorphisms.

The (absolute) geometric Satake equivalence is originally due to
\cite{ginzburgSatake}, \cite{mirkovicvilonenSatake}. See
\cite{modularRamified} for the current state of affairs, and a comprehensive overview.
It is a canonical equivalence of monoidal categories:
$$\left(\Perv\left(\Hk{G}\right),\star\right)\cong (\Rep(\check G), \otimes),$$
where $\Rep(\check G)$ denotes the category of finite dimensional $\Qell$-representations of
$\check G$, and
$\otimes$ its usual monoidal structure given by the tensor product on the underlying vector spaces.

We have the following variant for ramified groups, due in our setting to
\cite{zhuRamified} and \cite{richarzRamified}.
Let $\G$ be a parahoric model of $G$, associated to a special vertex $\f$.
Then, there is an equivalence of monoidal categories
$$(\Perv(\Hk{\G}),\star)\cong (\Rep(\check G^I), \otimes),$$
where $\check G^I$ denotes the group of fixed points.

Note also that since $I$ acts by pinned automorpshims
 of $\check G$, it also acts on the Borel $\check B$ and the torus $\check T$ determined by the
 pinning.
\subsection{BD-Grassmannian and nearby cycles}\label{subsection:BDGrass} The geometric Satake
equivalence and its ramified version are related to
each other through \textit{Beilinson-Drinfeld Grassmannians}, which are
ind-schemes that deform affine Grassmannians into twisted affine flag varieties. We refer
to \cite[Section 2]{richarzTestFunction} for their definition.

Let $S:=\Spec(O_K)$, and denote by $s$ and $\eta$ its
special and  generic points, respectively.
Let $\GrBD{\G}$ be the The BD-Grassmannian associated to a parahoric model $\G$ of $G$. It
satisfies:
$${\GrBD{\G}}\x_S\eta \cong\Gr_{G}, \quad {\GrBD{\G}}\x_S s\cong\Fl{\G}.$$
Then, the nearby cycles construction induces a functor $$\Psi_\G:\Dbc(\Hk{G,\bar K})\rightarrow\Dbc(\Hk{\G}),$$
which is perverse t-exact and monoidal with respect to the convolution product, see \cite[Proposition 8.6]{modularRamified}.
\begin{lem}\label{lemma:nearbyrest}
The geometric Satake equivalences fit in the commutative diagram:
$$\begin{tikzcd}
	\Perv(\Hk{G,\bar K})\ar[d,"\Psi_\G"']\ar[r, dash, "\sim"]& \Rep(\check
	G)\ar[d,"\Res"]\\
	\Perv(\Hk{\G})\ar[r, dash,"\sim"]& \Rep(\check G^I).
\end{tikzcd}$$
\end{lem}
\begin{proof}
\cite[Corollary 8.9]{modularRamified}.
\end{proof}

The nearby cycles functor comes equipped with an quasi-unipotent action of $I$, from which we can
extract the \textit{monodromy operator}. Briefly, for a sheaf $\cA$ on $\Gr_{G,\bar K}$,
there is a functorial assignment of
continuous group homomorphisms
$$\rho_{\cA}:I\rightarrow \Aut\left(\Psi_\G(\cA)\right);$$
and for each such there exist a finite index subgroup $I'\subset I$ acting unipotently.
The monodromy operator is the unique nilpotent map
\begin{equation}
	\bn_{\cA}:\Psi_\G(\cA)\rightarrow \Psi_\G(\cA)(-1),
\end{equation}
satisfying for every $\gamma\in I'$, $\rho_{\cA}(\gamma)=\exp(t_\ell(\gamma)\cdot\bn_{\cA})$, where
$t_\ell:I\rightarrow
\Z_\ell(1)$ is the
natural projection to the Tate module (for more details see e.g. \cite[p. 21]{richarzRamified}).

The monoidal structure on $\Psi_\G$ is further compatible with the monodromy:
for every $\cA,\cB\in\Dbc(\Gr_{G,\bar K})$, the diagram
\begin{equation}\label{eq:monodromymonoidal}
	\begin{tikzcd}[column sep=7 em]
		\Psi_\G(\cA)\star\Psi_\G(\cB)\ar[d,"\sim" {anchor=south, rotate=90, inner sep=.5mm}]\ar[r,"\bn_{\cA}\star
		\id_\cB+\id_\cA\star\bn_{\cB}"] &
		\Psi_\G(\cA)\star\Psi_\G(\cB)\ar[d,"\sim"{anchor=south, rotate=90, inner sep=.5mm}]\\
		\Psi_\G(\cA\star\cB)\ar[r,"\bn_{\cA\star\cB}"']&
		\Psi_\G(\cA\star\cB)
	\end{tikzcd}
\end{equation}
is commutative, by arguments going back to \cite{gaitsgoryCentral}.

\begin{rmk}\label{remark:commutativitymonodromy}
	When the group $G$ is tamely ramified, the monodromy operator in each parahoric model commutes with every other
	morphism. More precisely, let $\cA,\cB\in \Dbc(\Hk{G,\bar{K}})$ and $\G$ be a parahoric
	model of $G$. Then, for every morphism
	$f:\Psi_\G(\cA)\rightarrow \Psi_\G(\cB)$, there is an equality
	$$f\circ\bn_{\cA}=\bn_{\cB}\circ f.$$
	This is a consequence of the existence of a $\Gm$-action on $\GrBD{\G}$ via loop rotations
	\cite[Lemma 5.4]{coherenceZhu}. Indeed, the
	monodromy of the nearby cycles can be identified with the opposite of Verdier's
	construction of the monodromy for $\Gm$-monodromic sheaves \cite[Proposition
	7.1]{verdierMonodromy}. The assertion then follows from that of $\Gm$-monodromic
	sheaves, see e.g. \cite[Proposition 9.3.2]{centralRiche}.
\end{rmk}
\subsection{Central functor}\label{subsection:centralfunctor}
When the parahoric model is an Iwahori, the corresponding nearby cycles functor $\Psi_\I$ will
instead be denoted
$$\Zent:\Rep(\check G)\rightarrow \Perv(\Hk{\I}),$$
and dubbed \textit{the central functor}. It was introduced by Gaitsgory in \cite{gaitsgoryCentral} to
categorify the Bernstein isomorphism, and used by Arkhipov-Bezrukavnikov in
\cite{arkhipovBezrukavnikov} as the fundamental building block of their strategy.

The following theorem summarizes the properties of $\Zent$ which are relevant to us:
\begin{thrm}[Gaitsgory, Zhu]\label{theorem:gaitsgoryzhu}
	For every $V\in\Rep(\check G)$ and $F\in\PervI$, both $\Zent(V)\star F$ and
	$F\star\Zent(V)$ are objects in $\PervI$, and as such there exist a canonical
	isomorphism
	$$\Zent(V)\star F\simeq F\star \Zent(V).$$
\end{thrm}
\begin{proof}
	\cite[Theorem 6.14]{richarzTestFunction}.
\end{proof}

In a monoidal category, central objects are those for which tensoring from either side is
isomorphic, objectwise. Although $\Perv(\Hk{\I})$ is not monoidal,
the theorem still asserts that the objects in the essential image of $\Zent$
satisfy a similar property, justifying the name.
\begin{rmk}
Although we will not make use of it, let us mention further structure of the central functor.
Namely, $\Zent$
can be upgraded to a monoidal functor to the Drinfeld center of $\Dbc(\Hk{\I})$ which intertwines the
symmetry isomorphims of $\Rep(\check G)$ with the natural braiding structure of the Drinfeld
center. See for instance \cite[Chapter 3]{centralRiche} for a concrete discussion of what this
means or \cite[Theorem 1.1]{joaoAB} for a highly structured $(\infty,1)$-categorical version.
\end{rmk}
\subsection{Standard and costandard functors}
For $w\in W$, the locally closed inclusion of the Schubert cell $j_w:\Fl{\I,w}\rightarrow \Fl{\I}$
is the pullback of the map
$${s_w:\Lplus\mathcal{G}\backslash\Fl{\I,w}\rightarrow \Hk{\I}}$$
along the projection $h:\Fl{\I}\rightarrow\Hk{\I}$.

The \textit{standard}, resp. \textit{costandard}, functor associated to $w$ is defined as:
\begin{center}
\begin{tabular}{cc}
	&\\
	$\Delta_w(\cdot):\Dbc\left(\Spec(k)\right)\rightarrow \Dbc(\hk),$ &
	$\nabla_w(\cdot):\Dbc\left(\Spec(k)\right)\rightarrow \Dbc(\hk),$\\
	$M\mapsto s_{w!}(\underline M)[\ell(w)]$ & $M\mapsto s_{w*}(\underline M)[\ell(w)]$\\
	&
\end{tabular}
\end{center}
where $(\underline{\;\cdot\;}):\Dbc(\Spec(k))\rightarrow\Dbc\left(\Lplus\mathcal{G}\backslash\Fl{\I,w}\right)$
denotes the constant complex functor.
Since the coefficient $\Qell$ will frequently appear  in the rest of the document,
we establish the convention $\std_w:=\std_w(\Qell)$ and $\costd_w:=\costd_w(\Qell)$,
whenever they appear without an argument.

Using smooth base change, it is straightforward to verify that
	\begin{equation}\label{eq:stdcostd}
		h^*\std_w(\cdot)\cong j_{w!}(\underline\cdot)[\ell(w)],\quad
		h^*\costd_w(\cdot)\cong j_{w*}(\underline\cdot)[\ell(w)].
	\end{equation}
	Using that $h^*$ detects perversity, we may then prove:
\begin{lem}
	For any $w\in W$, the functors $\std_w(\cdot),\costd_w(\cdot):\Dbc\left(\Spec(k)\right)\rightarrow
	\Dbc(\Hk{G})$ are perverse t-exact.
\end{lem}
\begin{proof}
	By \eqref{eq:stdcostd}, it suffices to prove the analogous statement for
	$j_{w!}(\underline\cdot)[\ell(w)]$ and $j_{w*}(\underline\cdot)[\ell(w)]$. Since $j_w$ is a locally closed immersion
	and $\Fl{\I,w}\cong\mathbb{A}^{\ell(w)}$, and pushforward along locally closed affine morphisms are
	t-exact by \cite[Corollaire 4.1.3]{BBDG} the assertion follows.
\end{proof}
Moreover, we may use $\std_w$ and $\costd_w$ to obtain the simple perverse sheaf
$$\IC_w\defined \im(\std_w\rightarrow \costd_w);$$
as the image of the natural morphism. In fact, \cite[\textsection 8]{olssonPerverse} and
standard arguments show that; the assignment $w\mapsto \IC_w$ induces a bijection between $W$
and isomorphism classes of simple objects in $\Perv(\Hk{G})$.

Alongside their connection to simple objects, one of the most important properties
of standard and costandard functors is their compatibility with convolution.
As \cref{std:conv} below makes precise, they intermingle the quasi-Coxeter structure on $W$
with the convolution product on $\Perv(\Hk{G})$.

\begin{prop}\label{std:conv}
	\begin{enumerate}
		\item
For $w_1,w_2\in W$ such that $\ell(w_1w_2)=\ell(w_1)+\ell(w_2)$, there exist canonical isomorphisms
	$$\Delta_{w_1}(\cdot)\star\Delta_{w_2}(\cdot)\isoto\Delta_{w_1w_2}(\cdot\otimes
			\cdot),\quad \costd_{w_1}(\cdot)\star\costd_{w_2}(\cdot)\isoto\costd_{w_1w_2}(\cdot\otimes
			\cdot)$$
			satisfying associativity.
\item For $w\in W$, there exist isomorphisms
	$$\Delta_w\star\nabla_{w^{-1}}\isoto
			\nabla_{w^{-1}}\star\Delta_{w}\isoto\Delta_e.$$
			Therefore, $\Delta_w$ and $\nabla_w$ are invertible for the convolution product. \end{enumerate}

	\end{prop}
\begin{proof}
	(1) It suffices to prove it after pullback by $h:\Fl{\I}\rightarrow\Hk{\I}$.
	Under the hypothesis, the Bruhat decomposition yields
	$$\Lplus\I(k)w_1\Lplus\I(k)w_2\Lplus\I(k)=\Lplus\I(k)w_1w_2\Lplus\I(k),$$ which implies the
			multiplication map
			$$m:\Fl{\I,w_1}\wttimes\Fl{\I,w_2}\rightarrow \Fl{\I,w_1w_2}$$
			is an isomorphism.
		For $\Delta$, the result then follows from K\"unneth formula for
	shriek pushforward, and smooth base change:
	\begin{align*}
		h^*\left(\Delta_{w_1}(\cdot)\star\Delta_{w_2}(\cdot)\right)\cong&
		m_!\circ \left( j_{w_1!}(\underline\cdot)[\ell{(w_1)}]\tboxtimes j_{w_2!}
		(\underline \cdot)[\ell{(w_2)}])\right)\\
		\cong& m_!(j_{w_1}\wttimes j_{w_2})_!(\underline\cdot\otimes
		\underline \cdot)[\ell{(w_1w_2})]\\
		\cong& j_{w_1w_2!}(\underline\cdot\otimes\underline\cdot)[\ell{(w_1w_2})]\\
		\cong& \Delta_{w_1w_2}(\cdot \otimes
		\cdot).
	\end{align*}
			The proof for $\costd$ is similar, using the fact that star pushforward
			satisfies K\"unneth formula for ULA-sheaves, which includes every
			constructible sheaf
			when we are over a field \cite[Theorem 4.1]{RelativePerversity}.

			(2) The proof of the analogous statement in \cite[Lemma 4.1.4]{centralRiche}
			goes through. By choosing a reduced word for $w$, we can reduce to the
			case $\ell(w)\in\{0,1\}$, of which only the non-zero case is
			non-trivial. In this case, $w=s$ is a simple reflection, and the proof proceeds via a
			calculation for
			$$m:\Fl{\I,\leq s}\wttimes\Fl{\I,\leq s}\rightarrow \Fl{\I,\leq s},$$
			which can be identified explicitly as the projection
			$\mathbb{P}_k^1\x\mathbb{P}_k^1\rightarrow \mathbb{P}_k^1$ using the
			Bruhat decomposition. For the details of the calculation see loc. cit.
\end{proof}
\section{Wakimoto sheaves and central objects}\label{sec:wakimoto}
In this section we introduce certain functors $\Dbc(\Spec(
k))\rightarrow \Dbc(\Hk{\I})$ attached to elements of $\cham$. These
functors provide a
categorification of Bernstein's
presentation for the center of the Hecke algebra. They were introduced in the setting of
split reductive groups by Mirković in an unpublished note. For ramified groups, a version of them appears in
\cite[Section 7.3]{coherenceZhu}, and more recently in \cite{joaoAB} in the mixed characteristic
setting.
We will, in the ramified equal characteristic setting, explain how to construct
them and prove analogues of their properties, due in the split setting to Arkhipov-Bezrukavnikov
\cite{arkhipovBezrukavnikov}.  In our treatment, statements and proofs
appearing closely follow \cite[Section 4.2]{centralRiche}. Although they work in the setting of
a split reductive group over complex numbers, and sheaves with respect to the analytic topology,
many proofs carry over to our situation.

Motivated by the definition of Bernstein translation elements,
the Wakimoto sheaf associated to $\mu\in\cham$ should be isomorphic to
$\nabla_\lambda\star\Delta_{\lambda-\mu}$ for a
$\lambda\in\chamb$ such that $\mu \trianglelefteq \lambda$ with respect to the dominance order
\eqref{orders}.
However, due to choices involved in
\cref{std:conv}, we cannot expect such an isomorphism to be
canonical. In order to eliminate choice,
we will define the Wakimoto sheaf as a limit
of a diagram consisting of objects
$\Delta_\lambda(\cdot)\star\nabla_{\lambda-\mu}$ for all
 $\lambda\in \chamb$ such that $\mu\trianglelefteq \lambda$; and morphisms being the isomorphisms
indexed by the relevant choices. Instead of explicitly writing this diagram out,
we will repackage it following the strategy of \cite{centralRiche}.
\subsection{Definition of Wakimoto functors}
\label{subsec:definition}
	Let $\mu\in \cham$. Consider the partially ordered set
	$$S_\mu=\{\lambda\in \chamb|\mu\trianglelefteq \lambda\}.$$
	Then there is a diagram $d:S_\mu\rightarrow
	\operatorname{Func}\left(\Dbc(\Spec (k))\op\times
	\Dbc(\Hk{\I}),\operatorname{Set}\right)$ specified on the objects by:
		$$d(\lambda)=\Hom\left(\nabla_\lambda(\cdot),(\cdot)\star\nabla_{\lambda-\mu}\right),$$
		and sending the morphism $d(\lambda\trianglelefteq  \lambda')$ to:
		$$\Hom\left(\nabla_\lambda(\cdot),(\cdot)\star\nabla_{\lambda-\mu}\right)\isoto
		\Hom\left(\nabla_{\lambda}(\cdot)\star\nabla_{\lambda'-\lambda},(\cdot)\star\nabla_{\lambda-\mu}\star\nabla_{\lambda'-\lambda}\right).$$
		Above we identify $d(\lambda')=
		\Hom\left(\nabla_{\lambda}(\cdot)\star\nabla_{\lambda'-\lambda},(\cdot)\star\nabla_{\lambda-\mu}\star\nabla_{\lambda'-\lambda}\right)$ through the canonical isomorphism from
	\cref{std:conv}.

	\begin{lem}
		Fix $M \in \Dbc(\Spec (k))$
		and consider the functor
		\begin{gather*}
		A_M^\mu:\Dbc(\Hk{\I})\rightarrow
	\operatorname{Set},\\
		F\mapsto \colim d(M,F)
			\end{gather*}
		defined likewise on morphisms. Then, $A^\mu_M$ is corepresentable.
		\label{noncanon}
	\end{lem}
	\begin{proof}
        Pick some $\lambda\in S$. Then \cref{std:conv} implies
		$$\Hom\left(\nabla_\lambda(M)\star\Delta_{\mu-\lambda},(\cdot)\right)\isoto\Hom\left(\nabla_\lambda(M),(\cdot)\star\nabla_{\lambda-\mu}\right).$$
		Therefore, each functor in the diagram $d(M,\cdot)$
		is corepresentable. As all morphisms in the diagram are isomorphisms,
		the colimit of the corresponding diagram in $\Dbc(\Hk{\I})$ exists,
		and is the desired corepresenting object.
	\end{proof}
	\begin{defn}
	For $\mu\in\cham$, define the \textit{Wakimoto functor},
		$$J_\mu: \Dbc(\Spec (k))\rightarrow
		\Dbc(\Hk{\I})$$
		by declaring $J_\mu(M)$ to be object corepresenting $A^\mu_M$ and likewise on
		morphisms.
						\end{defn}
\begin{rmk}
	Proof of \cref{noncanon} shows that for any $\lambda\in S_\mu$,
	$$J_\mu(\cdot)\isoto
	\nabla_\lambda(\cdot)\star\Delta_{\mu-\lambda}.$$
	However the isomorphism depends, among other things, on the choice
	of an isomorphism of
	$\nabla_{\lambda-\mu}\star\Delta_{\mu-\lambda}$
	to the monoidal unit of $\star$. Therefore it is
	non-canonical as discussed in the beginning of this section. As before, we use the
	convention that
	$J_\mu=J_\mu(\Qell)$ whenever it appears without an argument.
	\label{remark}
\end{rmk}

The following lemma is useful for studying the behaviour of
convolution with respect to the perverse t-structure.

\begin{lem}\label{lem:orbitmultaffine}
	Let $w\in W$ and $X$ be a closed finite union of $\Lplus\I$-orbits
	in $\Fl{\I}$. Denote by $\pi:\LG\rightarrow\Fl{\I}$ the quotient morphism. Then the
	following restrictions of the multiplication morphism are affine:
	$$(\Lplus\I w\Lplus\I)\times^{\Lplus\I}X\rightarrow \Fl{\I},\quad
	\pi\inv(X)\times^{\Lplus\I} \left((\Lplus\I w\Lplus\I)/\Lplus\I \right) \rightarrow
	\Fl{\I}.$$
\end{lem}
\begin{proof}
	We start with the first map.
	Let $\dot w$ be a representative of $w$ in $N_G(T)$ and $\Lplus\I_w$
	the stabilizer of $\dot w\in\Fl{\I}$. Pick a closed
	finite union of $\Lplus\I$-orbits $Y\supset \dot w\cdot X$. Then the
	map can be factored as:
	$$\begin{tikzcd}
		\Lplus\I \times^{\Lplus\I_w}\dot w\cdot
		X\ar[r,hookrightarrow]&
		\Lplus\I \times^{\Lplus\I_w} Y\ar[r,"\simeq"] &
		\Lplus I/\Lplus\I_w\times Y\ar[r,"pr_Y"] &
		Y\ar[r,hookrightarrow] &\Fl{\I} \\
		& \left[i,y\right] \ar[r,mapsto]& (i\Lplus\I_w,i\cdot y).
		&&
	\end{tikzcd}$$
	Now, the first and the last map are closed immersions, thus are affine. The
	projection
	$pr_Y$ is also affine since $\Lplus\I/\Lplus\I_w\isoto
	\Fl{\I,w}\isoto \mathbb{A}_k^{\ell(w)}$, proving the lemma for the first map. The proof
	for the second map is analogous.
\end{proof}
An immediate application of \cref{lem:orbitmultaffine} is:
\begin{prop}
	For any $w\in W$ and $M\in\Vect_{\Qell}$,
	the functors
	$$\nabla_w(M)\star(\cdot):\Dbc(\Hk{\I})\rightarrow \Dbc(\Hk{\I}),\quad \Delta_w(M)\star(\cdot):\Dbc(\Hk{\I})\rightarrow \Dbc(\Hk{\I})$$
are perverse t-exact. Same holds for convolution from the right.
\end{prop}
\begin{proof}
	Let $h:\Fl{\I}\rightarrow \hk$, $\pi:\LG\rightarrow\Fl{\I}$ denote the respective quotient maps. The t-exactness
	of $\Delta_w(M)\star(\cdot)$ can be checked objects wise, and after
	pullback by $h$. Note that, ind-properness of the multiplication map
	$m:\Fl{\I}\wttimes\Fl{I}\rightarrow \Fl{I}$ implies that
\begin{align}\label{stdcostd:convolution}
		m_*\left(\underline M\tboxtimes(\cdot)\right)
	\cong
	h^*(\Delta_w(M)\star(\cdot))\cong m_!\left(\underline M\tboxtimes(\cdot)\right).
	\end{align}
		For any $y\in W$, \cref{lem:orbitmultaffine} implies the restiction of the multiplication map
$$\begin{tikzcd}
	\Fl{\I,w}\wttimes{\Fl{\I,\leq y}}\ar[r,"m"]&\Fl{\I}.
\end{tikzcd}$$
is affine. By \cite[Corollaire 4.1.2]{BBDG}, shriek pushforward of an affine morphism is perverse
	left t-exact, while the star pusforward of such is right t-exact \cite[Th\'eor\`eme
	4.1.1]{BBDG}. Therefore, \ref{stdcostd:convolution} exhibits an isomorphism of
	$\Delta_w(M)\star(\cdot)$ to a left t-exact and right t-exact
	functor simultaneously, proving it is t-exact. The other cases may be proven similarly.
\end{proof}
In particular, Wakimoto functors are perverse t-exact. When
restricted to the heart, they are even fully-faithful and have
an essential image which is closed under extensions:
\begin{prop}\label{wakimoto:essim}
	For all $\mu\in \cham$,
	\begin{enumerate}
		\item  The functor
			$$J_\mu:\Vect_{\Qell}\rightarrow \Perv(\Hk{\I})$$ is fully faithful.
		\item For all $M,M'\in \Vect_{\Qell}$, the natural map
			$$0=\Ext^1(M,M')\rightarrow \Ext^1(J_\mu(M),J_\mu(M'))$$
			is an isomorphism.
			Therefore the essential image of $J_\mu$ is closed under extensions.
				\end{enumerate}
\end{prop}
\begin{proof}
	\begin{enumerate}
		\item
			By \cref{lem:perversedescent}, we can check this after pullback by
			$h:\Fl{\I}\rightarrow\Hk{I}$.
	Pick $\lambda\in\chamb$ such that $\mu-\lambda\in\chamb$ as well.
			As $\nabla_{\lambda-\mu}$ is tensor invertible (\cref{std:conv}),
	$$\Hom(h^*J_\mu(M),h^*J_\mu(M'))$$ is isomorphic to
				$$\Hom\left(\left(h^*\circ
	J_\mu(M)\right)\star\nabla_{\lambda-\mu},h^*
	\left(J_\mu(M')\star\nabla_{\lambda-\mu}\right)\right).$$
	Using once
			again \cref{std:conv} and \cref{lem:perversedescent}, we see this is in turn isomorphic to
	$\Hom(\nabla_\lambda(M),\nabla_\lambda(M'))$. Finally, smooth base change and
			$\mathbb{A}^1$-invariance gives the desired bijection with $\Hom(M,M')$.
\item
	One can reduce to $\mu\in\chamb$ as above, in which case $J_\mu(\cdot)\cong
			\nabla_\mu(\cdot)$. The assertion then follows from the fact that the equivalence
			$$\Vect_{\Qell}\cong \Perv(\Lplus\I\backslash\Fl{\I,t^\mu})$$ induced by
			$h^*_{t^\mu}:\Perv(\Lplus\I\backslash\Fl{\I,t^\mu})\rightarrow
			\Perv(\Fl{\I,t^\mu})$ and $\mathbb{A}^1$-invariance, identifies $j_{t^\mu*}$ with
			$J_\mu$.
	\end{enumerate}
\end{proof}
Wakimoto functors are further compatible with convolution in the following
sense:
\begin{lem}\label{lem:wakimotomonoidality}
	For all $\mu,\lambda\in{X_*(T)}$ and all $M,M'\in \Dbc(\Spec (k))$ there exist
	canonical isomorphisms
	$$J_\mu(M)\star J_\lambda(M')\cong J_{\mu+\lambda}(M\otimes M').$$
\end{lem}
\begin{proof}
	This follows from \cref{std:conv} for $\mu,\lambda\in \chamb$.
	General case can be reduced to this after convolving with
	$\nabla_{\nu-\mu-\lambda}$ for some $\nu\in\chamb$, such that
	$\nu\trianglelefteq \mu$ and $\nu\trianglelefteq \mu+\lambda$ in the dominance order
	\eqref{order}.
\end{proof}
\subsection{Wakimoto filtration}
\begin{prop}\label{wakimoto:filtration}
	Convolution exact central objects in $\PervI$ admit a finite filtration, whose graded
	pieces are of the form $J_\lambda(M)$ for $\lambda\in\cham$, and
	$M\in\Vect_{\Qell}$.
\end{prop}
Before proceeding with the proof we introduce some
notation and general facts.
For a set $S\subset \text{Ob}(\Dbc(\Fl{\I}))$, denote by
$\langle S\rangle$ the set of objects contained in the smallest thick subcategory generated by objects in $S$.
For an object $F\in\Dbc(\Fl{\I})$, its star and shriek support are the sets
$$\starsup(F):=\{w\in  W \;|\; j_w^*F\neq 0\},\quad
\shrieksup(F):=\{w\in  W \;|\; j_w^!F\neq 0\}$$
respectively. We will also use the same notation for objects in $\Dbc(\Hk{\I})$, understood to be
applied to the corresponding object of $\Dbc(\Fl{\I})$ after pullback.

\begin{lem}\label{lemma:wakimotoaux2}
	For any $F\in\Dbc(\Fl{\I})$,
	$$F\in\langle\{j_{w!}\circ j_w^* F\;|\; w \in\starsup(F)\}\rangle,\quad
	  F\in\langle\{j_{w*}\circ j_w^! F\;|\; w \in\shrieksup(F)\}\rangle.$$
\end{lem}
\begin{proof}
	Let $X$ denote the closure of $\starsup(F)$, and $w\in W$ an element such that
	$j_w:\Fl{\I,w}\rightarrow X$ is an open immersion. Denote by $i:X\setminus \Fl{\I,w}\rightarrow X$,
	the inclusion of the complement. The first assertion follows from a
	straightforward induction on the number of $\Lplus\I$ orbits contained in $X$, using the localization triangle:
		$$i_!i^!F\rightarrow F\rightarrow j_{w*}j_w^*F.$$
		The other one is similar, using instead the localization triangle
		$$j_{w!}j_w^!F\rightarrow F\rightarrow i_*i^*F.$$
\end{proof}
\begin{lem}\label{lemma:wakimotoaux1}
	For $F\in\Dbc(\Hk{\I})$, $w\in W$, there exist a finite subset
	$S_{F,w}\subset W$ such that
	\begin{gather*}
	\starsup(\Delta_w\star F),\; \shrieksup(\nabla_w\star F)\subset
	w\cdot S_{F,w},\\
	\starsup(F\star\Delta_w),\; \shrieksup(F\star\nabla_w)\subset
	 S_{F,w}\cdot w.
	\end{gather*}
\end{lem}
\begin{proof}
	Let $X\subset \LG$ be the smallest closed union of $\Lplus\I$
	orbits containing the support of $p^*F$. The lemma follows if
	there exist a finite
	set $S_X$ such that
	$$\Lplus\I w\Lplus\I \times^{\Lplus\I}X\subset\bigcup_{\nu\in
	S_Xw}\Lplus\I\nu\Lplus\I/\Lplus\I,\quad
	X\times^\mathcal{I}\mathcal{I}w\mathcal{I}/\mathcal{I}\subset\bigcup_{\nu\in
	wS_X}\mathcal{I}\nu\mathcal{I}/\mathcal{I}.$$
	Existence of such an $S_X$ can be shown by a straightforward induction on $\ell(w)$.
\end{proof}
Let $\cham^{-}\defined-\chamb$ be the closure of the antidominant chamber, and $\cham^{--}$ its interior.
A coweight
contained in $\cham^{-}$ (resp. $\cham^{--}$) is called
\textit{antidominant} (resp. \textit{strongly antidominant}).
\begin{proof}[Proof of \cref{wakimoto:filtration}]
	Let $F\in\PervI$ be a convolution exact central sheaf.
	Since $S_{F,e}$ is finite, we
	find a strongly antidominant coweight $\mu$ such that
	$$t^\nu\cdot S_{F,e}\subset\{ t^{\mu}\;|\;\mu\in(\cham^{--}\cdot\Wfin)\},\quad S_{F,e}\cdot
	t^\nu\subset\{ t^{\mu}\;|\;\mu\in(\Wfin\cdot\cham^{--})\}.$$
	Consequently, by \cref{lemma:wakimotoaux1} and the centrality of $F$, the star support
	of $F\star J_\nu$ is contained in the intersection $(\Wfin\cdot \cham^{--})\cap(\cham^{--}\cdot\Wfin)$; which by
	\cite[\textsection 13.2, Lemma A]{humphreysIntro} implies
	$$\starsup(F\star
	J_\nu)\subset \cham^{--}.$$
	For an antidominant coweight $\mu$,
	there is an isomorphism $J_\mu\cong \std_{t^\mu}$;
	which using \cref{lemma:wakimotoaux2} shows
	$$F\star J_\nu\in\left<\{J_\lambda[n]\;|\; \lambda\in \cham^-, n\geq 0\}\right>.$$
	Picking a large enough $\eta\in\chamb$, as we may using the finiteness of $S_{F,e}$, we
	can obtain by convolution monoidality of Wakimoto sheaves
	that
	$$F\star J_{\eta+\nu}\in\left<\{J_\lambda[n]\;|\; \lambda\in\chamb, n\geq 0\}\right>.$$
	The inclusion above implies that, for every
	$\lambda\in\chamb$, the restriction $j_{t^\lambda}^*(F\star J_{\eta+\nu})$ is
	concentrated in degrees less than or equal to $-\ell(t^\lambda)$. Moreover, the perversity of $F\star J_{\eta+\nu}$ implies the reverse inequality, showing that
	the restriction is concentrated in degree $-\ell(t^\lambda)$.

	Finally, via a straightforward induction on the cardinality of $\shrieksup(F\star
	J_{\eta+\nu})$ (cf. \cite[Lemma 4.3.7]{centralRiche}), this
	implies that $F\star
	J_{\eta+\nu}$ admits a Wakimoto filtration; implying likewise for $F$ by convolving with
	$J_{-(\eta+\nu)}$.
\end{proof}

\subsection{Associated graded functor} In this subsection, we will construct functors which take
a Wakimoto filtered perverse sheaf to its $J_\lambda$-graded piece.
Let us first record the following result about extensions between Wakimoto sheaves:
\begin{prop}\label{wakimoto:extension}
	For $\mu,\lambda\in{X_*(T)_I}$,  $M,M'\in \Db(\Spec k)$, and $n\in \Z$ one has
	$$\Hom(J_\mu(M),J_\lambda(M')[n])=0,$$
	unless $\mu\leq \lambda$ in the coroot order.
\end{prop}
\begin{proof}
	Pick $\nu\in{X_*(T)_I}$ such that $\nu+\mu,\nu+\lambda\in\chamb$.
	Then
	$$\Hom(J_\mu(M),J_\lambda(M')[n])\isoto\Hom(J_\mu(M)\star
	J_\nu,J_\lambda(M')\star J_\nu[n]).$$
	Therefore, by the compatibility of Wakimoto functors with convolution
	(\cref{lem:wakimotomonoidality}), they are isomorphic to
	$$\Hom(\nabla_{\mu+\nu}(M),\nabla_{\lambda+\nu}(M'))\isoto
	\Hom(j_{t^{\lambda+\nu}}^*\nabla_{\mu+\nu}(M),\underline{M}').$$
	This concludes the proof, as $(t^{\lambda+\nu}\cdot\Lplus\I)/\Lplus\I\in \overline{\Fl{\I,t^{\mu+\nu}}}$ if and only if
	$\lambda+\nu\leq_{\operatorname{Bru}}\mu+\nu$, which is equivalent to $\lambda\leq \mu$ by \cref{comb}.
\end{proof}
Given a subset $\Omega\subset \cham$, denote by $\PervI^\Omega$ the full subcategory of
$\PervI$, consisting of objects which admit a finite filtration by
Wakimoto sheaves labelled by cocharacters in $\Omega$. We will mainly be interested in subsets
$\Omega$ which satisfy: for every $\lambda\in \Omega$, $\mu\leq\lambda$ implies that
$\mu\in\Omega$. Such subsets of a partially ordered set are called ideals.
\begin{lem}
	Given an ideal $\Omega\subset\cham$, the full embedding
	$$\PervI^\Omega\rightarrow\PervI^{\cham}$$ admits a
	right adjoint. Moreover, for every $F\in \PervI^{\cham}$, the cokernel $F/F_{\Omega}$ of the unit of the
	adjunction, lies in $\PervI^{\cham\setminus\Omega}$.
\end{lem}
\begin{proof}
	The proof follows as in \cite[Lemma 4.3.2]{centralRiche}, by induction on
	the length of the filtration, using \cref{wakimoto:extension} to control extensions.
\end{proof}
Given a cocharacter $\lambda\in\cham$, we will denote by $\{\leq\lambda\}$ (resp. $\{<\lambda\}$) the ideal of
elements in $\cham$ which are less than (resp. strictly less than) $\lambda$.
\begin{defn}
For every $\lambda\in\cham$, one defines the functors:
	$$\gr_\lambda:\PervI^{\cham}\rightarrow \PervI,\quad
	\Grad_\lambda:\PervI^{\cham}\rightarrow \Vect_{\Qell},$$
	as
	$$\gr_\lambda(F)\defined F_{\leq\lambda}/F_{<\lambda},\quad
	\Grad_\lambda(F)\defined J_\lambda\inv\circ \gr_\lambda(F).$$
\end{defn}

\subsection{Cohomology of Wakimoto Sheaves}
\begin{lem}\label{wakimoto:shiftlem}
	For any $\lambda\in \cham$ and $F\in\Dbc(\Fl{\I})$, there is a unique isomorphism of
	complexes
	\begin{align*}\label{wakimoto:shift}
		R\Gamma\left(\Fl{\I},F\star J_\lambda\right)\cong
		R\Gamma\left(\Fl{\I},F \right)\left[\left<\lambda,2\rho\right>\right]
	\end{align*}
	in $\Db(\Spec (k))$, such that for $F=\IC_e$ and $\lambda$ dominant, its pullback to the base point $\Fl{I,e}$
	        agrees with the isomorphism
		$$R\Gamma\left(\Fl{\I,\leq t^\lambda}, \costd_{t^\lambda}\right)\cong
		R\Gamma\left(\Fl{\I,t^\lambda},\underline{\Q_\ell}[\ell(t^\lambda)]\right),$$
		induced by the adjunction.
		Moreover, the isomorphims obtained for $\lambda\in\cham$ are additive in $\lambda$.
\end{lem}
\begin{proof}
	We first treat the case of a dominant coweight $\lambda\in\chamb$. In this case, by
	\cref{remark} we have an isomorphism $J_\lambda\cong\std_\lambda$, thus
	\begin{equation}\label{wakimoto:shift1}
	R\Gamma\left(\Fl{\I},F\star J_\lambda\right)\cong
		R\Gamma\left(\Fl{\I}\wttimes\Hk{\I}, pr_1^*F\otimes pr_2^*
		\left(\costd_\lambda\right)\right),
	\end{equation}
	where $(pr_1,pr_2)\Fl{\I}\wttimes\Hk{\I}\rightarrow\Fl{\I}\x\Hk{\I}$ is the map
	$(a,b)\mapsto (a\Lplus\I,\Lplus\I b\Lplus\I)$.
	Identifying
	$\Fl{\I}\tilde\times\Fl{\I}$ with $\Fl{\I}\times\Fl{\I}$ along the isomorphism $pr_1\times
	m$ identifies $pr_2$ with the map,  $r:\Fl{\I}\times\Fl{\I}\rightarrow
	\Hk{\I}$ given by
	$(a\Lplus\I,b\Lplus\I)\mapsto \Lplus\I a\inv b\Lplus\I$. Therefore, we may write the
	right hand side of \eqref{wakimoto:shift1} as:
		\[
	R\Gamma\left(\Fl{\I}\times\Fl{\I}, pr_1^*F\otimes r_*
	\left(\costd_\lambda\right)\right)\cong
	R\Gamma\left(\Fl{\I}, F\otimes pr_{1*}r_*\left(\costd_\lambda\right)\right) .\]
	Consequently, it will suffice to show $F\otimes
	pr_{1*}r_*\left(\costd_\lambda\right)\cong
	F\left[\left<\lambda,2\rho\right>\right]$.

	Let $\pi:\LG\rightarrow\Hk{\I}$ be the projection map, and
	$r'$ the pullback of $r$ by $\pi$.
	Using proper base change, it can be readily verified that
	the splitting of $\pi^*pr_{1*}r_*\Qell$,  induced by the section
	$g\mapsto(e,g)$ of $r'$, yields an isomorphism
	$$\pi^*pr_{1*}r_*\left(\costd_\lambda\right)\cong\underline{\Q_\ell}[\ell(\lambda)].$$
	Moreover, as $r$ and $pr_1$ are $\LG$-equivariant, so is
	$pr_{1*}r*\left(\costd_\lambda\right)$; showing it is also isomorphic to the shifted
	constant sheaf. Finally, since $\lambda$ is dominant,
	${\ell(\lambda)=\left<\lambda,2\rho\right>}$ \cite[Corollary 1.8]{TimoTwistedSchubert}, concluding the proof.

	The case of general $\lambda\in\cham$ follows by writing it as a difference of dominant
	ones and using the monoidality of Wakimoto sheaves (\cref{lem:wakimotomonoidality}).
	\end{proof}
		From this, we can deduce the following corollaries:
	\begin{cor}\label{cor:wakimotocohomology}
		For $M\in\Db(\Spec k)$ and $\lambda\in\cham$,
				$$R\Gamma\left(\Fl{\I},J_\lambda(M)\right)\cong M
				\left[{\left<\lambda,2\rho\right>} \right].$$
	\end{cor}
	\begin{proof}
		Follows from setting $F=\IC_0$ in \cref{wakimoto:shiftlem}.
	\end{proof}
	\begin{cor}\label{cohomology:wakimoto}
		For $F\in\PervI^{\cham}$,
		$$R\Gamma\left(\Fl{\I},F\right)\cong\bigoplus_{\lambda\in
		\cham}\Grad_\lambda(F)\left[{\left<\lambda,2\rho\right>}\right].$$
	\end{cor}
	\begin{proof}
		We may assume that $F$ is supported on a connected component of $\Fl{\I}$. In
		this case, for $\lambda\in \cham$ with $\Grad_\lambda(F)$ non-zero,
		${\left<\lambda,2\rho\right>}$ are of the same parity.
		Therefore, the claim follows by induction on the number of such $\lambda$, with
		base case being \cref{cohomology:wakimoto}.
	\end{proof}

A more functorial framework for the cohomologies of the Wakimoto sheaves is
provided by \textit{constant term functors}:
\begin{equation}\label{eq:constantterms}
\operatorname{CT}_{B}:\Dbc(\Hk{\I})\rightarrow \Dbc(\Hk{\mathcal{T}})
\end{equation}
where $\mathcal{T}$ is the unique parahoric model of $T$, \cite[Definition 6.4]{richarzTestFunction}.
The underlying topological space of $\Hk{\mathcal{T}}$ is a disjoint union of points
indexed by $\cham$ \cite[Equation 3.15]{modularRamified}.

We will only need the existence, and well-known properties of these
functors. We refer the
interested reader to \cite[Section 6]{richarzTestFunction}.

\begin{lem}\label{lem:constanttermswakimotograding}
	Let $B^-$ denote the opposite Borel subgroup to $B\subset G$. For any
	$F\in\Dbc(\Hk{\mathcal{I}})$, and $\lambda\in\cham$; there is a canonical isomorphism
	$$\operatorname{CT}_{B^-}(J_\lambda\star F)_{\lambda}\cong
	\operatorname{CT}_{B^-}(F)_{0}\left[\langle \lambda,2\rho\rangle\right]$$
	between the stalks.
\end{lem}
\begin{proof}
	The proof given in \cite[Proposition 3.23]{joaoAB} in the mixed characteristic setting
	works verbatim in our equal characteristic setting as well.
\end{proof}
\subsection{Associated graded and the restriction}
The associated graded functors associated to each $\lambda\in\cham$, assemble into a single functor
$$\Grad\defined\bigoplus_{\lambda\in\cham}
\Grad_\lambda:\Rep(\check T^I)\rightarrow \Perv(\Hk{\I}),$$
where we use the identification
$$\Rep(\check T^I)\cong \Vect_{\Qell}^{\cham}$$
with the category of $\cham$-graded $\Qell$-vector spaces $\Vect_{\Qell}^{\cham}$, induced by
decomposition into weight spaces.

\begin{cor}\label{cor:gradedconstant}
	The functors
	$\operatorname{CT}_{B^-}$ and $\Grad$
	$$\Perv(\Hk{I})^{\chamb}\rightarrow \Rep(\check T^I)$$
	are naturally isomorphic.
\end{cor}
\begin{proof}
	Follows from a straightforward induction on the length of Wakimoto filtrations, using
	\cref{wakimoto:shiftlem} and \cref{lem:constanttermswakimotograding}.
\end{proof}
The functor $\Grad$ is compatible with the monoidal structures:
\begin{lem}\label{lemma:monoidalgraded}
	For any $F,G\in\Perv(\Hk{\I})^{\chamb}$, there is a canonical isomorphism
	$$\Grad(F)\otimes\Grad(G)\cong\Grad(F\star G).$$
\end{lem}
\begin{proof}
	Follows from the similar compatibility of a single Wakimoto sheaf
	(\cref{lem:wakimotomonoidality}); by induction on the number of non-trivial graded
	pieces $\Grad_\lambda(F)$.
\end{proof}
We can neatly sum this section up with the following theorem:
\begin{thrm}\label{theorem:torusfunctor}
	The Wakimoto functors $J_\lambda$, $\lambda\in\cham$, assemble to a faithful, monoidal functor
	$$\mathbb{J}\defined \bigoplus_{\lambda\in\cham}J_\lambda:\operatorname{Rep}({\check T}^{I})\rightarrow
	\Dbc(\Hk{\I}).$$
	Moreover, the essential image of $\mathbb{J}$ lands in $\PervI$, and is closed under extensions.
\end{thrm}
\begin{proof}
	Since ${\check T}^{I}$ is a multiplicative group scheme, $\operatorname{Rep}({\check
	T}^{I})$ is a coproduct of categories of $\Qell$-vector spaces indexed by
	$X^*(\check T^I)\cong \cham$. The result is then immediate from
	\cref{wakimoto:essim} and \cref{lem:wakimotomonoidality}.
\end{proof}

\subsection{Comparison with the central functor}
\begin{thrm}\label{theorem:filtration}
	The essential image of the central functor
	$$\Zent:\Rep(\check G)\rightarrow \PervI,$$ is contained in the full subcategory
	consisting of Wakimoto
	filtered objects. Moreover, for every $V\in\operatorname{Rep}(\check G)$, the graded pieces of the filtration on
	$\operatorname{Z}(V)$ are canonically isomorphic to the graded
	pieces of the filtration on
	$\mathbb{J}\circ \res^{\check G}_{{\check
	T}^I}(V)$, where $\res^{\check G}_{{\check
	T}^I}(\cdot)$ denotes the restriction of representations along the inclusion $\check
	T^I\rightarrow \check G$.
\end{thrm}
\begin{proof}
	In light of \cref{theorem:gaitsgoryzhu}, the first assertion follows from
	\cref{wakimoto:filtration}.

	To prove the second assertion, it suffices to observe that there is a canonical
	isomorphism
	$$\Grad(\Zent(V))\cong\res^{\check G}_{{\check T}^I}(V).$$
	Let $\G$ be the parahoric group scheme associated to a special vertex, which is contained in
	the alcove corresponding to $\I$, which always exists by
	\cite[Proposition 1.3.45]{kalethaPrasad}.
			As the natural quotient morphism $p:\GrBD{\I}\rightarrow \GrBD{\G}$ is ind-proper
		by \cite[Lemma 4.9]{richarzTestFunction}; pushforward along $p$ commutes with
		nearby cycles. On the other hand, the fiber functor for the absolute Satake
		equivalence and the ramified Satake equivalence are related to each other via
		pre-composition with nearby cycles \cite[Lemma 8.8]{modularRamified}, which
		under the respective Satake equivalences
		corresponds to the restriction functor by \cref{lemma:nearbyrest}. Therefore,
		using the comparison with the constant term functors \cref{cor:gradedconstant}, the claim
		follows from the identification \cite[Equation
		6.3]{modularRamified} of the ramified Satake fiber functor.
\end{proof}
\subsection{Unipotency of the monodromy}
The monodoromy acts trivially on the Wakimoto filtration:
\begin{lem}\label{lemma:wakimotogradtriv}
	For every $V\in \Rep(\check G)$ and $\lambda\in\cham$,
	the $I$-action induced by the monodromy on $\Grad_\lambda\Zent(V)$ is trivial.
\end{lem}
\begin{proof}
	By \cref{cor:wakimotocohomology}, it suffices to check that the induced action on cohomology is trivial.
	This follows from the existence of the ind-proper map $p:\GrBD{\I}\rightarrow\GrBD{\G}$
	as in the proof of \cref{theorem:filtration}, combined with the triviality of the
	monodromy for $\Psi_\G$ \cite[Proposition 3.10]{richarzRamified}.
\end{proof}
\begin{cor}
	For every $V\in \Rep(\check G)$, the $I$-action induced by the monodromy on $\Zent(V)$
	is unipotent.
\end{cor}
\begin{proof}
	Follows from \cref{lemma:wakimotogradtriv} by induction on the length, after picking a linear
	refinement, of the Wakimoto filtration.
\end{proof}
\subsection{Highest weight arrows}\label{subsection:highestweight}
We will study the collection of \textit{highest weight arrows}, which are geometric counterparts
to the projection to the highest weight line of a representation, cf.
\cite[Section 4.6.3]{centralRiche}. Recall that, the unique irreducible $\check{G}$-representation of  highest weight
$\lambda\in\chamb$
is given by $\coweyl(\lambda):=\Ind_{\check B^-}^{\check G}(\lambda)$. Here, $\lambda$ is viewed as
a 1-dimensional representation of the opposite Borel $\check B^-$, through the projection
$\check B^-\rightarrow \check
T$.
\begin{lem}\label{lemma:hwtadjunction}
There is a canonical isomorphism
	$$j_{t^{\bar\lambda}
	*}j^*_{t^{\bar\lambda}}\Zent(\coweyl(\lambda))\cong\gr_{\bar\lambda}\Zent(\coweyl(\lambda)).$$
\end{lem}
\begin{proof}
	It suffices to prove that there exist a canonical isomorphism
	$$j_{t^{\bar\lambda}}^*\Zent(\coweyl(\lambda))\cong
	\underline{\Grad_{\bar\lambda}\Zent(\coweyl(\lambda))}\left[\left<\lambda,2\rho\right>\right].$$
	Let $U$ be the unipotent radical of the Borel $B$.
	By the argument of \cite[Lemma 2.12]{joaoAB}, intersection of the support of
	$\Zent(\coweyl(\lambda))$
	with the $\L U$-orbit of $t^{\bar\lambda}$ in $\Fl{\I}$ is precisely
	$\Fl{\I,t^{\bar\lambda}}$. By the computation of constant terms
	through such orbits, we have a canonical isomorphism
	$${\operatorname{CT}_{B^-}}\left(\Zent(\coweyl(\lambda))\right)_{\bar\lambda}\cong
	j_{t^{\bar\lambda}}^*\Zent(\coweyl(\lambda)).$$
	Therefore, we conclude by the agreement of Wakimoto gradeds and constant terms
	\cref{lem:constanttermswakimotograding}.
\end{proof}
Using the identification from \cref{lemma:hwtadjunction}, the unit of the adjunction provides a map
$$\ff_\lambda:\Zent(\coweyl(\lambda))\rightarrow \gr_{\bar\lambda}\Zent(\coweyl(\lambda)),$$
called the \textit{highest weight arrow} associated to $\lambda$. The highest weight arrows are
compatible with convolution in the following sense:
\begin{prop}\label{prop:highestweightmonoidal}
	For $\lambda,\mu\in\splitchamb$, the diagram
	$$\begin{tikzcd}
		\Zent(\coweyl(\lambda))\star\Zent(\coweyl(\mu))\ar[r,"\sim"]\ar[dd,"\ff_{\lambda}\star\ff_{\mu}"]
		& \Zent(\coweyl(\lambda)\otimes \coweyl(\mu))\ar[d] & \Zent(\coweyl(\mu))\star
		\Zent(\coweyl(\lambda))\ar[l,"\sim"']\ar[dd,"\ff_{\mu}\star\ff_{\lambda}"]\\
		&\Zent(\coweyl(\lambda+\mu))\ar[d,"\ff_{\lambda+\mu}"]&\\
		\gr_{\bar\lambda}\Zent(\coweyl(\lambda))\star\gr_{\bar\mu}\Zent(\coweyl(\mu))\ar[r,"\sim"] &
		\gr_{\overline{\lambda+\mu}}\Zent(\coweyl(\lambda+\mu)) &
		\gr_{\bar\mu}\Zent(\coweyl(\mu))\star \gr_{\bar\lambda}\Zent(\coweyl(\lambda)\ar[l,"\sim"'],
	\end{tikzcd}$$
	with the evident unlabelled maps, commutes.
\end{prop}
\begin{proof}
	This follows from the identification $\Grad\circ\Zent\cong \Res^{\check G}_{\check T^I}$
	of \cref{theorem:filtration}, together with the fact that the unit of the adjunction in
	the definition of highest weight arrows is symmetric monoidal.
\end{proof}
As the monodromy acts trivially on the Wakimoto filtration, it interacts trivially with the
highest weight arrows:
\begin{lem}\label{lemma:highestweightmonodtriv}
	For a coweight $\lambda\in\splitchamb$ we have the equality
	$$\ff_\lambda\circ \bn_{\coweyl(\lambda)}=0.$$
\end{lem}
\begin{proof}
	By construction, $\ff_\lambda$ is the projection on to the $\bar\lambda$-graded part of
	the Wakimoto filtration on $\Zent(\coweyl(\lambda))$. Since the monodromy acts trivially on the
	Wakimoto filtration by \cref{lemma:wakimotogradtriv}, we conclude.
\end{proof}

\section{Decategorification to the Hecke algebra and mixed sheaves}\label{decategorification}
In this section, we first recall the Iwahori-Hecke algebra
associated to the tuple $(G,\I)$ of a reductive group $G$ over $K$, and an Iwahori model thereof. This algebra depends only on the quasi-Coxeter structure of the
associated
Iwahori-Weyl group, therefore many of the basic properties carry over from the setting of split
reductive groups.

Next, we discuss extension of the results of \cref{sec:wakimoto} to the setting of mixed sheaves
in the sense of \cite{BBDG}. This leads to a combinatorial description of the weight filtration
on such mixed sheaves, which will be an important ingredient in the proof of
\cref{main:tilting}.

We fix a tuple $(G,B,S,\I)$ as in \cref{iwahoriweylgroup}, and denote by
$\ogm$ the ring of Laurent polynomials in the variable $\mathbf{v}$.
\subsection{Iwahori-Hecke algebra}
\begin{defn}
	The \textit{Iwahori-Hecke algebra} $\mathcal{H}$ associated to $(G,\I)$,
	is the $\ogm$-algebra with basis $H_w$, $w\in W$ subject to the relations:
\begin{align*}
	&(H_s+v)(H_s-v\inv)\quad \text{if } \ell(s)=1,\\
	&H_wH_y=H_{wy}\quad \text{if } \ell(wy)=\ell(w)+\ell(y).
\end{align*}
\end{defn}
The Grothendieck group $K_0\left(\Dbc(\Hk{\I})\right)$ also admits a basis indexed by
$W$. In fact, as the perverse t-structure on $\Dbc(\Hk{\I})$ is bounded,
the classes of intersection sheaves $IC_w$ form a basis. Moreover, this shows that the endomorphism induced by Verdier
duality on $K_0\left(\Dbc(\Hk{\I})\right)$ is the identity; and classes of $\nabla_w$ and
$\Delta_w$ agree. Regarding $\mathcal{H}$ as a $\Z$-algebra by setting $\mathbf{v}=1$, we deduce:
\begin{prop}\label{prop:naivecategorification}
The assignment
	\begin{gather}\label{decategorification:iso}
		K_0(\Dbc(\Hk{\I}))\rightarrow \Z \otimes_{\ogm}\mathcal{H},\\
	[\mathcal{F}]\mapsto \sum_{w\in
	W}(-1)^{\ell(w)}\chi(\Fl{G,w},j^!_w \mathcal{F})\cdot H_w;\nonumber
\end{gather}
is a $\Z$-algebra isomorphism, with the algebra structure induced by the convolution product on
the left hand side. Under this isomorphism $[\costd_w]$ maps to $(-1)^{\ell(w)}\cdot H_w$.
\end{prop}
\begin{proof}
This is \cite[Lemma 5.2.1]{centralRiche}.
\end{proof}

\subsection{Bernstein translation elements}
The isomorphism \eqref{decategorification:iso} leads to the aforementioned connection between Bernstein translation elements and Wakimoto
sheaves.
\begin{defn}
	Given $\mu\in \cham$, the associated translation element is defined to be:
	\begin{equation}\label{def:bernsteintranslation}
		\theta_\mu:=(-1)^{\left<\mu,2\rho\right>}H_{t^\lambda}\cdot H_{t^{\lambda-\mu}}\inv,
	\end{equation}
for a choice of $\lambda\in\chamb$ such that $\mu\trianglelefteq\lambda$ in the dominance order \eqref{order}.
	Independence of the definition from the choice of such $\lambda$ can be proven via arguments similar to those in \cref{noncanon}.
\end{defn}
\begin{rmk}
	Under the isomorphism \eqref{decategorification:iso}, the class of the
	Wakimoto sheaf $[J_\lambda]$ corresponds to $1\otimes \theta_\lambda$.
        Mirroring the monoidality of Wakimoto sheaves,
	Bernstein translation elements carry the addition in $\cham$, to multiplication in
	$\mathcal{H}$, the proof of which is also entirely analogous to that of
	\cref{lem:wakimotomonoidality}. Consequently, they span a large commutative subalgebra of $\mathcal{H}$ which is closely
	related to the center (cf. \cref{theorem:torusfunctor}).
\end{rmk}
	\subsection{Mixed sheaves}
Let $\mathbb{F}_q$ be a finite field of
characteristic $p$, and assume that $k=\overline{\mathbb{F}}_q$.
In order to categorify the Iwahori-Hecke algebra $\mathcal{H}$ in its entirety, as opposed to
the specialization at $\mathbf{v}=1$, we need to remember more
structure, namely, those of weights \`a la \cite{weil2}.

We will denote by $\Dmix(\Spec(\mathbb{F}_q))$ the category
of mixed $\Qell-$sheaves on $\Spec(\mathbb{F}_q)$ (see \cite[\textsection 5.1.5]{BBDG}).
Moreover, we fix once and for all a square root of the Tate twist $\Q_\ell(\frac{1}{2})$ in
$\Dmix(\Spec(\mathbb{F}_q))$. For all $\mathbb{F}_q$-schemes $X$, the choice of this square root equips
$K_0\left(\Dmix(X)\right)$ with a $\ogm-$algebra
structure by sending $\mathbf{v}$ to the class of $\underline{\Q_\ell}(-\frac{1}{2})$.

Without loss of generality, we may assume that the tuple $(G,B,S,\I)$ admits a model
$(G_\circ,B_\circ,S_\circ,\I_\circ)$ over
$\mathbb{F}_q(\!(t)\!)$ (\cite[Corollary
B.2]{richarzRamified}). Possibly enlarging $\mathbb{F}_q$, we may also assume $G_\circ$ is residually split.
Therefore, the corresponding flag variety
$\Fl{\I_\circ}$ is defined over $\mathbb{F}_q$.
We can then consider the categories of mixed $\Qell-$sheaves:
$$\Dmix(\Fl{\I_\circ}),\quad \Dmix(\Hk{\I_\circ}).$$

The residual splitness implies by
\cite[Lemma 1.6]{richarzIwahoriWeyl} that
the Schubert cells in
$\Fl{\I_\circ}$ are already defined over $\mathbb{F}_q$. Therefore, results of
\cref{sec:loopgroups} go through in this setting as well. We then have the mixed versions of standard, costandard and
simple objects:
\begin{defn}
Given $w\in W$, denote by
$s_{w}:\left[\Lplus\I_\circ\backslash\Fl{\I_\circ,w}\right]_{\text{\'et}}\rightarrow
\Hk{\I_\circ}$, the inclusion of the
point corresponding to the Schubert cell.
The standard, resp. costandard, functor attached to $w$ is
$$\stdmix_w(M):=s_{w!}(\underline M)(\tfrac{\ell(w)}{2})[\ell(w)],\quad
\costdmix_w(M):=s_{w*}(\underline M)(\tfrac{\ell(w)}{2})[\ell(w)].$$
\end{defn}

As before, $\ICmix_w$ is defined to be the image of the natural map
$\stdmix_w\rightarrow \costdmix_w$, and is a pure perverse sheaf of weight $0$.
By repeating \cref{subsec:definition}, we may extend the notion of Wakimoto functor
$\Jmix_\lambda$ and Wakimoto filtered perverse sheaves to the mixed setting.
Equip $R:=K_0\left(\Dmix(\Spec(\mathbb{F}_q)\right)$ with the ring structure induced by the usual
tensor product. Here is the promised decategorification into the Iwahori-Hecke algebra:
\begin{prop}\label{prop:categorification}
	The map
$$R\otimes_{\ogm}\mathcal{H}\rightarrow K_0\left(\Dmix(\Hk{\I_\circ})\right)$$
	specified by sending $(-1)^{\ell(w)}\cdot H_w$ to the class of $\costd_w$, is an $R$-algebra
	isomorphism with respect to the algebra structure induced by the convolution product on
	the right hand side.
\end{prop}
\begin{proof}
	Denote by $h:\Fl{\I_\circ}\rightarrow \Hk{\I_\circ}$ the quotient map. Then the desired
	inverse is given by
	$$[F]\mapsto\sum_{w\in W}\left[\mathfrak{j}_w^!h^*F(\tfrac{\ell(w)}{2})\right]\cdot H_w,$$
	where $\mathfrak{j}_w:\Spec(\mathbb{F}_q)\rightarrow \Fl{\I_\circ,w}$ is the inclusion of a
	rational point. For details, see \cite[Lemma 5.3.2]{centralRiche}.
\end{proof}
In particular, the preimages $\underline{H_w}$ of classes $[\ICmix_w]$, provide a
$\ogm$ basis of $\mathcal{H}$, which is known as \textit{Kazdhan-Lusztig basis}.

The map $\ogm\rightarrow R$ determined by the
choice of $\Qell(\tfrac{1}{2})$, admits a section. Indeed, the map $R\rightarrow
\ogm$
induced by
base changing the associated weight graded pieces of a mixed perverse sheaf in
$\Dmix(\Spec(\mathbb{F}_q))$ to the algebraic closure, can be readily verified to be such a section.  Therefore \cref{prop:categorification} supplies a surjective
$\ogm-$algebra morphism
\begin{equation}\label{eq:heckemorphism}
	K_0\left(\Dmix(\Hk{\I_\circ})\right)\rightarrow\mathcal{H}.
\end{equation}
	More explicitly, the morphism (\ref{eq:heckemorphism}),
	sends the class of a mixed perverse sheaf $F$ to:
	\begin{equation}\label{eq:weightfilt}
	\sum_{w\in
		W}(-1)^{\ell(w)}\cdot\left(\sum_{i\in\Z}\mathbf{v}^i\cdot\left[\grW_i\left(F\right):\IC_w\right]\cdot\underline{H_w}\right).
	\end{equation}

For $V\in\Rep(\check G)$, denote with
$F_V\in\Pervmix(\Gr_{G_\circ})$ the unique pure perverse sheaf of weight 0
whose base change to the algebraic closure corresponds to $V$ under the absolute Satake
equivalence. Let $\Zent^\mix(V)$ be the nearby cycles applied to $F_V$.
\begin{prop}\label{mixed:central}
	For  $V\in\Rep(\check G)$, denote by $\left[\Zent^{\mix}(V)\right]$ the image of the
	class of $\Zent^{\mix}(V)$, under the algebra morphism \eqref{eq:heckemorphism}. Then we have the
	equality
	$$\left[\Zent^{\mix}(V)\right]=\sum_{\bar\mu\in\cham}\dim\res(V)_{\bar\mu}\cdot\theta_{\bar\mu},$$
	where $\res(V)_{\bar\mu}$ denotes the $\bar\mu$ weight space of
	$V$ restricted to $\Rep(\check G^I)$.
\end{prop}
\begin{proof}
	By the mixed analogue of \cref{theorem:filtration},
	$\Zent^\mix(V)$ is Wakimoto filtered; and the equality
	$$\left[\Zent^\mix(V)\right]=\sum_{\bar\mu\in\cham}\dim\res(V)_{\bar\mu}\cdot
	\left[J^\mix_{\bar\mu}\right]$$
	holds in $K_0\left(\Dmix(\Hk{\I_\circ})\right)$.
	On the other hand, we have $\left[J^\mix_{\bar\mu}\right]=\theta_{\bar\mu}$ in
	$\mathcal{H}$,
	as is clear from \cref{prop:categorification}, concluding the proof.
	\end{proof}
\subsection{An algebra morphism for weight multiplicities}
	We can use \cref{mixed:central} to get a combinatorial handle on the
	weight filtration on mixed central sheaves. The following algebra morphism is useful for
	extracting multiplicities:
\begin{gather}\label{eq:multiplicitymorph}
	m:\mathcal{H}\rightarrow \ogm,\\
	m(H_w):=(-\mathbf{v})^{\ell(w)}.\nonumber
\end{gather}
From the definition of the Bernstein
translation elements $\theta_{\bar\mu}$ and \cref{comb} one can easily calculate:
\begin{equation}\label{m:theta}
	m(\theta_{\bar\mu})=\mathbf{v}^{\left<\bar\mu,2\rho\right>}.
\end{equation}
It is similarly easy to observe:
\begin{equation}\label{equation2}
	m(\underline{H_w})=\begin{cases}
			1 & \text{if}\: \ell(w)=0;\\
			0 & \text{otherwise}.
			   \end{cases}
\end{equation}

\section{Iwahori-Whittaker Sheaves}\label{section:iwahoriwhittaker}
We fix a tuple $(G,B,S)$ as in \cref{iwahoriweylgroup}.
\subsection{Iwahori-Whittaker datum}\label{section:iwahoriwhittakerdatum}
Let $\f\in\mathcal{A}(G,S,K)$ be a special vertex, which we will use as the origin. There is a
unique alcove $\a$, which contains
$\f$ in its closure and lies in the Weyl chamber determined by $B$.

Let $\Par$, resp. $\I$, be the parahoric group schemes associated to $\f$, resp. $\a$. The containment
$\f\subset \overline\a$ induces a map $\I\rightarrow\Par$, from which we obtain a map
$\I_k\rightarrow \Par_{k}^\red$ in the special fiber. Its image $\mathsf{b}^+$
is a Borel subgroup of $\Par_{k}^\red$ (see \cite[Theorem 8.4.19]{kalethaPrasad}). Let $\mathsf{b}^-$,
$\mathsf{u}^-$ be its opposite Borel, and the unipotent radical thereof.

\begin{defn}
	A linear functional $\widetilde{\chi}:\mathsf{u}^-\rightarrow \G$ is called \textit{generic}, if its restriction to each simple root subspace is non-trivial.
\end{defn}
The following simple lemma is well known:
\begin{lem}\label{lemma:genericcharacter}
	Given $G,B,S,$ and $\f$ as above, there exist a generic linear functional $\widetilde{\chi}$ of
	$\mathsf{u}^-$.
\end{lem}
\begin{proof}
	The linear functional obtained as the composition
	$$\mathsf{u}^-\rightarrow \mathsf{u}^-/\left[\mathsf{u}^-,\mathsf{u}^-\right]\isomto
	\prod_{\text{$\alpha$ simple root}}\mathsf{u}^-_\alpha\isomto
	\prod_{\text{$\alpha$ simple root}}\mathbb{G}_{a,k}\rightarrow\mathbb{G}_{a,k},$$
	where the last arrow is given by the sum, is generic.
\end{proof}

\begin{defn}
Let $a:\mathbb{G}_{a,k}\rightarrow
\mathbb{G}_{a,k}$ be the Artin-Schreier map $x\mapsto x^p-x$.
The choice of a primitive $p$-th root unity in $\Qell$ determines a rank 1 local system, which is a direct
	summand of $a_*\underline\Qell$.
	Such a local system $L_{\mathrm{AS}}$ is called an \textit{Artin-Schreier local
	system}.
\end{defn}
Artin-Schreier local systems are \textit{character sheaves} in the sense of \cite{yunEndoscopy}. In
particular, they satisfy:
$$m^*L_{\mathrm{AS}}\cong L_{\mathrm{AS}}\boxtimes
L_{\mathrm{AS}},\quad e^*L_{\mathrm{AS}}\cong\Qell$$
where $m$, $e$ denote the multiplication and the unit map respectively.
Moreover, they have vanishing cohomology:
$$R\Gamma(L_{\mathrm{AS}},\mathbb{G}_{a,k})=0.$$
Finally, by putting it all together we define:
\begin{defn}
	Let $G,B,S,$ and $\f$ be as above.
	The pair $(\widetilde\chi,L_{\mathrm{AS}})$ of a generic linear functional
	$\widetilde\chi$ of $\mathsf{u}^-$,
	and an Artin-Schreier local system $L_\AS$ is called an \textit{Iwahori-Whittaker datum}.
\end{defn}
	Let $\Iop$ be the Iwahori group scheme associated to the alcove $\a\op$, opposite to
	$\a$ with respect to $B$. Given an Iwahori-Whittaker datum $(\widetilde\chi,L_{\mathrm{AS}})$, we get an induced character $\chi$ of the pro-unipotent radical $\Lplus
	\Iop_{u}:=\ker(\Lplus\Iop\rightarrow \I^\red_k)$ by postcomposition.
	We will denote by $\Las$, the pullback $\chi^*L_{\mathrm{AS}}\in\Dbc(\Lplus\Iop_{u})$.
	For the rest of the document, we will choose an arbitrary but fixed Iwahori-Whittaker datum
	for each tuple $(G,B,S,\f)$; as we may in light of \cref{lemma:genericcharacter}.
	\subsection{Iwahori-Whittaker sheaves}
\begin{defn}\label{def:iwahoriwhittaker}
	The category of \textit{Iwahori-Whittaker sheaves}
	is defined to be (the homotopy category of) the limit in $\iCat$ of the
	twisted cobar resolution,
	$$\DIW\defined\lim \operatorname{CoBar}(\Lplus\Iop_{u},\Fl{\I},\Las).$$
	Recall that, the twisted cobar resolution is the following cosimplicial diagram:
	\begin{align*}
				\left(\xymatrix{
	\Dbc(\Fl{\I})
		\ar@<0.5ex>[r]^{a^*\phantom{hhhh}}
		\ar@<-0.5ex>[r]_{\Las\boxtimes-\phantom{hhhhhh}}
		&
	\Dbc(\Lplus\Iop_{u} \x \Fl{\I})
	\ar@<-1ex>[r]%^{p_{23}}
	\ar[r]%^{m \x \id_X}
	\ar@<1ex>[r]%^{\id_G \x a}
	&
	\Dbc(\Lplus\Iop_{u} \x \Lplus\Iop_{u} \x \Fl{\I})
	 \ar@<-1.5ex>[r] \ar@<-0.5ex>[r] \ar@<0.5ex>[r] \ar@<1.5ex>[r]
		&
	\dots
	}\right)
	\end{align*}
	in which each pullback by a projection morphism in the usual cobar resolution is twisted via
	the corresponding box product with $\Las$.
\end{defn}
Next, we will observe that Iwahori-Whittaker equivariance is a condition, rather than a datum.
This observation hinges on the fact that $\Lplus\Iop_u$, as a pro-algebraic group, is an infinite succesive extension of
vector groups. More precisely, following \cite{richarzModuliShtuka}, we will call a pro-algebraic group $G$ \textit{split
pro-unipotent}, if it admits a presentation $\lim_{i\in\N} G_i$ such that for all $i\in \N$,
$\ker(G_{i+1}\rightarrow G_i)$ is a vector group.
\begin{lem}
The pro-algebraic group $\Lplus\Iop_u$ is split
	pro-unipotent.
\end{lem}
\begin{proof}
	Denote by $U$, the kernel of the map $\Lplus\Iop\rightarrow
	\Iop_k$
	induced by setting $t=0$. By \cite[Proposition A.4.9]{richarzModuliShtuka}, $U$ is split
	pro-unipotent. By definition,
	$\Lplus\Iop_u$ is the extension of the unipotent radical $\operatorname{R}_u(\Iop_k)$,
	by $U$. This conludes the proof after the straightforward verificiation that
	extension of a unipotent group,  by a split pro-unipotent one is
	again split pro-unipotent.
	\end{proof}
\begin{lem}\label{lem:conditionnotdatum}
	The natural pullback functor $p^*:\DIW\rightarrow \Dbc(\Fl{\I})$ is fully faithful.
	Moreover, the essential image is given by objects $F\in\Dbc(\Fl{\I})$ such that
	$$a^*F\cong \Las\boxtimes F,$$
	where $a:\Lplus\Iop_u\times\Fl{\I}\rightarrow \Fl{\I}$ denotes the action map.
\end{lem}
\begin{proof}
	To prove the fully faithfulness, it suffices to prove pullbacks for torsors of split pro-unipotent groups are fully
	faithful. This is proven in \cite[Proposition 2.2.11]{richarzModuliShtuka} for the
	categories of motivic sheaves, $\DM(-)$. However, the proof only depends on abstract
	properties of the six functor formalism such as $\mathbb{A}^1$-invariance and relative
	purity; which are shared with the categories of \'etale $\Qell$-sheaves we use.
	Therefore, their proof applies in our setting aswell.

	In light of the fully faithfullness of $p^*$, the identification of the essential image
	is immediate from the definition of $\DIW$.
	\end{proof}
We will call sheaves in the essential image of $p^*$, \textit{Iwahori-Whittaker equivariant
sheaves}.
\subsection{Perverse t-structure on Iwahori-Whittaker sheaves}
Through its inclusion in $\Dbc(\Fl{\I})$,
the category $\DIW$ inherits a t-structure;
which will also be dubbed the perverse t-structure. To prove this, it suffices to
observe that Iwahori-Whittaker equivariance is stable under perverse truncations.
Before moving on to the proof, we outline an alternate characterization of
Iwahori-Whittaker sheaves along the lines of \cref{subsec:sheaves}.

Let $\colim_i X_i$ be a presentation of $\Fl{\I}$ as a strict ind-scheme, in which each $X_i$ is stable under the action of
$\Lplus\Iop_u$. For each $X_i$,
we can define the category of Iwahori-Whittaker equivariant sheaves $\DIWX(X_i)$ by replacing
$\Fl{\I}$ with $X_i$ in \cref{def:iwahoriwhittaker}, and observe the
analogue of \cref{lem:conditionnotdatum} holds. In particular, $\DIWX(X_i)$ is equivalent to the
full subcategory of sheaves in $\DIW$ with support in $X_i$.
As each sheaf in $\Dbc(\Fl{\I})$ is also supported in some $X_i$ \eqref{eq:sheafdefinition}, this implies
\begin{equation}\label{eq:xwhittaker}
\DIW\cong\colim_i\DIWX(X_i).
\end{equation}
\begin{prop}
	The full subcategory $p^*:\DIW\rightarrow \Dbc(\Fl{\I})$ is stable under
	perverse truncations $^p\tau^{\leq 0}$ and $^p\tau^{\geq 0}$, therefore inherits the perverse $t$-structure.
\end{prop}
\begin{proof}
		By \eqref{eq:xwhittaker}, it suffices to check this for the inclusion
	$$p^*:\DIWX(\Fl{\I,\leq w})\rightarrow \Dbc(\Fl{\I,\leq w})$$
	for every $w\in W$.

	Before proceeding, we need a quick lemma.
	Denote by $n_w \geq 0$ a natural number for which the action of
	$\Lplus\Iop_u$ on $\Fl{\I,\leq w}$ factors through
	$\Lplus_{n_w}\Iop_u:=\Lplus_{n_w}\Iop\cap\Lplus\Iop_u$.
	\begin{lem}\label{lem:objectwiseav}
		Denote by $m:\Lplus_{n_w}\Iop_u\times\Fl{\I,\leq w}\rightarrow \Fl{\I,\leq w}$ the
		multiplication. The functor
	\begin{gather*}
		\mathrm{av^!_w}:\Dbc(\Fl{\I,\leq w})\rightarrow \Dbc(\Fl{\I,\leq w}),\\
		F\mapsto
		m_!\left(\left.F\boxtimes\Las\right|_{\Lplus_{n_w}\Iop_u}\right)[\dim(\Lplus_{n_w}\Iop_u)]
	\end{gather*}
	satisfies the following properties:
	\begin{itemize}
		\item[(1)] The essential image of $\mathrm{av_w}$ lands in $\DIWX(\Fl{\I,\leq w})$;
		\item[(2)] The functor $\mathrm{av^!_w}$ is perverse left t-exact;
		\item[(3)] There is an isomorphism $\mathrm{av^!_w}(F)\cong F$ if and only if $F$ is Iwahori-Whittaker equivariant.
	\end{itemize}
	\end{lem}
	\begin{proof}
		Property (1) follows using projection formula together with the multiplicative structure
		$m^*\Las\cong \Las\boxtimes\Las$ of the Artin-Schreier local system. Property
		(2) follows from the fact that $m$ is an affine morphism. Finally, property (3)
		follows by replacing $\left.F\boxtimes\Las\right|_{\Lplus_{n_w}\Iop_u}$ with
		$pr^*F$ using the equivariance.
	\end{proof}
	Given an Iwahori-Whittaker sheaf $F$, \cref{lem:objectwiseav} provides a chain of isomorphisms
	$${}^p\tau^{\geq 0}(F){\hspace{0.1cm}\stackrel{\text{\tiny \rm (3)}}{\cong}\hspace{0.1cm}} {}^p\tau^{\geq 0}\left(\mathrm{av^!_w}(F)\right){\hspace{0.1cm}\stackrel{\text{\tiny \rm (2)}}{\cong}\hspace{0.1cm}}
	\mathrm{av^!_w}\left({}^p\tau^{\geq 0}(F)\right).$$
	Appealing once again to \cref{lem:objectwiseav}.(3), this implies ${}^p\tau^{\geq 0}(F)$ is
	Iwahori-Whittaker equivariant.

	Finally, stability under ${}^p\tau^{\leq 0}$ can be proven from the stability under
	${}^p\tau^{\geq 0}$ via a straightforward induction argument, using that the perverse
	t-structure is bounded.
	\end{proof}
\subsection{Simple Iwahori-Whittaker perverse sheaves}
	Let $\PIW$ be the heart of the perverse $t$-structure inherited by $\DIW$.
	We proceed by classifying the simple objects in $\PIW$. Note that, $\Lplus\Iop_u$-orbits
	in $\Fl{\I}$ coincide with $\Lplus\Iop$-orbits, which are also enumareted by $W$,
	and their dimensions are described by the length function.
	For any $w\in W$, denote by $w_s$ the shortest length element in the set
	$W_{\mathsf{fin}}\cdot w$ of its left finite Weyl group conjugates. Consider $j_{w_s}\op:\Fl{w_s}\op\rightarrow \Fl{\I}$, the inclusion of the
	$\Lplus\Iop$-orbit associated to
	$w_s$.
	In the next proposition, we prove such orbits are the only ones which support a
	non-trivial Iwahori-Whittaker equivariant local system.
\begin{prop}\label{orbit:prop}
	The $\Lplus \Iop_u$-orbit associated to $w\in W$, $\Fl{\I,w}\op$, admits a non-trivial Iwahori-Whittaker
	local system if and only if, $w$ is of minimal length
	in  $\Wfin\cdot w$. Moreover, in this case, there is a unique non-trivial simple local system up to isomorphism.
\end{prop}
\begin{proof}
	First we prove a few lemmas.
	\begin{lem}
	Let $\operatorname{Stab}(w)\subset\Lplus\Iop(k)$ be the stabilizer of the $k$-point
	$\dot{w}\in \Fl{\I}$. The orbit $\Fl{\I,w}\op$ admits a
	non-trivial Iwahori-Whittaker local system if and only if
	$\ker(\chi)\supset\operatorname{Stab}(w)$.
	\end{lem}
	\begin{proof}
	If $\ker(\chi)\supset \operatorname{Stab}(w)$, then
	the pullback of $\Las$ under the induced morphism
	$$\Fl{\I,w}\op\cong
	\Lplus\Iop_u/\operatorname{Stab}(w)\rightarrow\Lplus\Iop_u/\ker(\chi),$$
        is a non-trivial Artin-Schreier local system on $\Fl{\I,w}\op$.
	On the other hand, existence of such a non-trivial local
	system forces $\chi(\operatorname{Stab}(w))$ to a point, which is necessarily $0$ as
	$\operatorname{Stab}(w)$ contains the identity.
	\end{proof}
	For the purpose of proving \cref{orbit:prop}, we may then check that the minimality of the length of $w$ among its $\Wfin$-orbit is equivalent to the vanishing of $\operatorname{Stab}(w)$ under $\chi$.
	First, note that
	  \cite[Lemma 8.4.4]{kalethaPrasad}, together with the definition of
	$\chi$ implies; if for every positive root
	$a\in\posroot$, the containment
	\begin{align}\label{containment}
	U_{-a}(K)\cap \operatorname{Stab}(w)(k)
	\subset U_{-a}(K)_{0+}
	\end{align}
	holds, then $\ker(\chi)\supset \operatorname{Stab}(w)$. Here,
	 $$U_{-a}(K)_{0+}\defined \underset{\substack{\nabla\psi=-a \\ \psi(\a)>0}}\bigcup U_\psi$$
	is the union of certain filtration subgroups, indexed by affine functionals on the apartment associated to $S$.
	These subgroups are determined by the choice of vertex $x$.
	For details, see \cite[Section 13]{kalethaPrasad}.

        Meanwhile, for a positive root $a\in\posroot$, the fact $\operatorname{Stab}(w)=\Lplus I\op(k)\cap \Lplus I_{w(\a)}$ implies that
	\begin{align*}
		&U_{-a}(K)\cap \operatorname{Stab}(w)(k)=\underset{\substack{\nabla\psi=-a \\
		\psi\left(w_0(\a)\right),\psi\left(w(\a)\right)\subset \R_{\geq 0}}}\bigcup U_\psi;\\
	\end{align*}
        as can easily be verified from the definition of parahoric subgroups
	\cite[Definition 7.4.1]{kalethaPrasad}. From this description, it is straightforward to
	observe that the {containment \eqref{containment}} holds if and only if the following condition holds:
	\begin{align}\label{condition}
		\text{For every wall $d\subset w(\a)$, $d+a$ is the hyperplane associated to a positive affine root.}
	\end{align}
	For a simple reflection $s\in\Wfin$, denote by $\alpha_s$ the corresponding positive
	affine root.
	\begin{lem}
		The condition \eqref{condition} holds for every $a\in\posroot$
		if and only if the affine root $w(\alpha_s)$ is positive for every simple
		reflection $s\in\Wfin$.
	\end{lem}
	\begin{proof}
	Suppose there
	exist a simple reflection $s\in\Wfin$, for which $w(\alpha_s)$ is not positive.
		Denote by $d_s$, the wall of
	$\a$ corresponding to $\alpha_s$. Then $w(d_s)+\alpha_s$ is the hyperplane
	associated to the non-positive affine root $w(\alpha_s)+\alpha_s$. The other implication
	follows similarly.
	\end{proof}
	Finally, for a simple reflection $s\in \Wfin$, the positivity of $w(\alpha_s)$ is equivalent to the
	inequality $\ell(sw)\geq\ell(w)$, concluding the proof of \cref{orbit:prop}.
\end{proof}
 	We will identify the set of left $\Wfin$-orbits on $W$ by $\cham$, using the
	splitting ${W\cong\cham\rtimes \Wfin}$ from \cref{subsection:coweights}.
	For $\lambda\in\cham$, denote by $L_\lambda$ the unique non-trivial rank 1 Iwahori-Whittaker
	equivariant local system on $\Fl{\lambda}\op:=\Fl{(t^\lambda)_s}\op$. We may
	then define the standard and
	costandard Iwahori-Whittaker equivariant sheaves as
	\begin{align}\label{def:stdcostdiw}
	\Delta_\lambda^{\IW}\defined
		(j_{(t^\lambda)_s}\op)_*L_\lambda[\ell((t^\lambda)_s)],\quad \nabla_\lambda^{\IW}\defined
	(j_{(t^\lambda)_s}\op)_!L_\lambda[\ell((t^\lambda)_s)]
	\end{align}
	which are in $\PIW$. Consequently, we obtain all simple
	objects of $\PIW$, as the images $\ICIW_\lambda$ of the
	natural morphisms from standard objects to costandard objects.
\subsection{Highest weight structure}
	We can use this collection of objects \eqref{def:stdcostdiw} to equip
	$\PIW$ with the structure of a highest weight category in
	the sense of \cite[Subsection 1.12.3]{baumannRiche}.
	\begin{prop}\label{prop:quotientbruhat}
		Let $\leq_{\mathrm{Bru}}$ be the  Bruhat order on $\cham\cong\Wfin\backslash W$
		from \cref{subsection:coweights}. The category $\PIW$ is a highest weight
		category for the collection of standard and costandard
		objects \eqref{def:stdcostdiw} associated to the poset $(\cham,\leq_{\mathrm{Bru}})$.
\end{prop}
\begin{proof}
	For the partial order on $\cham$ given by the $\Lplus\Iop$-orbit
	closures, the analogous statement follows from \cite[Theorem 3.2.1, Theorem
	3.3.1]{BGS}. Then, it remains to identify this partial order with
	$\leq_{\mathrm{Bru}}$.

	For $\lambda\in\cham$, the dimension of the $\Lplus\Iop$-orbit associated to
	$(t^\lambda)_s$, is equal to $\ell((t^\lambda)_s)$.
	As proven in \cite[Proposition 2.8]{TimoTwistedSchubert}, $\ell((t^\lambda)_s)$ also equals
	to the
	dimension of the closure of the associated
	$\Lplus\Par$-orbit of $\lambda$.  Therefore, the natural
	inclusion of the former, in the latter, is dense.
	The assertion then follows from identification of
	orbit closure relations for the $(\Par,\I)-$Schubert varieties worked out also in
	loc. cit.
\end{proof}
Note the following interpretation of the partial order $\leq_{\mathrm{Bru}}$ in terms of representation
theory:
\begin{lem}\label{lem:concrete}
	For a dominant coweight $\lambda\in\chamb$, and an arbitrary coweight $\mu\in\cham$, the
	inequality $\mu\leq_{\mathrm{Bru}}\lambda$ holds if and only if the $\mu$-weight space of the irreducible
	$\Rep(\check G^I)$-representation corresponding to $\lambda$ is non-zero.
\end{lem}
\begin{proof}
Denote by $\mu_d$ the unique dominant translate of $\mu$ under the natural $\Wfin$ action on
	$\cham$. Then, via arguments analogous to the proof of \cref{prop:quotientbruhat}, the
	$(\I,\Par)$-Schubert cell
	 associated to $\mu_d$ is dense in the $(\Par,\Par)$-Schubert variety associated to
	$\mu_d$. By \cref{rmk:quotientbruhat}, this shows that the inequality
	$\mu\leq_{\mathrm{Bru}}\lambda$ is equivalent
to the inequality $\mu_d\leq_{Bru}\lambda$; which in turn is equivalent to $\mu_d\leq\lambda$ in the
coroot order via \cref{comb}. Finally, equivalence of $\mu_d\leq\lambda$ to the
	weight space condition in the statement is well-known (e.g \cite[Eq.
	9.6]{acharRicheReductive}).
\end{proof}

\begin{cor}\label{cor:realization}
	The realization functor (\cite{beilinsonReal})
	$$\Db(\PIW)\rightarrow \DIW$$
	where $\Db(\PIW)$ is the bounded derived category of the abelian category
	$\PIW$, is an equivalence of categories.
\end{cor}
\begin{proof}
	In \cite[\textsection 3.2-3.3]{BGS} the authors provide an axiomatisation for existence of enough
	projective objects in categories of perverse sheaves constructible with respect to a
	fixed stratification. They then use this to prove the corresponding realization functors
	are equivalences. In fact the highest weight structure is based on, and implies these
	axioms, concluding the proof.
\end{proof}
In light of \cref{cor:realization}, we will tacitly identify $\Db(\PIW)$ with $\DIW$.
\subsection{Averaging}\label{subsection:averaging}
For $\lambda\in \cham$, the inclusion $\Fl{\lambda}\op\subset \Fl{\leq
0}\op$ holds if and only if $\lambda\leq_{\mathrm{Bru}}W 0$ by \cref{prop:quotientbruhat}, and
consequently
only if $\lambda=0$. In particular, the orbit $\Fl{0}\op$ is closed and
$$\stdiw_0\cong \ICIW_0\cong \costdiw_0.$$
We can then consider \textit{Iwahori-Whittaker averaging}:
$$\aviw:\Dbc(\Hk{\I})\rightarrow \DIW,\quad F\mapsto \ICIW_0\star F.$$
We now study the various compatibilities satisfied by this functor.
\begin{lem}\label{lem:icav}
	For $w\in W$,
	$$\aviw(\IC_w)=\begin{cases}
		\ICIW_{w_s} & \text{if } w=w_s\\
		0 & \text{otherwise}
	\end{cases}$$
	\label{averaging:simples}
\end{lem}
\begin{proof}
The proof follows from arguments identical to \cite[Lemma
6.4.4]{centralRiche} and
\cite[Lemma 6.5]{joaoAB}. We sketch it here for the readers convenience. If $w\neq
w_s$, there exist $s\in W_{\mathsf{fin}}$ such that $\ell(sw)<\ell(w)$. Denote by
$\Par_s$ the parahoric group scheme associated to the corresponding wall of $\a$. The
inequality implies that $\IC_w$ descends to $\Dbc(\Lplus\Par_s\backslash \Fl{\I})$. Therefore
$$\ICIW_0\star^{\Lplus\I}\IC_w\cong\left(\pi_{s*}\ICIW_0\right)\star^{\Lplus\Par_s}\IC_w$$
	after their respective pullbacks to $\Dbc(\Fl{\I})$. Here
	$\pi_s:\Fl{\I}\rightarrow\Fl{\Par_s}$ is the quotient morphism and
$\star^{\Lplus\Par_s}$ is the convolution product in $\Dbc(\Lplus\Par_s\backslash \Fl{\I})$. However,
genericity of $\chi$ and proper base change implies that $\pi_{s*}\ICIW_0$
computes cohomology of $\Las$ and is thus $0$.

The other case follows from the observation that, if $w=w_s$, the multiplication map $\Fl{\leq
	0}^{op}\times \Fl{\I,w}\rightarrow \Fl{w}^{op}$ is an isomorphism.
\end{proof}
\begin{lem} \label{std:averaging}
	For $w\in W$,
	$$\aviw(\Delta_w)\cong\costdiw_{w_s},\quad\aviw(\nabla_w)\cong\stdiw_{w_s}.$$
\end{lem}
\begin{proof}
	Let $t\in W_{\mathsf{fin}}$ be the element
	such that $w=t\cdot w_s$. If $t$ is the
	identity, statement follows via arguments
 	as in \cref{lem:icav}. Otherwise,
	we can convolve the natural morphism $\IC_e \rightarrow\nabla_t$ with
	$\nabla_{w_s}$ from the left to obtain a morphism
	$\nabla_{w_s}\rightarrow \nabla_w$. The cofiber of this morphism is a
	succesive extension of objects of the form $\IC_y\star \nabla_{w_s}$ with
	$y\in W_{\mathsf{fin}}\cdot w$ which are killed by $\aviw$ by \cref{lem:icav}.
\end{proof}
\begin{cor}\label{cor:aviewtexact}
	The functor $\aviw$ is perverse t-exact.
\end{cor}
\begin{proof}
	This follows from \cref{std:averaging} together with the fact that
	connective and coconnective parts of the t-structure are generated under
	extensions by shifts of standard and costandard objects.
\end{proof}

\subsection{Aspherical quotient}
Denote by $\Pasph$, the Serre quotient of $\PervI$ by the Serre subcategory generated by
$\IC_w$ such that $w\neq w_s$. Then, \cref{averaging:simples} implies $\aviw$ factors through a
functor $\avasph:\Pasph\rightarrow \PIW$. In fact, we will later prove that this functor is an
equivalence in \cref{theorem:asphiwequiv}.
\begin{thrm}\label{antishperical:averaging}
	The functor $\aviw$ factors through a fully faithful functor
	$$\avasph:\Pasph\rightarrow \PIW,$$
	whose essential image is closed under
	subquotients and extensions.
\end{thrm}
\begin{proof}
	To prove fully faithfulness of $\avasph$, it suffices to construct a left inverse to it.
	Indeed, suppose a left inverse to $\avasph$ exists. Existence of such
	implies that $\avasph$ is faithful, and that for every $X,Y$ in $\Pasph$,
	$$\operatorname{Ext}^1_{\Pasph}(X,Y)\rightarrow \operatorname{Ext}^1_{\PIW}\left(\avasph(X),\avasph(Y)\right)$$
	is injective. From this, one can prove that $\avasph$ also is full, via a straightforward induction on the sum of lengths of $X$ and $Y$. The base case is immediate
	from \cref{averaging:simples}, while the induction step can be verified via the injectivity on the Ext groups mentioned above and the 5-lemma.

	Now, we construct the desired left inverse. First, observe that the averaging functor
	$\aviw$, has a corresponding induction functor. As the induction should be a functor
	from $(\Lplus\Iop_{u},\chi)$-equivariant sheaves to
	$\Lplus\I$-equivariant ones, start by considering their intersection
	$\Lplus\I_0:= \Lplus\I\cap  \Lplus\Iop_{u}$.
	Clearly, $\Lplus \I_0 \subset \ker\chi$, which implies Iwahori-Whittaker
	equivariant objects are on the nose $\Lplus\I_0$-equivariant. Therefore, there exist a functor
		$$q^*:\DIW \rightarrow\Dbc(\Lplus \I_0\backslash \Fl{\I}),$$
	which factorizes the inclusion $\DIW\rightarrow\Dbc(\Fl{\I})$, via postcomposition with
	the pullback associated to the quotient map $\Fl{\I}\rightarrow \Lplus \I_0\backslash
	\Fl{\I}$.
	Similarly, the containment $\Lplus\I_0\subset \Lplus\I$ implies we can consider the pushforward
	$$p_*:\Dbc(\Lplus \I_0\backslash \Fl{\I})\rightarrow \Dbc(\Hk{\I}),$$
	associated to the quotient map $p:\Lplus \I_0\backslash \Fl{\I}\rightarrow \Hk{\I}.$
	Finally we define:
	$$\operatorname{ind}_\IW:=p_*\circ q^*:\DIW \rightarrow\ \Dbc(\Hk{\I}).$$ %^pH^{-\mathsf{rk}G}\circ}
        It is straightforward to observe from the definition that
	$$\operatorname{ind}_\IW(\aviw(F))\cong
	\operatorname{ind}_\IW(\ICIW_0)\star F,$$
	for every $F\in \Dbc(\Hk{\I})$.
	\begin{lem}\label{lemma:conebound}
		There is a cofiber sequence,
		$$\IC_e\rightarrow \operatorname{ind}_\IW(\IC_0)[-\mathsf{rk}G]\rightarrow \mathsf{C},$$
		such that $\mathsf{C}$ is a succesive extension of $\IC_w[n]$ with $n\in\Z_{\leq
		0}$,
		and $w\in \Wfin\setminus \{e\}$. \label{cofiber:induction}
	\end{lem}
	\begin{proof}[Proof of \cref{lemma:conebound}]
		The image of the special fiber of $\I$, denoted $\I_k$,
		in the reductive quotient of the special fiber of
		$\Par$, denoted $\Par_k^\red$,
		is a Borel subgroup \cite[Theorem 8.4.19]{kalethaPrasad}.
		Moreover, the action of
		$\Lplus\Par$ on the base point of $\Fl{\I}$
		identifies $\Fl{\I,\leq w_o}$
		with $\Lplus\Par/\Lplus\I$,
		which is precisely the flag variety of $(\Par)_k^\red$, \cite[Proposition
		8.4.16]{kalethaPrasad}. In light of this, proof of the lemma is identical to
		the proof of \cite[Lemma 6.4.8]{centralRiche} applied to
		the split reductive group $\Par_k^\red$, together with
		its Borel determined by $\I_k$, and the unipotent radical thereof.
	\end{proof}

	We now finish the proof of \cref{antishperical:averaging}. Let $\Pi_{\mathsf{asph}}:\PervI\rightarrow \Pasph$ denote the natural
	projection functor. Then, the composition  $R:=\Pi_{\mathsf{asph}}\circ {}^pH^0\circ
	{\operatorname{ind}_\IW}[-\mathsf{rk}G]$ is the right inverse of $\avasph$. To
	prove this, it
	suffices to check, for every $F\in \PervI$,
	\begin{equation}\Pi_{\mathsf{asph}}\circ{}^pH^{-1}(\mathsf{C}\star F)\cong 0\cong
	\Pi_{\mathsf{asph}}\circ {}^pH^{0}(\mathsf{C}\star F).\label{cone:bound} \end{equation}
	Indeed, convolving the cofiber sequence defining $\mathsf{C}$ with $F$ produces the cofiber sequence
	$$F\rightarrow \operatorname{ind}_\IW(\IC_0)\star F[-\mathsf{rk}G]\rightarrow \mathsf{C}\star F.$$
	If \eqref{cone:bound} holds, applying $\Pi_{\mathsf{asph}}\circ {}^pH^0$ shows
	$$\Pi_{\mathsf{asph}}(F)\cong R\circ\avasph(\Pi_{\mathsf{asph}})(F).$$
	Finally, to check $\ref{cone:bound}$ it suffices to verify the same statement, where $\mathsf{C}$
	is replaced by $\IC_w[n]$ as in \cref{lemma:conebound}. In this case, there exists a
	simple reflection $s\in\Wfin$ such that $\IC_w[n]$, and consequently
	the perverse cohomology sheaves of the convolution $\IC_w[n] \,\star\, F$, are $\Lplus\Par_s$
	equivariant; showing they vanish in the antispherical quotient (cf. \cite[Lemma
	6.4.10]{centralRiche}).

	The fact that the essential image is closed
	under subquotients follows from a simple
	induction on the length of composition
	series, using that
	$\avasph$ preserves simple objects. An analogous argument is spelled out in \cite[Lemma
	7.12]{bezrukavnikovRicheRider}.
\end{proof}

\section{Tilting objects}\label{section:tilting}
We fix a tuple $(G,B,S,\mathbf{f})$ as in
\cref{section:iwahoriwhittaker} and assume additionally that $G$ is tamely ramified.
There, we have endowed the corresponding category $\PIW$ with the structure of a highest weight
category. Therefore, we can speak of \textit{tilting objects}, which are those objects
that admit two filtrations: one whose graded pieces are sums of standard objects $\stdiw_\lambda$, and
another
whose graded pieces are sums of costandard objects $\costdiw_\lambda$ of \eqref{def:stdcostdiw}.

This section will be devoted to the proof of the following:
\begin{prop}\label{main:tilting}
	Let $\lambda\in\splitchamb$ be such that $\bar\lambda\in\chamb$ is minimal in
	$\chamb\setminus\{0\}$ with respect to the coroot order (see \cref{subsection:coweights}).
	Let $V$ be the irreducible $\check G$-representation of highest weight $\lambda$. Then the object
	$$\aviw\circ\Zent(V)$$
	is tilting.
\end{prop}
In the remainder of the section we will assume that the group $G$ is adjoint and tamely
ramified, unless explicitly stated otherwise. The adjoint assumption is not essential, as the
proof of the result for a group $G$ reduces to the case of its adjoint group $G_\ad$ as we explain below:

Given an arbitrary reductive group $G$,
the quotient map $G\rightarrow G_\ad$ induces a map between the respective Bruhat-Tits buildings. This
yields a map
$\mathcal{G}\rightarrow \mathcal{G}_\ad$, of the corresponding parahoric group schemes of $G$, $G_\ad$ respectively.
  Finally,
the induced map
$$\Hk{\mathcal{G}}\rightarrow \Hk{\mathcal{G}_\ad},$$
 is surjective, and restricts to universal homeomorphisms between associated connected
components \cite[Corollary A.4]{modularRamified}. In light of this, the statement for $G$
reduces to one of $G_\ad$, via pushforwards.
\subsection{Preliminaries}
To prove a given object in $\PIW$ is tilting, we will check the following criterion:
\begin{prop}\label{criterion}
	An object $F\in \PIW$ is tilting, if and only if for every $\lambda\in\cham$, the stalk $j_{(t^\lambda)_s}^{\mathrm{op}\: *}F$ and
	the costalk $j_{(t^\lambda)_s}^{\mathrm{op}\: !}F$ are both concentrated in degree
	$-\ell((t^\lambda)_s)$.
\end{prop}
\begin{proof}
	In \cite[Proposition 1.3]{BBM04}, the authors give an axiomatized proof of this fact, for highest
	weight structures on categories of certain categories of perverse sheaves, arising from
	stratifications with contractible strata. The highest weight structure we have endowed on $\PIW$
	satisfies the necessary axioms for their proof to be applicable.
\end{proof}
We will denote by $\Zent^\IW$ the composition $\aviw\circ\Zent$. For
$V\in\operatorname{Rep}(\check G)$, the stalk and costalk at $(t^\lambda)_s$ of $\Zent^\IW(V)$
are related to the $\lambda$ weight space of the restriction $\res(V)_\lambda$, of $V$ to a
$\check G^I$-representation:
\begin{prop}\label{whittaker:eulerchar}
	For $V\in\operatorname{Rep}(\check G)$ and $\lambda\in\cham$,
	\begin{align*}\label{euler:char}
		\chi^!_\lambda\left(\Zent^\IW(V)\right)\defined\sum_{n\geq
		0}(-1)^n\dim\left(\Hom_{\DIW}\left(\stdiw_\lambda,\Zent^\IW(V)[n]\right)\right)=\dim\left(\res(V)_\lambda\right),\\
		\chi^*_\lambda\left(\Zent^\IW(V)\right)\defined\sum_{n\geq
		0}(-1)^n\dim\left(\Hom_{\DIW}\left(\Zent^\IW(V),\costdiw_\lambda[n]\right)\right)=\dim\left(\res(V)_\lambda\right).
	\end{align*}
\end{prop}
\begin{proof}
	From \cref{theorem:filtration}, we obtain the following equality in $K_0(\PervI)$:
	\begin{equation}\label{K0:Zent}
		\left[\Zent(V)\right]=\sum_{\lambda\in\cham}\dim(\res(V)_\lambda)\cdot\left[J_\lambda\right].
	\end{equation}
	Moreover, by \eqref{decategorification:iso} we have
	$\left[J_\lambda\right]=\left[\std_{t^\lambda}\right]=\left[\costd_{t^\lambda}\right]$.
	The formula
	$$\aviw(\std_{t^\lambda})=\stdiw_\lambda$$ of \cref{std:averaging}, together with the ring structure of $K_0(\PervI)$ lead to similar equalities
	$$
		\left[\Zent^\IW(V)\right]=\sum_{\lambda\in\cham}\dim\left(\res(V)_\lambda\right)\cdot\left[\stdiw_\lambda\right]=\sum_{\lambda\in\cham}\dim\left(\res(V)_\lambda\right)\cdot\left[\costdiw_\lambda\right],$$
		in $K_0(\PIW)$. As it is clear that $\chi^!_\lambda$ and $\chi^*_\lambda$ factor through the
		Grothendieck group, the following equalities conclude the proof:
		$$\dim\left(\Hom_{\DIW}\left(\stdiw_\lambda,\costdiw_\mu[n]\right)\right)=\begin{cases}
			\delta_{\lambda\mu}&\text{if } n=0,\\
			0&\text{else}.
		\end{cases}$$
\end{proof}

The following lemma will be used for reduction steps later.
\begin{lem}\label{wfin:orbit}
	Let $V\in \Rep(\check G)$, $\lambda\in\cham$ and $x\in \Wfin$. Then for every $n\in \Z$,
	$$\Hom\left(\stdiw_\lambda,\Zent^\IW(V)[n]\right)\cong
	\Hom\left(\stdiw_{x\cdot\lambda},\Zent^\IW(V)[n]\right),$$
	and
	$$\Hom\left(\Zent^\IW(V)[n],\costdiw_\lambda\right)\cong
	\Hom\left(\Zent^\IW(V)[n],\costdiw_{x\cdot\lambda}\right);$$
	in the category $\DIW$.
\end{lem}
\begin{proof}
	Proof follows the strategy of \cite[Lemma 6.5.11]{centralRiche}.
	Without loss of generality, we may assume $\lambda$ is dominant, therefore
	${(t^\lambda)_s=t^\lambda}$. Now, there exist $y\in
	\Wfin$, such that $t^\lambda=(t^{x\cdot\lambda})_s \cdot y\inv$,
	and
	$y$ is of minimal length with $x\cdot\lambda=\lambda\cdot y$ \cite[Lemma
	2.4]{richeMautnerExotictstructure}. Using the averaging formula for standard objects
	from \cref{std:averaging}, this implies the following chain of isomorphisms:
	$$\std_\lambda^\IW\defined\std_0^\IW\star\std_{(t^\lambda)_s}\cong\stdiw_0\star\std_{(t^{x\cdot\lambda})_s}\star\std_{y\inv}\cong\stdiw_{x\cdot\lambda}\star\std_{y\inv}.$$
	Therefore, by the compatibility of standard objects with convolution (\cref{std:conv}), and
	the averaging formula once again; we get the following chain of identifications:
	\begin{align*}
	\Hom\left(\stdiw_\lambda,\Zent^\IW(V)[n]\right)\cong&
	\Hom\left(\stdiw_\lambda\star\std_y,\Zent^\IW(V)\star\std_y[n]\right)\\
		\cong&\Hom\left(\stdiw_{x\cdot\lambda},\Zent^\IW(V)\star\std_y[n]\right)\\
		\cong&\Hom\left(\stdiw_{x\cdot\lambda},\stdiw_0\star\Zent(V)\star\std_y[n]\right)\\
		\cong&\Hom\left(\stdiw_{x\cdot\lambda},\stdiw_0\star\std_y\star\Zent(V)[n]\right)\\
		\cong&\Hom\left(\stdiw_{x\cdot\lambda},\Zent^\IW(V)[n]\right).
	\end{align*}
	The remaining case can be proven similarly.
\end{proof}
\subsection{Representations with (quasi-)minuscule highest weight}
Recall from \cref{iwahoriweylgroup} that $\chamb$ are the dominant coweights for the
\`echelonage root system $\breve\Sigma$. An element $\bar \mu \in\chamb\setminus\{0\}$, minimal
among such
with respect to the coroot order  is called
\begin{itemize}
	\item \textit{minuscule} if $\langle\alpha,\bar\mu\rangle \in \{0,1\}$ for all roots $\alpha \in\breve\Sigma$,
	\item \textit{quasi-minuscule}, otherwise.
\end{itemize}
\begin{lem}\label{lemma:minusculerep}
	Let $\lambda\in\splitcham$ with image $\bar\lambda\in\chamb$. Let $V$ denote irreducible $\check G$-representation with
	highest weight $\lambda$, and $V_{\bar\lambda}$ the irreducible $\check
	G^I$-representation of highest weight $\bar\lambda$. Then
	$$\Res(V)\cong
	V_{\bar\lambda}\oplus\bigoplus_{\bar\mu\leq\bar\lambda}
	V_{\bar\mu}^{c_{\bar\lambda,\bar\mu}},$$
	where $\bar\mu\in\chamb$ is distinct from $\bar\lambda$, $\leq$ denotes the coroot order, and $c_{\bar\lambda,\bar\mu}$
	are non-negative integers.
\end{lem}
\begin{proof}
This is
	proven for tamely ramified groups in \cite[Lemma 4.10]{zhuRamified}. The argument was
	observed to work in general in the proof of \cite[Theorem 4.11]{richarzRamified}.
	Indeed, since the map $\splitchamb\rightarrow\chamb$ preserves the respective coroot orders (see
	\cref{subsection:coweights}), it suffices to prove that the multiplicity of $\bar\lambda$
	weight space in $\Res(V)$ is 1. This is obtained via geometric methods, see
	\cite[Lemma 2.6]{zhuRamified}.
\end{proof}
\subsection{Proof in the minuscule case}
\begin{lem}\label{minuscule}
	Let $\lambda\in\splitchamb$ be such that $\bar\lambda\in\chamb$ is
	minuscule in the \`echelonnage root system. Denote by $V$ the
	irreducible $\check G$-representation corresponding to $\lambda$. Then $\Zent^\IW(V)$ is tilting.
\end{lem}
\begin{proof}
	We check the criterion in \cref{criterion}. Using the description of $\Res(V)$ in
	\cref{lemma:minusculerep} and \cref{whittaker:eulerchar},
	we have that  $\Zent^\IW(V)$ is supported on $\overline{\Fl{G,\bar\lambda}\op}$.
	Since $\bar\lambda$ is minuscule, $\Fl{G,\bar\mu}\op\subset \overline{\Fl{G,\bar\lambda}\op}$ implies
	that $\bar\mu\in \Wfin\cdot\lambda$ (see \cref{lem:concrete}). Therefore, by
	\cref{wfin:orbit}, it suffices to check the criterion for $\bar\lambda$. Since
	$\Fl{G,\bar\lambda}\op$ is open in $\overline{\Fl{G,\bar\lambda}\op}$, the lemma follows.
\end{proof}
\subsection{Proof in the quasi-minuscule case}
\begin{prop}\label{quasi-minuscule}
	Let $\lambda\in\splitchamb$ such that
	${\bar\lambda\in\chamb}$ is quasi-minuscule in the \`echelonage root system. Denote by
	$V$ the irreducible $\check G$-representation corresponding to $\lambda$. Then
	$\Zent^\IW(V)$ is tilting.
\end{prop}
For the proof, we will need the following lemma, whose proof will be postponed to the end of
\cref{subsection:weightmonodromy}.
\begin{lem}\label{lemma:inequality0wtspc}
	Let $\lambda\in\splitchamb$ be such that $\bar\lambda\in\chamb$ is quasi-minuscule.
	Let $V\in\Rep(\check G)$ be the irreducible representation with highest weight
	$\lambda$. Then, there are inequalities
	\begin{align*}
	\dim\left(\Hom_{\DIW}\left(\stdiw_0,\Zent^\IW(V)\right)\right)\leq\dim\Res(V)_0,\\
	\dim\left(\Hom_{\DIW}\left(\Zent^\IW(V),\costdiw_0\right)\right)\leq\dim\Res(V)_0;
	\end{align*}
	where $\Res(V)_0$ denotes the $0$ weight space of the restriction of $V$ to a $\check G^I$-representation.
\end{lem}
\begin{proof}[Proof of \cref{quasi-minuscule}]
	We proceed by checking the criterion in \cref{criterion}. First, note that
	using \cref{lemma:minusculerep}, \cref{whittaker:eulerchar}
	implies that $\Zent^\IW(V)$ is supported on $\overline{\Fl{G,\bar\lambda}\op}$.
	Since $\bar\lambda$ is quasi-minuscule, the containment
	${\Fl{G,\bar\mu}\op\subset \overline{\Fl{G,\bar\lambda}\op}}$ holds
	if and only if $\bar\mu\in \left((\Wfin\cdot\bar\lambda)\cup\{0\}\right)$,
        as follows from the representation
	theoretic interpretation of the closure order in \cref{lem:concrete}.

	For non-zero coweights $\bar\mu$, the criterion is then satisfied by noting that it
	holds for $\bar\lambda$ by \cref{wfin:orbit}.
	It remains to prove it for the stalk and costalk at $0$. Denote by
	$i:\overline{\Fl{G,0}\op}\rightarrow \Fl{G}$ the inclusion of the orbit closure and by
	$j$ the immersion associated to its open complement. The case of non-zero $\bar\mu$
	above show that $j^*\Zent^\IW(V)$ is a perverse sheaf on $\Fl{G}\setminus
	\overline{\Fl{G,0}\op}$. Then, localization
	triangles associated to the decomposition $(i,j)$ imply that shriek and star restrictions of
	$\Zent^\IW(V)$ to $\overline{\Fl{G,0}\op}$ are concentrated in perverse degrees
	$\{0,-1\}$ and $\{0,1\}$ respectively. As the orbits $\Fl{\bar\mu}\op$ are affine
	spaces \eqref{eq:orbitsaffine}, this perverse amplitude implies that the first equality in
	\cref{whittaker:eulerchar} associated to the coweight zero reduces to:
	\begin{align*}
		\dim\Res(V)_0=\sum_{n\in\{0,1\}}(-1)^n\dim\left(\Hom_{\DIW}\left(\stdiw_0,\Zent^\IW(V)[n]\right)\right).
	\end{align*}
		Therefore, \cref{lemma:inequality0wtspc} implies:
	\begin{align*}
	\dim\left(\Hom_{\DIW}\left(\stdiw_0,\Zent^\IW(V)[1]\right)\right)=0
	\end{align*}
	as desired. The argument for $\costdiw_0$ is analogous, concluding the proof.
\end{proof}
\subsection{Propagation through tensor products}
Since all (coroot order) minimal coweights in $\chamb\setminus \{0\}$ are by definition either
minuscule or quasi-minuscule, \cref{minuscule} and \cref{quasi-minuscule} together prove
\cref{main:tilting}. Now, we discuss the extension to arbitrary coweights.
First note the following, which shows that the tilting property propagates through tensor products:
\begin{lem}\label{propagation}
	For $V,W\in\operatorname{Rep}(\check G)$, if $\Zent^\IW(V),\Zent^\IW(W)$ are
	tilting then so is $\Zent^\IW(V\otimes W)$.
\end{lem}
\begin{proof}
 This follows from the strategy of
	\cite[Proposition 6.5.7]{centralRiche}, using \cref{std:averaging}. Briefly, using the
	centrality of $\Zent$, one observes that
	$$\Zent^\IW(V)\star\Zent(W)\cong\Zent^\IW(V\otimes W)\cong\Zent(V)\star\Zent^\IW(W).$$
	Using this, one leverages the respective filtrations of $\Zent^\IW(V)$ and $\Zent^\IW(W)$
	which exhibit them to be tilting, to construct such for $\Zent^\IW(V\otimes W)$.
\end{proof}
Using this we finally deduce:
\begin{thrm}\label{cor:key}
	For every $V\in\Rep(\check G^I)$, there exist $V'\in\Rep(\check G)$ such that
	\begin{enumerate}
		\item the representation $V$ is a direct summand in $\Res(V')$;
		\item the object $\Zent^\IW(V)$ is tilting.
	\end{enumerate}
\end{thrm}
\begin{proof}
	As $G$ is adjoint by assumption, $\check G^I$ is connected by \cite[Proposition
	4.1.(d)]{hainesSatake}.
Therefore, \cite[Lemme 10.3]{ngoDemazure} applies and any irreducible representation
$V\in\Rep(\check G^I)$ is a direct summand of a tensor
product of irreducible representations with
minuscule and quasi-minuscule highest weights. Clearly, the direct summand of a tilting object
is also tilting; therefore, \cref{propagation} reduces the assertion to the case where $V$ is
such an irreducible representation with (quasi-)minuscule highest weight. Since the projection
	$\splitchamb\rightarrow\chamb$ is surjective for adjoint groups, \cite[Lemma
	2.6]{modularRamified}, this case follows directly from \cref{main:tilting}.
\end{proof}
\subsection{Regular quotient}\label{subsection:regular}
Let $\mathsf{P}^{>0}$ be the Serre subcategory of $\PervI$
generated by $\IC_w$ with $\ell(w)>0$, and denote by $\Po$ the corresponding Serre quotient. We
will denote by $\Pi^0$ the natural quotient
functor. The following is a remarkable observation of Bezrukavnikov
\cite{bezrukavnikovTensor}.
\begin{prop}
	The truncated convolution product $^pH^0(\cdot\star \cdot)$ on $\PervI$
	descends to an exact monoidal structure $\circledast$ on $\Po$.
\end{prop}
\begin{proof}
Consider $\IC_w$ with $\ell(w)>0$. Then there exist a simple reflection $s$ such that
	$\ell(sw)<\ell(w)$ thus, $\IC_w$ is equivariant with respect to the associated parahoric
	subgroup $\Par_s$. As this property is stable under convolution (see \cref{remark:convolution}) and taking perverse
	cohomology sheaves, we deduce $\mathsf{P}^{>0}$ is stable under truncated convolution;
	proving the bifunctor descends to $\Po$. It remains to prove exactness, which
	follows from the fact that for $w,v\in W$ with $\ell(w)=\ell(v)=0$
	we have
	$\IC_w\star \IC_v=\IC_{wv}$ by \cref{std:conv}.
\end{proof}
The following lemma is immediate yet important.
\begin{lem}\label{lemma:zent0}
	The composition $\Zent^0=\Pi^0 \circ \Zent:\operatorname{Rep}(\check G)\rightarrow
	\Po$ is monoidal and central. Moreover, it carries an endomorphism
	$n^0_\cdot: \Zent^0(\cdot)\rightarrow \Zent^0(\cdot),$
 	which is objectwise nilpotent
\end{lem}
\begin{proof}
	Follows immediately from identical properties of $\Zent$ introduced in
	\cref{subsection:centralfunctor}, by postcomposition with $\Pi^0$ and functoriality.
\end{proof}
Now we can employ some general machinery developed in \cite{centralRiche}, see also
\cite[Proposition 5.8]{bezrukavnikovRicheRider} for a multiplicative perspective.
\begin{prop}\label{regular:tannakian}
	Consider the full subcategory $\tilde{\mathsf{P}}^0$ of $\Po$ generated under subquotients by
	the essential image of $\Zent^0$. Then there exist
	\begin{center}
		\begin{itemize}
			\item[a)] a subgroup scheme $H\subset \check G$;
			\item[b)] an element $n^0\in\frak{\check g}$ such that $H\subset
				Z_{\check G}(n^0)$;
			\item[c)] an equivalence of monoidal categories
				$$\Phi:(\tilde{\mathsf{P}}^0,\circledast)\cong
				(\operatorname{Rep}(H),\otimes);$$
			\item[d)] and an isomorphism of functors $\eta:\Phi\circ \Zent^0\cong
				\operatorname{Res}^{\check G}_H$
				such that for $V\in\operatorname{Rep}(\check G)$, the
				endomorphism $\eta\left(\Phi(n^0_V)\right)$ coincides
				with the action of $n^0$ on the underlying vector
				space of $V$.
		\end{itemize}
	\end{center}
	\begin{rmk}\label{remark:inertiaequiv}
		After proving \cref{main:tilting}, we will observe
		that $\tilde{\mathsf{P}}^0$ is equal to $\Po$.
	\end{rmk}
		\end{prop}
	\begin{proof}
		The proof of the analogous statement \cite[Proposition
		6.5.18]{centralRiche} in the split setting depends only on the
		properties of $\Zent^0$ from \cref{lemma:zent0}, therefore applies also in our
		setting.
	\end{proof}
		From \cref{remark:commutativitymonodromy}, we obtain the following equality:
\begin{lem}\label{inequality}
	For every $V\in\Rep(\check G)$ the inequalities
	\begin{align*}
		&\dim\left(\Hom_{\DIW}\left(\stdiw_0,\Zent^\IW(V)\right)\right)\leq\dim(\ker
		n^0_V),\\
		&\dim\left(\Hom_{\DIW}\left(\Zent^\IW(V),\costdiw_0\right)\right)\leq\dim(\ker
		n^0_V)
	\end{align*}
		hold.
\end{lem}
\begin{proof}
	Since Iwahori-Whittaker averaging factors fully-faithfully through the antispherical
	quotient by \cref{antishperical:averaging}, we have that
	$$\Hom_{\DIW}\left(\stdiw_0,\Zent^\IW(V)\right)\cong\Hom_{\Pasph}\left(\Pi(\IC_e),\Pi(\Zent(V))\right).$$
	By definition, the projection to the regular quotient $\Pi_0$ factors through the
	antispherical projection $\Pi$. Moreover, as $\Pi(\IC_e)$ is a simple object, there is
	an injection
	$$\Hom_{\Pasph}\left(\Pi(\IC_e),\Pi(\Zent(V))\right)\hookrightarrow\Hom_{\Po}\left(\Pi^0(\IC_e),\Pi^0(\Zent(V))\right).$$
	We claim that
	$$\Hom_{\Po}\left(\Pi^0(\IC_e),\Pi^0(\Zent(V))\right)\isomto\Hom_{\Po}\left(\Zent^0(\Qell),\ker(\Pi^0(\bn_V))\right).$$
	Indeed,
	$f:\Zent^0(\Qell)\rightarrow\Zent^0(V)$, $f\circ\Pi^0(\bn_{\Qell})=\Pi^0(\bn_V)\circ f$
        using
	\cref{remark:commutativitymonodromy} and $\bn_{\Qell}=0$, which proves the claim.
	Finally, using the Tannakian description \cref{regular:tannakian}, we obtain an injection
	$$\Hom_{\Po}\left(\Zent^0(\Qell),\ker(\Pi^0(\bn_V))\right)\hookrightarrow\Hom_{\Vect_{\Qell}}\left(\Qell,\ker(n^0_V)\right)\cong\ker(n^0_V).$$
	This implies the first of the asserted inequalities. The second inequality can be obtained similarly.
\end{proof}
\subsection{Weight and monodromy filtrations}\label{subsection:weightmonodromy}
	Recall that, to a nilpotent endomorphism of an object in an abelian category, one can associate
	the Jacobson-Morozov-Deligne filtration \cite[Proposition
	1.6.1]{weil2}. For $V\in\Rep(\check G)$, such filtration on
	$\Zent(V)$ induced by the logarithm of the monodromy $\bn_V$, is also known as
	\textit{the monodromy filtration}. More explicitly, it is the unique bounded increasing
	filtration $F_i(\Zent(V))$ such that:
	\begin{gather*}
		\bn_V(F_i(\Zent(V))\subset F_{i-2}(\Zent(V)),\\
		\gr_i^F(\Zent(V))\isomto \gr_{-i}^F(\Zent(V))
	\end{gather*}
	where the isomorphism between graded pieces is obtained from $\bn_V^i$, the $i$-th
	power of the monodromy.

		On the other hand,
	$\Zent(V)$ carries the apriori distinct filtration, obtained from the weight
	filtration of $\Zent^\mix(V)$ under base change. By an important theorem of Gabber
	\cite[\textsection 5.1]{beilinsonBernsteinJantzen}, the weight and monodromy filtrations coincide.
	This remarkable fact will allow us to understand the monodromy filtration of $\Zent(V)$, in terms of
	Langlands duality.
	\begin{lem}
		Let $V\in\Rep(\check G)$, and $n^0_V$ the nilpotent endomorphism of the
		underlying vector space of $V$, considered in \cref{regular:tannakian}. Then,
		$$\dim(\ker n^0_V)=\sum_{\substack{\bar\mu\in\cham \\
		\left<\mu,2\rho\right>\in\{0,1\}}}\dim\res(V)_{\bar\mu}.\label{ker:dim}$$
	\end{lem}
\begin{proof}
	Consider the Jacobson-Morozov-Deligne filtration induced by $n^0_V$
	of $\Zent^0(V)$. By uniqueness, it agrees with
	the filtration obtained under $\Pi^0$, from the monodromy filtration on $\Zent(V)$.
	We will also call this the monodromy filtration of $\Zent^0(V)$, and by extension of
	the underlying vector space of $V$.
	By the explicit description of the monodromy filtration \cite[Section 1.6.7]{weil2},
	the dimension on the left hand side of the desired equality is equal to the
	sum of the dimensions of $0$-th and $1$-st graded pieces of the monodromy filtration on $V$. We will proceed by
	calculating this dimension, using the agreement of weight and monodromy filtrations. In
	other words, we will calculate the sum of the dimensions of $0$-th and $1$-st graded
	pieces of the filtration on $V$, which is obtained under $\Pi^0$ from the weight
	filtration of $\Zent(V)$.

	Note that, the image of $\IC_w$ in $\tilde{\mathsf{P}}^0$ is invertible for $\ell(w)=0$,
	and vanishes for
	$\ell(w)>0$.
	Therefore, using the description \eqref{eq:weightfilt}, we deduce that the dimension of the
	$i$-th graded piece of the weight filtration on $\Zent^0(V)$
	is the sum of multiplicities
	$$\sum_{w\in\Omega}\left(\grW_i\left(\Zent^{\text{mix}}(V)\right):\IC_w\right),$$
	ranging over the set $\Omega$ of length zero elements in $W$.
	On the other hand, we have the
	equality
	\begin{equation}\label{mixedcentral}
	[\Zent^{\text{mix}}(V)]=\sum_{\bar\mu\in\cham}\dim\res(V)_{\bar\mu}\cdot\theta_\mu
	\end{equation}
	in $\mathcal{H}$, by \cref{mixed:central}.
	Therefore, to calculate the desired multiplicities, we may apply the
	$\ogm$-algebra morphism
	$m$ \eqref{eq:multiplicitymorph}, carrying
	the equality \eqref{mixedcentral} to $\ogm$.
	The explicit calculations \eqref{m:theta} and \eqref{equation2}
	of $m$ then yield the equality:
	$$\sum_{i\in\Z}\sum_{w\in\Omega}\left(\grW_i\left(\Zent^{\text{mix}}(V)\right):\IC_w\right)\cdot\mathbf{v}^i=\sum_{\bar\mu\in\cham}\dim(\res(V)_{\bar\mu})\cdot\mathbf{v}^{\left<\bar\mu,2\rho\right>}$$
	in $\ogm$, which concludes the proof by comparing coefficients of the variable.
\end{proof}

We are finally in a position to prove \cref{lemma:inequality0wtspc}, which asserts the
inequalities:
\begin{align*}
	\dim\left(\Hom_{\DIW}\left(\stdiw_0,\Zent^\IW(V)\right)\right)\leq\dim\Res(V)_0,\\
	\dim\left(\Hom_{\DIW}\left(\Zent^\IW(V),\costdiw_0\right)\right)\leq\dim\Res(V)_0.
\end{align*}
\begin{proof}[Proof of \cref{lemma:inequality0wtspc}]
	Since $\bar\lambda$ is quasi-minuscule,
	each non-zero
	weight $\bar\nu$ of $\Res(V)$, is also a non-zero
	weight of the adjoint representation of $\check G^I$. Since all such weights
	are roots in the \`echelonage root system,
	the integers $\left<\bar\nu,2\rho\right>$ are even. Then, for this $V$, \cref{ker:dim} reads:
	$$\dim(\ker n^0_V)=\dim\Res(V)_0;$$
	and we conclude by \cref{inequality}.
\end{proof}
\subsection{Projection of tilting objects to the regular quotient}
Denote by $\PIW^0$ the Serre quotient of $\PIW$ by the Serre subcategory generated by the
collection of objects
$\{\ICIW_\lambda \:|\: \ell((t^\lambda)_s)>0\}$. Clearly,
\begin{equation}\label{eq:regulariwregular}
\Po\cong\PIW^0
\end{equation}
as a corollary to \cref{averaging:simples}.

The following analogue of \cite[Lemma
7.6]{bezrukavnikovRicheRider} will be repeatedly used in \cref{section:coherentfunc}:
\begin{lem}\label{lemma:tiltingregular}
	The quotient functor $\Pi^0_\IW:\PIW\rightarrow\PIW^0$ is fully faithful when
	restricted to tilting objects in $\PIW$.
\end{lem}
\begin{proof}
	The proof of the analogous statement \cite[Lemma 7.6]{bezrukavnikovRicheRider} in
	the split setting implements a general strategy from \cite[Section
	2.1]{tiltingExercises}. It amounts to an analysis of socle and top of (co)standard
	Iwahori-Whittaker sheaves \eqref{def:stdcostdiw} and also goes through in our
	setting.
\end{proof}
\subsection{Two geometrizations of the aspherical module}
\cref{cor:key} in particular implies that every indecomposable
tilting object is a subquotient of $\Zent^\IW(V)$ for some
$V\in\operatorname{Rep}(\check G)$. Thus we arrive at the following theorem:
\begin{thrm}\label{theorem:asphiwequiv}
	The functor
	$$\avasph:\Pasph\rightarrow\PIW $$
	is an equivalence.
	\end{thrm}
	\begin{proof}
By the general properties of highest weight categories,
		every object in $\PIW$ is a subquotient of a direct sum of indecomposable tilting
		objects \cite[Proposition 7.17]{richeThesis}. As the essential image of $\Zent^\IW$ is closed under subquotients by
		\cref{antishperical:averaging}, this shows that the functor $\Zent^\IW$
	 is essentially surjective. As $\Zent^\IW$ is also fully faithful by
		\cref{antishperical:averaging}, it is an equivalence.
	\end{proof}

\section{Coherent functor}\label{section:coherentfunc}
In this section, we explain the construction of a functor
\begin{align}\label{eq:coherentfunctor}
	F_\IW:\Db\Coh^{\check G^I}(\Sprgi)\rightarrow\DIW
\end{align}
 which is obtained in \cref{prop:construction}, and shown to be an equivalence in \cref{theorem:main}.
The basic strategy of the construction proceeds along the lines of the unramified case, which is
explained with great detail in \cite[Section 6.3]{centralRiche} and neatly summarized in
\cite[Section 5]{joaoAB}.
However, there are two notable differences:
\begin{enumerate}
	\item the lack of a central functor from $\Rep(\check G^I)$ to $\Perv(\Hk{\I})$;
	\item the possibility that $\check G^I$ is disconnected.
\end{enumerate}
Difference (1) necessitates different strategies, and is adressed most directly in
\cref{subsection:keylemma} and \cref{subsection:descending}.
In contrast, dealing with (2) turns out to be a suprisingly straightforward exercise in
generalizing the requisite constructions from the connected case.
\subsection{On the (dis)connectedness of $\check G^I$}\label{subsection:disconnectedness}
Recall that $\check G$ is a pinned group equipped with an action of $I$ by pinned automorphisms.
In such a case, the group of fixed points $\check G^I$ was studied in \cite[Proposition
5.1]{hainesEchelonnage} over a field and in \cite{richarzPinning} over a general base ring. It
turns out that the identity component
$\Gdual$ is a reductive a group, the identity component $\Bdual$ of $\check B^I$ is a Borel
subgroup and $\check U^I$ is the unipotent radical of $\Bdual$.
Since the action of inertia is
pinning preserving, the exact sequence
$$1\rightarrow \check U^I\rightarrow \check B^I\rightarrow \check T^I\rightarrow 1,$$
is split by the inclusion $\check T^I\rightarrow \check B^I$.

For $\bar\lambda\in\cham$, denote by $\coweyl(\bar\lambda)$ the $\check G^I$-representation
$\Ind_{\check B^{-I}}^{\check G^I}(\bar\lambda)$ induced from invariants $\check B^{-I}$ of the opposite Borel. Here is a curious observation:
\begin{lem}\label{lemma:curious}
	Let $\lambda\in\splitcham$ with image $\bar\lambda\in\cham$. Denote by
	$\Res(\coweyl(\lambda))$ the restriction of the $\check G$-representation
	$\coweyl(\lambda)=\Ind_{\check B^-}^{\check G}(\lambda)$ to a $\check G^I$-representation. The map
	$$\Res\left(\coweyl(\lambda)\right)\twoheadrightarrow\coweyl(\bar\lambda)$$
	of $\check G^I$-representations induced by restriction of functions $\mathcal{O}(\check
	G)\rightarrow\mathcal{O}(\check G^I)$, is a surjection.
\end{lem}
\begin{proof}
	Since the group $G$ splits after a finite field extension, the group $I$ acts on $\check
	G$ through a
	finite quotient. As $\check G$ is over the characteristic zero field $\Qell$, the map
	$$\mathcal{O}(\check G)\rightarrow\mathcal{O}(\check G)_I$$
	to the coinvariants is surjective. As this map gets identified with restriction
	of functions under the isomorphism $\mathcal{O}(\check G)_I\cong \mathcal{O}(\check
	G^I)$, the assertion follows.
\end{proof}
	Recall that $\coweyl(\lambda)$ is non-zero if and only if $\lambda$ is dominant with respect to
	$B$. In particular \cref{lemma:curious} implies that
	$\coweyl(\bar\lambda)$ is non-zero if and only if
	$\bar\lambda$ is in the image of the map $\splitchamb\rightarrow\chamb$.
\begin{rmk}\label{remark:curious}
	As the $\check G$-representations $\coweyl(\lambda)$
	for dominant $\lambda$ are precisely the irreducible ones, previous
	literature predominantly focused also on irreducible $\check G^I$-representations. This has
	previously caused a gap in the proof of the ramified Satake equivalence in
	\cite{zhuRamified} due to failure of the projection $\splitchamb\rightarrow\chamb$ to be
	surjective in general. The gap was later fixed in \cite{modularRamified} by reducing to the adjoint
	case, in which the projection between dominant coweights is surjective. In what follows,
	we have noticed that shifting the focus
	from irreducible representations to the induced representations
	$\coweyl(\bar\lambda)$ leads to a uniform treatment of split and non-split cases, which
	avoids the problem of the non-surjectivity using \cref{lemma:curious}.
\end{rmk}
In light of this, it turns out that simply replacing the tuple
$(\check G,\check B,\check T,\check U)$, with the tuple
$(\check G^I,\check B^I,\check T^I,\check U^I)$ leads to
a uniform treatment of the relavant constructions from \cite[Section
6.2]{centralRiche}. In fact, many proofs go through either verbatim, or can be reduced to the
connected case by passing to the identity component as we will observe below.
\subsection{The basic affine space}
The \textit{basic affine space} is the $\check G^I\x \check T^I$-variety
$\check G^I/\Udual$, equipped with the left action
$$(g,t)\cdot h\Udual=ght\inv\Udual.$$
Moreover, the isomorphism
$$(\check G^I\x\check T^I)/\check B^I\isomto\check G^I/\Udual$$
is compatible with the respective $\check G^I\x \check T^I$-actions;
where $\check B^I$ acts from the left by multiplication induced by the natural inclusion, times the natural
projection.

For $\bar\lambda\in\cham$, the $\bar\lambda$ weight space of the action
of $\{1\}\x \check T^I$
on $\mathcal{O}(\check G^I/\Udual)$ is given by
$\coweyl(-\bar\lambda):=\Ind_{\check B^I}^{\check G^I}(-\bar\lambda)$. This yields an isomorphism
\begin{equation}\label{eq:baseaffinespace}
\mathcal{O}(\check G^I/\Udual)\cong\bigoplus_{\bar\lambda\in\chamb} \coweyl(\bar\lambda)\otimes
-\bar\lambda
\end{equation}
of $\check G^I\times \check T^I$-modules.
We can explicitly identify the multiplication as well:
\begin{lem}
	The multiplication
	$$\mathcal{O}(\check G^I/\Udual)\otimes \mathcal{O}(\check G^I/\Udual)\rightarrow
	\mathcal{O}(\check G^I/\Udual)$$
	identifies via \eqref{eq:baseaffinespace}, with the map
	$$\bigoplus_{\bar\lambda,\bar\mu\in\cham}f_{\bar\lambda,\bar\mu}:
	\bigoplus_{\bar\lambda,\bar\mu\in\cham}\coweyl(\bar\lambda)\otimes
	\coweyl(\bar\mu)\rightarrow\bigoplus_{\cham}\coweyl(\bar\lambda+\bar\mu);$$
	induced naturally by Frobenius reciprocity.
\end{lem}
\begin{proof}
	This follows from noting that the multiplication is $\check T^I$-equivariant, and
	tracing the images of the weight spaces for each map respectively
	(cf. \cite[Lemma 6.2.1]{centralRiche}).
\end{proof}

Note that, the basic affine space is quasi-affine (\cite[Exercise 5.5.9(2)]{springerLinAlgGp}), thus
 embeds openly into $\mathcal{X}:=\Spec(\mathcal{O}(\check G^I/\Udual))$, its affine completion.
\subsection{Springer resolution and variants}\label{subsection:springer}
\begin{defn}
We call the space
	$$\Sprgi:=\check G^I\x^{\check B^I}\Lie\Udual,$$
	the \textit{Springer resolution} associated to the (possibly disconnected)
	group $\check G^I$.
\end{defn}
Note that the natural morphism $\Sprgd\rightarrow \Sprgi$ is an isomorphism; therefore
the Springer resolution for $\check G^I$ depends solely on its identity component.

However, the same is not true for the following $\check T^I$-torsor over $\Sprgi$:
$$\Sprgqafi:=\check G^I\x^{\Udual}\Lie\Udual.$$
The issue is again merely one of disconnectedness. The map $g\mapsto(g,0)$, induces an
isomorphism $\pi_0(\check G^I)\isomto\pi_0(\Sprgi)$; and the natural morphism
$\Sprgqafd\rightarrow \Sprgqafi$, is an isomorphism onto the neutral component.

Finally, we have the space $\Sprgafi\subset \check\fg^I\x\mathcal{X}$, which is defined to be the \textit{infinitesimal universal
stabilizer} for the $\check G^I$-action on $\mathcal{X}$. Briefly, it is the closed subscheme
determined by the kernel of an action of the Lie algebra $\check\fg^I$
on $\mathcal{O}(\mathcal{X})$ by derivations, obtained from differentiating the $\check G^I$-action; see e.g. \cite[p. 215]{centralRiche} for a detailed account.
Note that, it follows from the definition that the intersection of $\Sprgqafi$ with the basic
affine space, formed
inside $\check \fg^I\x\mathcal{X}$, is precisely $\Sprgafi$.
\subsection{Equivariant coherent sheaves on $\Sprgi$}
In this subsection, we will consider a tuple $(G,B,S)$, in which $G$ and $S$ are as
in \cref{iwahoriweylgroup}, but $B$ is \textit{opposite} to the Borel from
\cref{iwahoriweylgroup}.

Given $\mu\in\cham\cong X^*(\check T^I)$, there exists a natural geometric line bundle
$$\mathcal{L}_\mu\defined\check G^I\x^{\check B^I} \mathbb{A}^1\rightarrow \check G^I/\check B^I,$$
where $\check B^I$ acts on $\mathbb{A}^1$ through the composition of
\[\begin{tikzcd}[column sep=small]
	\check B^I\ar[r] & \check T^I \ar[r,"\mu"] &\Gm
\end{tikzcd}\]
with the natural $\Gm$-action. Let
$\mathcal{O}(\mu)$ denote the coherent sheaf on $\check G^I/\check B^I$ such that
$$\underline\Spec\left(\Sym^\otimes(\mathcal{O}(\mu)\inv)\right)\cong\mathcal{L}_\mu.$$

Consider now the embedding of $\Sprgi$ in $\check\fg^I\x \check G^I/\check B^I$ via $[g,x]\mapsto(g\cdot x,g\check B^I)$. This embedding gives rise to
two important classes of objects in the category $\Coh^{\check G^I}(\Sprgi)$ of
$\check G^I$-equivariant coherent sheaves on $\Sprgi$:
\begin{lem}\label{lemma:generatorssprg}
	The category $\Db\Coh^{\check G^I}(\Sprgi)$ is generated under cones and extensions by
	either:
	\begin{enumerate}
		\item for $\mu\in\cham$, the pullbacks $\mathcal{O}_{\Sprgi}(\mu)$
			of the line bundles $\mathcal{O}(\mu)$;
		\item the objects of the form $V\otimes\mathcal{O}_{\Sprgi}(\lambda)$, where
			$V\in\Rep(\check G^I)$ and $\lambda\in\chamb$.
	\end{enumerate}
\end{lem}
\begin{proof}
	After making the substitutions outlined in \cref{subsection:disconnectedness}, the proof
	of the analogous statement \cite[Lemma 6.2.8]{centralRiche} goes through verbatim.
\end{proof}
\subsection{Key Lemma}\label{subsection:keylemma} In the unramified case, the desired functor \eqref{eq:coherentfunctor}
is constructed from one valued in
$\Perv(\Hk{\I})$, via postcomposition with $\aviw$.
In contrast,
our construction of \eqref{eq:coherentfunctor} is not obtained through a factorization from
$\Perv(\Hk{I})$. In fact, existence of such a factorization is
equivalent to the existence of the central functor $\Rep(\check G^I)\rightarrow \Perv(\Hk{\I})$,
which is not known at the moment.

To overcome this difficulty, we prove the following key lemma, using the tilting property:
\begin{lem}\label{lemma:key}
	There exist a unique exact functor
	$$\Zent^I:\operatorname{Rep}{\check G^I}\rightarrow \PIW,$$
	such that the diagram
$$\begin{tikzcd}
	\Rep(\check G)\ar[d,"\res"']\ar[r,"\operatorname{Z}"] &
	\PervI \ar[d,"\aviw"]\\
	\Rep(\check G^I)\ar[r,"\Zent^I"]& \PIW
\end{tikzcd}$$
commutes.
\end{lem}
\begin{proof}
	By \cref{lemma:tiltingregular}, the projection $\Pi^0:\PIW\rightarrow\PIW^0$ is fully
	faithful when restricted to tilting objects. Therefore using \cref{cor:key}, it suffices to prove the
	statement after replacing $\PIW$ with the essential image of
	$\Pi^0\circ\Zent^\IW$.

	Note that the equivalence $\PIW^0\cong\Po$
	\eqref{eq:regulariwregular} identifies $\Pi^0\circ\Zent^\IW$ with the functor
	$\Zent^0:\Rep(\check G)\rightarrow\Po$. Under
	the equivalence $\Po\cong\Rep(H)$ of \cref{regular:tannakian},
	the functor $\Zent^0$ corresponds to the restriction $\Rep(\check G)\rightarrow
	\Rep(H)$.
	It then suffices to prove that $H\subset\check G$ is invariant under the $I$-action.

	Recall that the $I$-action on $\check G$ induces an $I$-action on $\Rep(\check
	G)$. Namely, an element $\gamma\in I$
	acts by the functor which leaves the underlying space of a representation
	fixed, but replaces the action of an element $g\in\check G$ by
	the action of $\gamma\cdot g$.
	Using Tannaka duality as in \cite[Lemma 4.5]{zhuRamified},
	invariance of $H$ under the $I$-action
	amounts to existence of an isomorphism
	$$\Zent^0(\gamma\cdot V)\cong \Zent^0(V),$$
        for every $\gamma\in I$ and $V\in\Rep(\check G)$, which is compatible with composition in $I$.
	The functor $\Zent$ is $I$-equivariant by \cite[Corollary 8.9]{modularRamified}\footnote{Although their parahoric group scheme
	is assumed to be special, the argument in fact works for an arbitrary parahoric.}, which
	means such isomorphisms
	$$\Zent(\gamma\cdot V)\cong\Zent(V)$$
	exist. These then
	give rise to desired ones for $\Zent^0\defined\Pi^0\circ\Zent$ via functoriality.
\end{proof}
\begin{rmk}\label{remark:H}
	Note that, along the way to the proof we have also observed that $H\subset\check G^I$
	holds.
\end{rmk}
\subsection{Deequivariantization algebra}
Start by considering the functor
\begin{gather}\label{eq:toruszent}
	F:\Rep(\check G\x \check T^I)\rightarrow \Perv(\Hk{\I}),\\
	V\boxtimes \lambda\mapsto\Zent(V)\star J_\lambda\nonumber
\end{gather}
where $V\boxtimes \lambda$ is the representation obtained by tensor product of a
$V\in\Rep(\check G)$ and $\lambda$ the one dimensional $\check T^I$-representation
corresponding to $\lambda\in\chamb\cong X^*(\check T^I)$.

The functor $F$ factors through the (non-full!) subcategory $\cC$, whose objects are those in the essential
image of $F$, and morphisms are those which commute with the images under $F$ of the symmetry
constraints of $\Rep(\check G\x \check T^I)$. As in \cite[Lemma 6.3.3]{centralRiche},
convolution induces a natural symmetric monoidal structure on $\cC$.

Consider the
$\Qell$-algebra
$$A:=\Hom_{\Ind(\cC)}\left(1_\cC,\cO(\check G\x\check T^I)\right)$$ whose multiplication is induced by that of
$\cO(\check G\x\check T^I)$. The algebra $A$ is naturally equipped with a
$\check G\x\check T^I$-action and arguments identical to
\cite[Proposition 6.3.5]{centralRiche} provide an equivalence
\begin{equation}\label{eq:nonfullsub}
\cC\cong A-\mathsf{mod}^{\check G\x\check T^I}_{\mathsf{fr}},
\end{equation}
where right hand side is the category of free $\check G\x\check
T^I$-equivariant $A$-modules. More precisely, those of the form $M\otimes_{\Qell} A$ where
$M$ is a $\check G\x\check T^I$-representation and the tensor product is equipped with the
diagonal action.
\subsection{Tannakian construction of certain algebra maps}\label{subsection:tannakian}
Each datum in the following list can be reconstructed from another on the list, in particular
they are equivalent: \begin{enumerate}
	\item a $\check G\x\check T^I$-equivariant algebra morphism
		$\mathcal{O}(\check\fg)\rightarrow A$;
	\item a $\check G\x\check T^I$-equivariant element in $\check\fg\otimes A$;
	\item for any $V\in\Rep(\check G\x\check T^I)$, a $\check G$-equivariant endomorphism $f_V$ of
		$V\otimes A$, satisfying
		$$f_{V_1\otimes V_2}=f_{V_1}\otimes_A \id_{V_2\otimes A}+\id_{V_1\otimes
		A}\otimes_A f_{V_2}$$
		for $V_1,V_2\in\Rep(\check G\x\check T^I)$.
\end{enumerate}
The asserted equivalence between each datum above follows from a general Tannakian set-up, as explained
in \cite[Example 6.3.1]{centralRiche}.
In particular, \eqref{eq:monodromymonoidal} implies that the collection of nilpotent endomorphisms $\bn_V$
of $\Zent(V)$ form a datum of type (3). From this, we can then extract
a datum of type (1), and in particular, an algebra morphism
$$\mathcal{O}(\check\fg)\rightarrow A.$$
\begin{lem}\label{lem:factorlie}
	The map $\mathcal{O}(\check\fg)\rightarrow A$ factors through
	a $\check G^I\x\check
	T^I$-equivariant algebra morphism $\mathcal{O}(\check\fg^I)\rightarrow A.$
\end{lem}
\begin{proof}
	By the discussion above, the map $\mathcal{O}(\check\fg)\rightarrow A$ corresponds to an element $a\in \check\fg\otimes
	A$.
        Casting the question in Tannakian terms, the desired factorization is equivalent to the invariance of the
	element $a$, under the induced $I$-action (see \cite[p.361]{yunIntegralHomology}). This in turn
	is equivalent to commutativity of the diagrams
	\[\begin{tikzcd}
		\Zent(\gamma\cdot V)\ar[d,"\bn_{\gamma\cdot V}"]\ar[r,"\sim"]& \Zent(V)\ar[d,"\bn_V"]\\
		\Zent(\gamma\cdot V)\ar[r,"\sim"]& \Zent(V)
	\end{tikzcd}\]
	where $V\in\Rep(\check G)$, $\gamma\in I$ and the horizontal isomorphims are those from
	the proof of \cref{lemma:key}.
	The commutativiy of these diagrams
	then follows from the constructions, as $\bn_V$ are constructed precisely as
	the logarithm of the $I$-action, which
	gave rise to the horizontal isomorphisms in the first place.
\end{proof}
\begin{rmk}\label{remark:factorlie}
	Note that, the argument in the proof of \cref{lem:factorlie} also shows
	the element $n_0\in\check\fg$ appearing in Tannakian description of the regular
	quotient, \cref{regular:tannakian}, is $I$-invariant.
\end{rmk}
The collection of highest weight arrows \cref{subsection:highestweight} live in the subcategory
$\cC$ by \cref{prop:highestweightmonoidal}, and can be used to construct a $\check G\x\check T^I$-equivariant
algebra map $\mathcal{O}(\check G/\check U)\rightarrow A$, as in \cite[Subsection
6.3.5]{centralRiche}. Briefly, by the explicit description
$$\mathcal{O}(\check G/\check
U)\cong\bigoplus_{\lambda\in \splitchamb}\coweyl(\lambda)\otimes-\bar\lambda;$$ it suffices to construct for each
$\lambda\in\splitchamb$, a $\check G\x\check T^I$-equivariant map
\begin{align}\label{eq:flambda}
	f_\lambda:\coweyl(\lambda)\otimes-\bar\lambda\rightarrow A.
\end{align}
The highest weight arrows $\mathfrak{f}_\lambda$ give rise to
morphims of $\check G\x\check T^I$-modules
\begin{equation}\label{eq:coherenthighestweight}
\coweyl(\lambda)\otimes A\rightarrow \bar\lambda \otimes A,
\end{equation}
using the description of $\cC$ in terms of the algebra $A$.
Finally, restricting \eqref{eq:coherenthighestweight} to $\coweyl(\lambda)$ and adding a
$-\bar\lambda$ twist, provides the desired maps.

Fixing the unique splitting of the surjection $\coweyl(\lambda)\rightarrow
\coweyl(\bar\lambda)$ from \cref{lemma:curious}, which induces the identity endomorphism on the
$\bar\lambda$ weight space; we can restrict
the collection $f_\lambda$ of \eqref{eq:flambda}
to obtain for each $\lambda\in\splitchamb$, a $\check G^I\x\check T^I$-equivariant morphism
$$f_{\bar\lambda}:\coweyl(\bar\lambda)\otimes-\bar\lambda\rightarrow A.$$
Correspondingly (\eqref{eq:baseaffinespace} and \cref{lemma:curious}), we get a
$\check G^I\x\check T^I$-equivariant algebra map
$$\mathcal{O}(\check G^I/\Udual)\rightarrow A.$$

Putting everything we constructed together, we finally arrive at a
$\check G^I\x\check T^I$-equivariant algebra morphism
\begin{align}\label{eq:algmap}
\mathcal{O}(\check{\fg}^I\times\mathcal{X})\rightarrow A.
\end{align}
\begin{prop}\label{prop:Sprgafi}
	The $\check G^I\x\check T^I$-equivariant map \eqref{eq:algmap} factors through
        a $\check G^I\x\check T^I$-equivariant map $\mathcal{O}(\Sprgafi)\rightarrow A.$
\end{prop}
\begin{proof}
	Recall that the equivalence \eqref{eq:nonfullsub} provides a morphism
	$$n_{\coweyl(\lambda)}\in\End(\coweyl(\lambda)\otimes A),$$
	corresponding to the logarithm of the monodromy $\mathbf{n}_{\coweyl(\lambda)}\in
	\End\left(F(\coweyl(\lambda)\boxtimes 1)\right)$. By \cref{lem:factorlie},
	$n_{\coweyl(\lambda)}$
	preserves the subspace $\coweyl(\bar\lambda)\otimes A$ in
	$\Res(\coweyl(\lambda))\otimes A$ under the chosen embedding.
	Denote by $n_{\coweyl(\bar\lambda)}$, restriction of
	$n_{\coweyl(\lambda)}$ to $\coweyl(\bar\lambda)\otimes A$.

	As in \cite[Lemma 6.3.7]{centralRiche}, if  for every $\bar\lambda\in\chamb$
	$$f_{\bar\lambda}\circ n_{\coweyl(\bar\lambda)}=0$$
	holds, then the generators of the ideal defining $\mathcal{O}(\Sprgafi)$ as quotient
	of $\mathcal{O}(\check\fg^I\x\mathcal{X})$ vanish under
	\eqref{eq:algmap} by construction.
	Since both $f_{\bar\lambda}$ and $n_{\coweyl(\bar\lambda)}$ arise as restrictions, this follows from the compatibility
	$\mathfrak{f}_\lambda\circ\mathbf{n}_{\coweyl(\lambda)}=0$
	of \cref{lemma:highestweightmonodtriv}.
\end{proof}

\subsection{Descending to $\check G^I$ via averaging}\label{subsection:descending}
We can then construct functor
\begin{gather}
	\widetilde{F}:\Coh^{\check G^I\x\check T^I}_{\mathsf{fr}}(\Sprgafi)\rightarrow
	A-\mathsf{mod}^{\check G^I\x\check T^I}_{\mathsf{fr}},\\
	V\otimes \mathcal{O}(\Sprgafi)\mapsto V\otimes A;\nonumber
\end{gather}
where $\Coh^{\check G^I\x\check T^I}_{\mathsf{fr}}(\Sprgafi)$ denotes the category of
equivariant coherent sheaves on $\Sprgafi$ of the form $V\otimes \mathcal{O}(\Sprgafi)$ for a
$\check G^I\x\check T^I$ representation $V$; by forming tensor product over
${\mathcal{O}(\Sprgafi)}$ along the map from \cref{prop:Sprgafi}.

\begin{lem}\label{lem:factoralg}
	There exist a unique exact functor
\begin{align*}
A-\mathsf{mod}^{\check G^I\x\check T^I}_{\mathsf{fr}}\rightarrow \PIW,
\end{align*}
such that the following diagram commutes:
$$
\begin{tikzcd}
	A-\mathsf{mod}^{\check G\x\check T^I}_{\mathsf{fr}} \ar[r,"i"] \ar[d,"\Res"'] &
	\PervI \ar[d,"\aviw"]\\
	A-\mathsf{mod}^{\check G^I\x\check T^I}_{\mathsf{fr}}\ar[r]& \PIW.
\end{tikzcd}$$
\end{lem}
\begin{proof}
	As the inclusion of $A-\mathsf{mod}^{\check G\x\check T^I}_{\mathsf{fr}}$ to
	$\Perv(\Hk{\I})$ factors through the essential image of
	$$F:\Rep(\check G\x \check T^I)\rightarrow \Perv(\Hk{\I}),$$
	the idempotent endomorphisms corresponding to
	$\check G^I\x\check T^I$-stable subspaces of the underlying $A$-module of
	$V\otimes A$, can be canonically lifted to idempotent endomorphisms in
	$\End(\aviw(i(V\otimes A)))$ using \cref{lemma:key}. In turn, this assigment can be
	uniquely upgraded to the desired functor.
\end{proof}
Finally, precomposition of the functor from \cref{lem:factoralg} with
$\widetilde{F}$ yields the functor:
\begin{gather}
\widetilde{F}^I:\Coh^{\check G^I\x\check T^I}_{\mathsf{fr}}(\Sprgafi)\rightarrow \PIW\\
	V\otimes\mathcal{O}(\Sprgafi)\mapsto V\otimes A.\nonumber
\end{gather}
\subsection{Complexes on the Springer resolution as a quotient}
Denote by $\mathcal{K}$, the kernel of the functor
\begin{equation}\label{eq:functor}
\Kb\left(\Coh^{\check G^I\x\check T^I}_{\mathsf{fr}}(\Sprgafi)\right)\rightarrow
\Db\Coh^{\check G^I\x\check T^I}(\Sprgqafi)
\end{equation}
induced by restriction along the inclusion $\Sprgqafi\subset\Sprgafi$. Note also the equivalence
\begin{equation}\label{eq:qafitospgr}
\Db\Coh^{\check G^I\x\check T^I}(\Sprgqafi)\isomto\Db\Coh^{\check
G^I}(\Sprgi),
\end{equation}
induced by pullback along the isomorphism
$$\Sprgqafi/\check T^I\isomto\Sprgi.$$
Putting these together, we arrive at:
\begin{lem}\label{lem:verdier}
	The functor \eqref{eq:functor} induces an equivalence
	$$\Kb\left(\Coh^{\check G^I\x\check T^I}_{\mathsf{fr}}(\Sprgafi)\right)/\mathcal{K}\isomto \Db\Coh^{\check
G^I}(\Sprgi),$$
	where the left hand side denotes the Verdier quotient with respect to the localizing subcategory
	$\mathcal{K}$; after postcomposition with \eqref{eq:qafitospgr}.
\end{lem}
\begin{proof}
	Note that, the map $\pi_0(\Sprgqafi)\rightarrow \pi_0(\Sprgafi)$ induced by
	the natural inclusion corresponds to the identity of $\pi_0(\check G^I)$ under the respective identifications.
	In light of this, by reduction to neutral components
	it suffices to check the assertion for
	the split reductive group $\Gdual$; which is \cite[Proposition 6.2.10]{centralRiche}.
	\end{proof}
Consider the functor
$$\Grad\widetilde{F}^I:\Coh^{\check G^I\x\check
T^I}_{\mathsf{fr}}(\Sprgafi)\rightarrow\Rep(\check T^I);$$
obtained by refining the composition
$$\begin{tikzcd}[column sep=large]
	A-\mathsf{mod}^{\check G\x\check T^I}_{\mathsf{fr}}\isomto
	\cC \ar[r,"\left.\Grad\right|_\cC"]& \Rep(\check T^I),
\end{tikzcd}
$$
as in \cref{subsection:descending}. Let $$e^*:\Coh^{\check G^I\x\check T^I}_{\mathsf{fr}}(\Sprgafi)\rightarrow\Rep(\check
T^I)$$
be the functor induced by restriction along the base point $[1,0]\in\Sprgqafi.$
\begin{lem}\label{lem:gradrestrict}
	The functors $\Grad\widetilde{F}^I$ and $e^*$ are isomorphic.
\end{lem}
\begin{proof}
	As in \cite[Lemma 6.3.8]{centralRiche}, it suffices to observe that the
	$\check T^I$-equivariant morphism $\mathcal{O}(\Sprgafi)\rightarrow \Qell$ corresponding
	to $\Grad\widetilde{F}^I$, is the evaluation at the base point. Recall that this
	morphism is ultimately induced via the endomorphisms of the fiber functor of
	$\Rep(\check T^I)$, provided by the
	monodromy and the highest weight arrows. Since the monodromy acts trivially
	on the Wakimoto filtration by \cref{lemma:wakimotogradtriv},  and the highest weight arrows correspond
	to the projection to the highest weight line, the lemma follows.
\end{proof}
\subsection{Factorization through $\Sprgi$}
\begin{prop}\label{prop:construction}
	There exist a unique functor
	\begin{equation*}\label{eq:fiw}
	F_\IW:\Db\Coh^{\check G^I}(\Sprgi)\rightarrow \DIW,
	\end{equation*}
	making the following diagram commute:
$$\begin{tikzcd}
	\Kb\left(\Coh^{\check G^I\x\check
	T^I}_{\mathsf{fr}}(\Sprgafi)\right)\ar[d]\ar[r,"\Kb(\widetilde{F}^I)"]&\Kb(\PIW)\ar[d]\\
	\Db\Coh^{\check G^I}(\Sprgi)\ar[r,"F_\IW"]&\DIW.
\end{tikzcd}
$$
\end{prop}
\begin{proof}
	By the universal property of Verdier quotients, using \cref{lem:verdier}, it suffices to
	check that the kernel $\mathcal{K}$ of \eqref{eq:functor} vanishes under
	$\Kb(\widetilde{F}^I)$.
	Via a straightforward induction on the Wakimoto filtrations, analogous to \cite[Proposition 4.6.1]{centralRiche},
	acyclicity of a complex $\Kb(\Grad\widetilde{F}^I)(K)\in\Kb(\Rep(\check T^I))$
	implies the acyclicity of $\Kb(\widetilde{F}^I)(K)$.
	Finally, since $\mathcal{K}$ is the kernel of the functor induced by restriction to
	$\Sprgqafi$, and \cref{lem:gradrestrict} implies $\Grad\widetilde{F}^I$ corresponds to the
	restriction to the base point which also lies in $\Sprgqafi$; $\Kb(\Grad\widetilde{F}^I)(K)$ is acyclic for every
	$K\in\mathcal{K}$, concluding the proof.
\end{proof}

\section{Proof of the equivalence}\label{section:proof}
Now that we have constructed our functor $F_\IW$ \eqref{eq:fiw}, we can formulate our main
theorem:
\begin{thrm}\label{theorem:main}
	The functor
	$$F_\IW:\Db\Coh^{\check G^I}(\Sprgi)\rightarrow \DIW,$$
	is an equivalence.
\end{thrm}
In contrast to \cref{section:coherentfunc}, the proof follows almost verbatim the split case.
\begin{lem}\label{lemma:generators1}
	The category $\DIW$ is generated by the objects $\aviw(J_{\bar\lambda})$ with
	$\bar\lambda\in\cham$, under cones and extensions.
\end{lem}
\begin{proof}
	Follows from standard arguments, using the agreement of classes
	$\left[\aviw(J_{\bar\lambda})\right]$ in the Grothendieck group
	with those of standard and costandard objects (see e.g. proof of
	\cref{whittaker:eulerchar}). Such arguments in the split setting are
	are outlined in \cite[Lemma 6.6.3]{centralRiche}.
\end{proof}
\begin{lem}\label{lem:relationspreserved}
	For every $V\in\Rep(\check G^I)$, the map
	\begin{equation}\label{eq:hominj}
		\Hom\left(\mathcal{O}(\Sprgi),V\otimes\mathcal{O}(\Sprgi)\right)\rightarrow
		\Hom\left(\ICIW_0,\Zent^I(V)\right)
	\end{equation}
	induced by $F_\IW$ is injective.
\end{lem}
\begin{proof}
	As both objects on the right hand side are tilting, we can, via
	\cref{lemma:tiltingregular}, check
	injectivity in the regular quotient $\PIW^0$. Recall that
	$\PIW^0\cong\Rep(H)$ for some subgroup $H\subset Z_{\check G^I}(n^0)$
	(\cref{regular:tannakian}, \cref{remark:H}, \cref{remark:factorlie}, and \eqref{eq:regulariwregular}).

	On the other hand, \cref{ker:dim} implies
	that the element $n_0\in\check\fg^I$ is regular. In particular, $n_0$
	admits a unique preimage $\tilde n_0$ under the Springer map
	\begin{gather*}
	\Sprgi\rightarrow \check\fg^I,\:
		[g,x]\mapsto g\cdot x.
	\end{gather*}
	This map
	restricts to an isomorphism over the regular nilpotent orbit in $\check\fg^I$, which
	also identifies with $\check G^I/Z_{\check G^I}(n^0)$, see \cite[Section 6.10]{jantzenNilpotent}.
	In particular, the $\check G^I$-orbit $\widetilde{\mathcal{O}}_r$ of $\tilde n_0$ satisfies
	$$\Rep(Z_{\check G^I}(n^0))\cong \Coh^{\check G^I}(\widetilde{\mathcal{O}}_r).$$

	Under these
	identifications, the map \eqref{eq:hominj} corresponds to the inclusion
	$$V^{Z_{\check G}(n^0)}\rightarrow
	V^H$$ between the respective fixed vectors (see e.g. \cite[Section
	6.5.11]{centralRiche}), which is clearly injective.
\end{proof}
Given $V\in\Rep(\check G^I)$, we will denote by $\Zent^I(V)\star J_{\bar\lambda}$ the direct
summand of $\aviw(\Zent(V')\star J_{\bar\lambda})$, for a $V'\in\Rep(\check G)$ as in
\cref{cor:key}.
\begin{cor}\label{cor:highervanish}
	For every $V\in\Rep(\check G^I)$ and $\bar{\lambda}\in\cham$, the map
	\begin{equation}\label{eq:almosteqinj}
		\Ext^n\left(\mathcal{O}(\Sprgi),V\otimes\mathcal{O}(\Sprgi)(\bar\lambda)\right)\rightarrow
		\Ext^n\left(\ICIW_0,\Zent^I(V)\star J_{\bar\lambda}\right)
	\end{equation}
	induced by $F_\IW$ is injective.
\end{cor}
\begin{proof}
	By \cite[Theorem 2.4]{broCohomology}, the cohomology groups
	$\sH^n(\Sprgi,\mathcal{O}(\Sprgi)(\bar\lambda))$ vanish for $n>0$. Consequently, it
	suffices to check the assertion for $n=0$.

	We can find $V'\in\Rep(\check G^I)$, such that there is a $\check G^I$-equivariant
	embedding $\mathcal{O}(\Sprgi)(\bar\lambda)\rightarrow V'\otimes\mathcal{O}(\Sprgi)$ (e.g.
	using the method of \cite[Lemma 6.2.9]{centralRiche}). This reduces the assertion to
	\cref{lem:relationspreserved} (cf. \cite[Corollary 6.6.5]{centralRiche}).
\end{proof}
Finally, we prove the ramified Arkhipov-Bezrukavnikov equivalence:
\begin{proof}[Proof of \cref{theorem:main}]
	By \cref{lemma:generators1}, the generators of $\DIW$ are in the essential image of
	$F_\IW$, reducing the claim to its fully faithfulness. To prove said fully faithfulness, by using 5-lemma and \cref{lemma:generatorssprg}, it suffices to check that
	the injective maps \eqref{eq:almosteqinj} are also surjective. As both are finite
	dimensional, we compare dimensions.

	To compute the dimension of the right hand side, we may first convolve with the
	invertible object $J_{-\bar\lambda}$, \cref{std:conv}. Then, we see
	by the tilting property of $\Zent^I(V)$ (\cref{cor:key}) that the right hand side
	vanishes for $n>0$;
	showing the desired surjectivity. On the other hand, for $n=0$
	the right hand side is equidimensional with the $\bar\lambda$ weight
	space of $V$, which coincides with the left hand side by an explicit calculation, see
	\cite[Section 6.6.3]{centralRiche}.
\end{proof}

\printbibliography
\end{document}